\documentclass[a4paper,11pt,twoside]{article}

\author{
	William Salkeld \\[8pt]
	\small School of Mathematical Sciences \\
	\small University of Nottingham, \\
	\small University Park \\
	\small Nottingham, NG7 2RD \\
	\small  william.salkeld@nottingham.ac.uk
}

\usepackage[nodayofweek]{datetime}

\usepackage[utf8]{inputenc}
\usepackage[T1]{fontenc}
\usepackage[english]{babel}

\usepackage{charter}

\usepackage{xcolor}
\usepackage{verbatim}

\usepackage{hyperref}

\usepackage[top=3.cm, bottom=4.0cm, left=2.2cm, right=2.2cm]{geometry}
\usepackage{amsmath}
\usepackage{amsthm}	
\usepackage{amsfonts}	
\usepackage{amssymb,bbm}
\usepackage{shuffle}
\usepackage{mathrsfs}
\usepackage{pifont}

\usepackage{enumerate}
\usepackage{enumitem}
\usepackage[nobysame, alphabetic,initials]{amsrefs}

\usepackage{stmaryrd}
\usepackage{appendix}
\usepackage{etoolbox}		

\usepackage{graphicx}

\usepackage{
						ulem		
} 

\numberwithin{equation}{section}

\theoremstyle{plain}
\newtheorem{theorem}{Theorem}[section]
\newtheorem{lemma}[theorem]{Lemma}
\newtheorem{proposition}[theorem]{Proposition}
\newtheorem{corollary}[theorem]{Corollary}

\newtheorem{definition}[theorem]{Definition}

\newtheorem{remark}[theorem]{Remark}
\newtheorem{example}[theorem]{Example}

\allowdisplaybreaks

\newcommand{\bE}{\mathbb{E}}

\newcommand{\bN}{\mathbb{N}}
\newcommand{\bP}{\mathbb{P}}

\newcommand{\bR}{\mathbb{R}}

\newcommand{\bW}{\mathbb{W}}

\newcommand{\cB}{\mathcal{B}}

\newcommand{\cF}{\mathcal{F}}

\newcommand{\cI}{\mathcal{I}}
\newcommand{\cJ}{\mathcal{J}}

\newcommand{\cL}{\mathcal{L}}

\newcommand{\cP}{\mathcal{P}}

\newcommand{\fD}{\mathfrak{D}}

\newcommand{\scG}{\mathscr{G}}

\newcommand{\scN}{\mathscr{N}}
\newcommand{\scP}{\mathscr{P}}

\newcommand{\rD}{\mathbf{D}}

\newcommand{\vertiii}{{\vert\kern-0.25ex \vert\kern-0.25ex \vert}}

\DeclareMathOperator{\supp}{supp}
\DeclareMathOperator{\Shuf}{Shuf}
\DeclareMathOperator{\lip}{Lip}

\DeclareMathOperator{\lin}{Lin}

\newcommand{\A}[1]{A_{#1}[0]}

\definecolor{darkgreen}{rgb}{0,0.35,0}

\usepackage{forest}
\forestset{
	decor/.style = {
		label/.expanded = {[inner sep = 0.1ex, font=\unexpanded{\tiny}]right:{$#1$}}  
	},
	root/.style = {minimum size = 0.5ex},   
	decorated/.style = {
		for tree = {
			fill, 
			circle,
			inner sep = 0ex, 
			minimum size = 0.5ex,   
			grow' = north, 
			l = 0, 
			l sep = 0.8ex, 
			s sep = 1.3em,
			fit = tight, 
			parent anchor = center, 
			child anchor = center,
			delay = {decor/.option = content, content =}
		} 
	},
	x/.style = {
		circle, draw, fill=white
	},
	default preamble = {decorated, root},
	begin draw/.code={\begin{tikzpicture}[baseline={([yshift=-0.5ex]current bounding box.center)}]},
	}
	
	\usepackage{pgf,tikz}
	\usetikzlibrary{arrows}
	
	\tikzstyle{vertex} = [fill, shape=circle,inner sep=2pt,]
	\tikzstyle{edge} = [fill, line width = 0.5pt]
	\tikzstyle{zhyedge1} = [opacity=.5,fill opacity=.5, line cap=round, line join=round, line width=27pt,color=black]
	\tikzstyle{zhyedge2} = [opacity=.5,fill opacity=.5, line cap=round, line join=round, line width=25pt,color=white]
	\tikzstyle{hyedge1} = [opacity=.5,fill opacity=.5, line cap=round, line join=round, line width=27pt]
	\tikzstyle{hyedge2} = [opacity=.5,fill opacity=.5, line cap=round, line join=round, line width=25pt, color=white]
	
	\tikzstyle{vertexS} = [fill, shape=circle,inner sep=1pt,]
	\tikzstyle{edgeS} = [fill, line width = 0.25pt]
	\tikzstyle{zhyedge1S} = [opacity=.5,fill opacity=.5, line cap=round, line join=round, line width=12pt,color=black]
	\tikzstyle{zhyedge2S} = [opacity=.5,fill opacity=.5, line cap=round, line join=round, line width=10pt,color=white]
	\tikzstyle{hyedge1S} = [opacity=.5,fill opacity=.5, line cap=round, line join=round, line width=12pt]
	\tikzstyle{hyedge2S} = [opacity=.5,fill opacity=.5, line cap=round, line join=round, line width=10pt, color=white]
	
	\pgfdeclarelayer{background}
	\pgfsetlayers{background,main}

\title{Higher order Lions-Taylor expansions}

\begin{document}
	
	\maketitle
	
	\begin{abstract} 
		In this paper, we provide some of the necessary mathematics to describe higher order Lions-Taylor expansions. The Lions derivative of a functional on the Wasserstein space of measures quantifies infinitesimal perturbations on measures in terms of infinite variation on a linear space of random variables. 
		
		The two contributions of this paper are establishing the link between partitions of ordered sets and the terms of a Lions-Taylor expansion, and explicit Lions-Taylor expansions and remainder terms for functionals of a spatial and measure variable. 
	\end{abstract} 
	
	
	{\bf Keywords:} Lions-Taylor expansions
	
	\vspace{0.3cm}
	
	\noindent
	{\bf 2020 AMS subject classifications:}\\
	Primary: 60Hxx 
	\quad
	Secondary: 60L30, 46T20
	
	
	\noindent
	{\bf Acknowledgements}: William Salkeld wishes to thank the London Mathematical Society for the award of an \emph{Early Career Fellowship} (ECF-1920-29) which facilitated this research.
	
	Further, William Salkeld was supported by MATH+ project AA4-2 and by the US Office of Naval Research under the Vannevar Bush Faculty Fellowship N0014-21-1-2887.  
	
	\setcounter{tocdepth}{2}
	\tableofcontents
	\newpage
	
	\section{Introduction}
	
	This paper finds its origin in the theory of probabilistic rough paths. First introduced in \cite{2019arXiv180205882.2B}, a probabilistic rough path intertwines the analytic and probabilistic properties of a driving signal to parsimoniously capture both individual and collective features, and their actions on dynamic systems. At their heart, a probabilistic rough path is an `\emph{abstract Taylor expansion}' (also referred to as a regularity structure \cite{hairer2014theory}) on the Wasserstein space of measures. The concept of probabilistic rough paths was generalised beyond the \emph{level 2} case where terms of the probabilistic rough path describe the increments of the path and its iterataed integrals in \cite{2021Probabilistic},  \cite{salkeld2021Probabilistic2} and the subsequently accepted \cite{salkeld2022ExamplePRP}. To achieve this, higher order Lions-Taylor expansions were required along with explicit upper bounds for their remainder terms. 
	
	This paper serves as a more accessible and complete introduction to the Lions calculus. However, it will not address any details of regularity structures or probabilistic rough paths but instead should be treated as a reference that we hope other mathematicians will find useful and informative. 
	
	\subsection{Motivation}
	
	Motivated by the McKean-Vlasov differential equation 
	\begin{equation}
		\label{eq:meanfield:equation}
		dX_{t} = f\Big(X_{t}, \cL^{X}_{t} \Big) dW_{t} \quad \cL_t^X = \bP \circ \big( X_t \big)^{-1}
	\end{equation}
	and the associated particle system
	\begin{equation}
		\label{eq:particle:system}
		dX_{t}^{i, N} = f\Big( X_{t}^{i, N}, \sum_{j=1}^N \delta_{X_t^{j, N}} \Big) dW_{t}^{i, N} \quad i \in \{1, ..., N \}
	\end{equation}
	we want to consider some Taylor expansion for a function
	$$
	f: \bR^e \times \cP_2(\bR^e) \to \lin(\bR^d, \bR^e)
	$$
	where $e$ is the dimension of the solution process and $d$ is the dimension of the driving signal. However, to streamline the notation somewhat for the reader, in this section we will simply consider 
	$$
	f: \bR^e \times \cP_2(\bR^e) \to \bR^d. 
	$$
	We emphasise that this does not change the mathematics beyond the dimension of the associated vector spaces. 
	
	Taylor's Theorem is a well-known result which states that for a function $f$ that is $n$ times differentiable, we have
	$$
	f(y) - f(x) = \sum_{k=1}^n \frac{\nabla^k f(x)}{k!} \Big[ (y-x)^{\otimes k} \Big] + R_n^{x, y}\big( f \big)
	$$ where the remainder term
	\begin{align*}
		R_n^{x, y}\big( f \big)=& \tfrac{1}{(n-1)!} \int_0^1 \Big( \nabla^n f\big( x + \xi(y-x) \big) - \nabla^n f\big( x \big) \Big) (1-\xi)^{n-1} d\xi \cdot (y-x)^{\otimes n}  
		\\
		=& O\Big( |y-x|^{n+1} \Big).
	\end{align*}
	Our objective is to provide a similar Theorem for functionals depending on a measure argument. 
	
	Throughout this paper, the derivatives we consider are thus constructed on the space $\cP_{q}(\bR^e)$ for $q=2$, the so-called `Wasserstein space' of probability measures with a finite second moment. For any $q \geq 1$, $\cP_q(\bR^e)$ can be equipped with the $\bW^{(q)}$-Wasserstein distance, defined by:
	\begin{equation}
		\label{eq:WassersteinDistance}
		\bW^{(q)}(\mu,\nu) = \inf_{\Pi \in \cP_{q}(\bR^e \oplus \bR^e)}
		\biggl( \int_{\bR^e \oplus \bR^e}
		\big| y-x \big|^q
		d \Pi(x,y) \biggr)^{1/q},
	\end{equation}
	the infimum being taken with respect to all the probability measures $\Pi$ on the product space $\bR^e \oplus \bR^e$ with $\mu$ and $\nu$ as respective $e$-dimensional marginal laws.
	
	\subsection{Previous work}
	
	Higher order Taylor expansions are central to approximation techniques throughout the mathematical sciences. Therefore, it is perfectly natural to desire a differential calculus on the space of measures when attempting to approximate the dynamics of large populations and their associated mean-field limits. The origins of this philosophy can be found in \cite{Jordan1998variation} where the connection between the Fokker-Planck equations and gradient flows on the Wasserstein space is first established. We refer the reader to the monograph \cite{Ambrosio2008Gradient} for a complete overview of the subject. 
	
	The classical difference of increments approach is used in \cite{buckdahn2017mean} to derive a PDE over the Wasserstein space of measures under the strong assumption that second order Fr\'echet derivatives exist. In \cite{chassagneux2014classical}, similar results are derived under weaker assumptions using a projection over empirical measures. Both of these results, along with many others, have been neatly reviewed in \cite{CarmonaDelarue2017book1}*{Chapter 5}. For this work, we address the case where the measure is an empirical measure but we do not use this to improve the assumptions on our work and leave this as an open problem for another day. More recently, \cite{dos2022Ito} presents a It\^o-Wentzell-Lions formula for measure dependent random fields and \cite{Ren2019Bismut} uses Lions calculus to solve a Bismut formula for the change in weak error of an SDE with respect to a change in the distribution of the initial condition. 
	
	To the best of our knowledge, this Taylor expansion for the Lions derivative is new (at least in this form, see \cite{TseHigher2021} for another formulation), and we strongly believe it to have its own interest beside the specific application that we address here. In this regard, a striking fact in our analysis is the form of the expansion itself: it is encoded by means of partition sequences that are used to encode grafting operations of Lions trees. 
	
	\subsection{Contributions}
	The core contribution of this work is Theorem \ref{theorem:LionsTaylor2} which provides a concise formulation for the Lions-Taylor expansion of a smooth function of two variables, a spatial and measure variable. Note that the choice of $\alpha, \beta$ and $\gamma$ allows for one to take different numbers of derivatives in each variable and this leads to the highly involved nature of the remainder term. In particular, we provide explicit remainder terms (see Equation \eqref{eq:theorem:LionsTaylor2_remainder1}) and asymptotic upper bounds. 
	
	Further, we provide a link between iterative Lions derivatives and partitions of ordered sets in Lemma \ref{Lemma:Bijection-Partition-simple}. While not explored in this work, this observation is central to \cite{salkeld2022LionsTrees} and the construction of coupled bialgebras. 
	
	\subsection*{Notations}
	\label{subsection:Notation}
	
	Let $\bN$ be the set of positive integers and $\bN_{0}=\bN \cup \{0\}$. Let $\bR$ be the field of real numbers and for $d \in \bN$, let $\bR^d$ be the $d$-dimensional vector space over the field $\bR$. For vector spaces $U$ and $V$, we define $\lin(U, V)$ to be the collection of linear operators from $U$ to $V$. Let $\oplus$ and $\otimes$ be the direct sum and tensor product operations. 
	
	For a vector space $U$, let $\cB(U)$ be the Borel $\sigma$-algebra. Let $(\Omega, \cF, \bP)$ be a probability space. For $q \in [1, \infty)$, let $L^q(\Omega, \cF, \bP; U)$ be the space of $q$-integrable random variables taking values in $U$. When the $\sigma$-algebra is not ambiguous, we will simply write $L^q(\Omega, \bP; U)$. Further, let $L^0(\Omega, \bP; U)$ be the space of measurable mappings $(\Omega, \cF) \mapsto (U, \cB(U))$. 
	
	For a set $\scN$, we call $2^\scN$ the collection of subsets of $N$ and $\scP(\scN)$ the set of all partitions of the set $\scN$. This means $\scP(\scN) \subseteq 2^{2^\scN}$. A partition $P \in \scP(\scN)$
	if and only if the following three properties are satisfied:
	$$
	\forall x \in \scN, \quad \exists p \in P: x \in p; 
	\qquad
	\forall p, q \in P, \quad p \cap q = \emptyset; 
	\qquad
	\emptyset \notin P.
	$$

	\section{Taylor expansions over the Wasserstein space}
	\label{section:TaylorExpansions}
	
	We build a series of differential operators on the $2$-Wasserstein space. For a function $f:\cP_2(\bR^e) \to \bR^d$, we consider the canonical lift $F: L^2(\Omega, \cF, \bP; \bR^e) \to \bR^d$ defined by $F(X) = f( \bP\circ X^{-1})$. We say that $f$ is $L$-differentiable at $\mu$ if $F$ is Fr\'echet differentiable at some point $X$ such that $\mu = \bP\circ X^{-1}$. Denoting the Fr\'echet derivative by $DF$, it is now well known (see for instance \cite{GangboDifferentiability2019} that $DF$ is a $\sigma(X)$-measurable random variable of the form $DF(\mu, \cdot):\bR^e \to \lin(\bR^e, \bR^d)$ depending on the law of $X$ and satisfying $DF(\mu, \cdot) \in L^2\big( \bR^e, \cB(\bR^e), \mu; \lin(\bR^e, \bR^d) \big)$. We denote the $L$-derivative of $f$ at $\mu$ by the mapping $\partial_\mu f(\mu)(\cdot): \bR^e \ni x \to \partial_\mu f(\mu, x) \in \lin(\bR^e, \bR^d)$ satisfying $DF(\mu, X) = \partial_\mu f(\mu, X)$. 
	
	This derivative is known to coincide with the so-called Wasserstein derivative, as defined in for instance \cite{Ambrosio2008Gradient}, \cite{CarmonaDelarue2017book1} and \cite{GangboDifferentiability2019}. As we explained in the introduction, Lions' approach is well-fitted to probabilistic approaches for mean-field models since, very frequently, we have  a `canonical' random variable $X$ for representing the law of a given probability measure $\mu$. 
	
	\begin{example}
		\label{example:Simple_Lions1}
		As a simple example, consider the measure functional
		\begin{equation}
			\label{eq:example:Simple_Lions1.1}
			f(\mu) = \int k(x) d\mu(x) \quad \mbox{so that} \quad F(X) = \bE \Big[ k\big( X(\omega) \big) \Big]. 
		\end{equation}
		For some $h\in L^2(\Omega, \cF, \bP; \bR^e)$, we have that
		$$
		F(X+h) - F(X) = \bE\Big[ k\big( X(\omega) + h(\omega) \big) - k\big( X(\omega) \big) \Big]
		$$
		Then the Gateaux derivative in direction $h$ is
		$$
		DF(X) [h] = \bE\Big[ \nabla k(X) \cdot h \Big]
		$$
		By extending this to a continuous linear operator over the whole space and using the duality identity that $L^2(\Omega, \bP; \bR^e)^* = L^2(\Omega, \bP; \bR^e)$, we get that the Fr\'echet derivative is the $L^2(\Omega, \bP; \bR^e)$ valued function
		$$
		DF(X) = \nabla k(X). 
		$$
		Then the Lions derivative of $f$ is
		$$
		\partial_\mu f(\mu)(x) = \nabla k(x). 
		$$
		Observe that $\partial_\mu f(\mu)(x)$ is independent of $\mu$. 
		
		More generally, the measure functional
		\begin{equation}
			\label{eq:example:Simple_Lions1.2}
			f(\mu) = \underbrace{\int ... \int}_{\times n} k\Big( x_1, ..., x_n \Big) d\mu(x_1) \times ... \times d\mu(x_n) 
		\end{equation}
		has Lions derivative
		$$
		\partial_\mu f(\mu)(x) = \sum_{i=1}^n \underbrace{\int ... \int}_{\times n} \nabla_{x_i} k\Big( x_1, ..., x_n \Big) d\mu(x_1)... d\mu(x_{i-1}) d\delta_x(x_i) d\mu(x_{i+1}) ... d\mu(x_n). 
		$$
		
		Further, for $N\in \bN$, let $x_1, ..., x_N \in  \bR^e$ and let $\boldsymbol{x} = (x_1, ..., x_N)$. We define $\overline{f}: (\bR^e)^{\oplus N} \to \bR$ by
		\begin{equation}
			\label{eq:Empirical-Dist}
			\overline{f}\big( \boldsymbol{x} \big) = \overline{f}\Big( (x_1, ..., x_N) \Big) = f\Big( \bar{\mu}_N[x] \Big) 
			\quad \mbox{where}\quad
			\bar{\mu}_N\big[ \boldsymbol{x} \big] = \tfrac{1}{N} \sum_{j=1}^N \delta_{x_j}
		\end{equation}
		Then for any choice of $i\in \{1, ..., N\}$
		\begin{equation}
			\label{eq:Empirical-1Deriv}
			\nabla_i \overline{f}\big( \boldsymbol{x} \big) = \nabla_{x_i} \overline{f}\big( (x_1, ..., x_N) \big) = \tfrac{1}{N} \partial_\mu f(\bar{\mu}_N, x_i). 
		\end{equation}
		In particular, we can see there is a deep connection between the free variable of the Lions derivative and the directional component of the vector containing each element of the particle system. 
	\end{example}
	
	The second order derivatives are obtained by differentiating $\partial_{\mu} f$ with respect to $x$ (in the standard Euclidean sense) and $\mu$ (in the same Lions' sense). The two derivatives $\nabla_{x} \partial_{\mu} f$ and $\partial_{\mu} \partial_{\mu} f$ are thus very different functions: The first one is defined on ${\cP}_{2}({\bR}^e) \times {\bR}^e$ and writes $(\mu,x) \mapsto \nabla_{x} \partial_{\mu} f(\mu,x)$ whilst the second one is defined on ${\cP}_{2}({\bR}^e) \times {\bR}^e \times {\bR}^e$ and writes $(\mu,x,x') \mapsto \partial_{\mu} \partial_{\mu} f(\mu,x,x')$. The $e$-dimensional entries of $\nabla_{x} \partial_{\mu} f$ and $\partial_{\mu} \partial_{\mu} f$ are called here the \textit{free} variables, since they are integrated with respect to the measure $\mu$ itself. In words, $x$ is the free variable of $\partial_{\mu} f$ and $(x,x')$ are the free variables of $\partial_{\mu} \partial_{\mu} f$. Accordingly, the quadratic form $L^2(\Omega,\cF,\bP; \bR^e)$ associated on with respect to these two second-order derivatives is
	\begin{align*}
		L^2(\Omega,\cF,\bP; \bR^e) \ni X \mapsto &\bE^1 \Big[ \nabla_{x} \partial_{\mu} f \big( \mu, X(\omega_1) \big) \cdot X(\omega_1) \otimes X(\omega_1) \Big] 
		\\
		&+ \bE^{1,2} \Big[ \partial_{\mu} \partial_{\mu}f \big( \mu, X(\omega_1), X(\omega_2) \big) \cdot X(\omega_1) \otimes X(\omega_2) \Big]. 
	\end{align*}
	In the first term of the right-hand side, the expectation makes sense if 
	$$
	\nabla_{x} \partial_{\mu}f(\mu,X) \in L^\infty\Big( \Omega,\cF,\bP; \lin\big( (\bR^e)^{\otimes 2}, \bR^d\big) \Big),
	$$
	which is the case if $\partial_{\mu} f$ is Lipschitz continuous in $x$.
	
	\begin{example}
		\label{example:Simple_Lions2}
		Recalling the functional defined in Equation \eqref{eq:example:Simple_Lions1.1}, we note that
		$$
		\partial_\mu \partial_\mu f( \mu)(x, x') = 0 
		\quad \mbox{and}\quad
		\nabla_{x} \partial_\mu f( \mu)(x) = \nabla_x^2 k(x). 
		$$
		
		Similarly, for $n\geq 2$ the measure functional defined in Equation \eqref{eq:example:Simple_Lions1.2} satisfies
		\begin{align*}
			&\partial_\mu \partial_\mu f( \mu)(x, x') 
			= 
			\sum_{\substack{i, j=1 \\ i\neq j}}^{n} \underbrace{\int ... \int}_{\times n} \nabla_{x_i} \nabla_{x_j} k \Big( x_1, ..., x_n \Big) d\mu(x_1) ... d \delta_{x}(x_i) ... d\delta_{x'}(x_j) ...d\mu(x_n)
			\\
			&\nabla_x \partial_\mu f( \mu)(x) 
			= 
			\sum_{i=1}^n \underbrace{\int ... \int}_{\times n} \nabla_{x_i}^2 k \Big( x_1, ..., x_n\Big) d\mu(x_1)... d\delta_{x}(x_i) ... d\mu(x_n)
		\end{align*}
		
		Finally, following on from Equation \eqref{eq:Empirical-Dist} for any $i, j\in \{1, ..., N\}$ we have that 
		\begin{align}
			\nonumber
			\nabla_{ij} \overline{f} \big( \boldsymbol{x} \big) &= \nabla_{x_j} \nabla_{x_i} \overline{f}\big( (x_1, ..., x_N) \big) 
			\\
			\label{eq:Empirical-2Deriv}
			&= \tfrac{1}{N^2} \partial_\mu \partial_\mu f\Big( \bar{\mu}_N\big[ \boldsymbol{x} \big], x_i, x_j \Big) + \tfrac{1}{N} \nabla_{x} \partial_\mu f\Big( \bar{\mu}_N\big[ \boldsymbol{x} \big] , x_i \Big) \delta_{i=j}. 
		\end{align}
	\end{example}
	
	Despite the obvious differences between the two second order derivatives, both capture necessary information for the Taylor expansion and we want to find a common system of notation that easily extends to higher order derivatives. This leads us to Definition \ref{def:a} below, the principle of which can be stated as follows for the first and second order derivatives: The derivative symbol $\partial_{\mu}$ can be denoted by $\partial_{1}$ and then the two derivative symbols $\nabla_{x} \partial_{\mu}$ and $\partial_{\mu} \partial_{\mu}$ can be respectively denoted by $\partial_{(1,1)}$ and $\partial_{(1,2)}$. In the first case, the length of the index is 1, hence indicating that the derivative is of order 1. In the other two cases, the length of the vector-valued index is 2, indicating that the derivative is of order 2. Also, in the notation $\partial_{(1,1)}$, the repetition of the index $1$ indicates that we use the same free variable for the second order derivative, or equivalently that the second derivative has to be $\nabla_{x}$. In the notation $\partial_{(1,2)}$, the fact that the second index (in $(1,2)$) is different from the first one says that we use a new free variable for the second order derivative, which, in turn, must be $\partial_{\mu} \partial_{\mu}$. 
	
	\subsection{The 1-Lip sup envelope and partition sequences}
	\label{subsection:1-Lip_sup_envelope}
	
	Motivated by Example \ref{example:Simple_Lions2}, we introduce an abstraction that will allow us to capture the properties of higher order Lions derivatives beyond order 2. 
	
	\begin{definition}
		\label{def:a}
		The sup-envelope of an integer-valued sequence $(a_{k})_{k=1,...,n}$ of length $n$ is the non-decreasing sequence $(\max_{l=1,...,k} a_{l})_{k=1,...,n}$. The sup-envelope is said to be $1$-Lipschitz (or just $1$-Lip) if, for any $k \in \{2,...,n\}$, 
		\begin{equation*}
			\max_{l=1,...,k} a_{l} \leq 1+ \max_{l=1,...,k-1} a_{l}.
		\end{equation*}
		We call $A_{n}$ the collection of all ${\bN}$-valued sequences of length $n$, with $a_{1}=1$ as initial value and with a 1-Lip sup-envelope. Thus $A_n$ is the collection of all sequences $(a_k)_{k=1, ..., n} \in A_n$ taking values on $\{1, ..., n\}$ such that 
		$$
		a_1 = 1, \quad a_k \in \Big\{1, ..., 1+ \max_{l=1, ..., k-1} a_l \Big\}. 
		$$
		We refer to 
		$$
		A^n = \bigcup_{i=0}^n A_i, \quad A = \bigcup_{i=0}^{\infty} A_i
		$$
		as the collection of \emph{partition sequences of length at most $n$} and the collection of \emph{partition sequences}. We use the convention that $A_0 = \{\emptyset\}$. 
		
		Given $a\in A_n$, we denote 
		$$
		|a| = n \quad \mbox{and} \quad m[a] = \max_{i=1, ..., |a|} a_i.
		$$ 
		We denote $\llbracket a\rrbracket$ to be the equivalence class of all sequences such that
		$$
		(b_i)_{i=1, ..., n} \in \llbracket a\rrbracket \quad \iff \quad \Big\{ b^{-1}[c]: c \in \{b_i:i=1, ..., n\} \Big\} = \Big\{ a^{-1}[j]: j=1, ..., m[a] \Big\}. 
		$$ 
	\end{definition}
	
	\begin{example}
		Following Definition \ref{def:a}, we have that
		\begin{align*}
			A_0 =& \Big\{ \emptyset \Big\}, \quad A_1 = \Big\{ (1)\Big\}, \quad A_2 = \Big\{ (1,1), (1,2) \Big\}, 
			\\
			A_3 =& \Big\{ (1,1,1), (1,1,2), (1,2,1), (1,2,2), (1,2,3) \Big\}, 
			\\
			A_4 =& \Big\{ (1,1,1,1), (1,1,1,2), (1,1,2,1), (1,1,2,2), (1,1,2,3), (1,2,1,1), (1,2,1,2), 
			\\
			& \quad (1,2,1,3), (1,2,2,1), (1,2,2,2), (1,2,2,3), (1,2,3,1), (1,2,3,2), (1,2,3,3), (1,2,3,4) \Big\}, 
		\end{align*}
		and so on. 
	\end{example}
	
	Observe that by construction, an element $a=(a_{i})_{i=1,...,n}$ is a surjective mapping from $\{1,...,n\}$ onto $\big\{ 1, ..., \max[a] \big\}$. However, it must be stressed that the representation of the arrival set does not matter so much for our purpose. In short, any other arrival set of cardinality $m[a]$ could be used in our analysis. In fact, what really matters in our definition of a sequence $a=(a_{i})_{i=1,...,n} \in A_{n}$ are the repetitions, since they permit us to distinguish between \textit{already used} free variables and \textit{new} free variables generated by the application of another Lions derivative. This idea may be formalised by identifying integer-valued sequences that can be labelled by the same element of $A_{n}$. 
	For example
	$$
	(1, 2), \  (2, 1), \  (4, 5) \in \bigl\llbracket (1,2)\bigr\rrbracket.  
	$$
	
	\begin{lemma}
		\label{Lemma:Bijection-Partition-simple}
		Let $n\in \bN$. Then there is a bijection between the set of partition sequences 
		$$
		A_n \quad\mbox{and}\quad \scP\big( \{1, ..., n\}\big).
		$$
	\end{lemma}
	
	\begin{proof}
		For $a \in A_{n}$, we associate the collection of sets
		$$
		p_{k}:=a^{-1}[k] = \big\{i \in \{1, ..., n\} : a_{i} = k \big\}, \quad k \in \{1, ..., \max[a] \},
		$$
		Then the collection of sets $P = \big\{ p_{1}, ..., p_{m[a]} \big\}$ is a partition of $\{ 1, ..., n\}$. This creates a mapping $\mathfrak{m}: A_{n} \to \scP\big( \{1, ..., n\} \big)$. This mapping is injective: For a contradiction suppose that it isn't and let $a, b \in A_{n}$ such that $\mathfrak{m}[a]= \mathfrak{m}[b]$. Then we have that
		$$
		\big\{ a^{-1}[k] : k = 1, ..., m[a] \big\} = \big\{ b^{-1}[k] : k=1, ..., m[b] \big\}  
		$$
		and two sets that are equal contain the same elements so that $m[a] = m[b]$. Next we have that $\exists k_1 \in 1, ..., m[b]$ such that $a^{-1}[1] = b^{-1}[k_1]$. We know that $a_1 = 1$ so that $1\in a^{-1}[1]$ and hence $1\in b^{-1}[k_1]$. On the other hand, we also have that $b_1 = 1$ so that $k_1 = 1$ and $a^{-1}[1] = b^{-1}[1]$. 
		
		Now, suppose for an inductive hypothesis that $a^{-1}[j] = b^{-1}[j]$ for $j=1, ..., k-1<m[a]$ and let us consider the sets $a^{-1}[k]$. Let $i = \inf a^{-1}[k]$ so that $i\in \{1, ..., n\}$ is the least value such that $a_i = k$. We have that for each $\tilde{i} \in \{1, ..., i-1\}$, $a_{\tilde{i}} < k$ and by the inductive hypothesis $a_{\tilde{i}} = b_{\tilde{i}}$. By the equality of the two partitions, we know that $\exists \tilde{k}\in \{1, ..., m[b]\}$ such that $a^{-1}[k] = b^{-1}[\tilde{k}]$. Hence $b_i = \tilde{k}$. By assumption, $b\in A_n$ so that $b_i \in \{ 1, ..., k-1, k\}$ and if $b_i \in \{1, ..., k-1\}$ then we would get a contradiction that $\big\{ b^{-1}[j]: j=1, ..., m[b]\big\}$ is a partition so $\tilde{k} = k$ and we conclude that $a^{-1}[k] = b^{-1}[k]$. Hence, for all $j=1, ..., m[a]$, we have that 
		$$
		a^{-1}[j] = b^{-1}[j]
		$$
		so that $a = b$ and we conclude that $\mathfrak{m}$ is injective. 
		
		It thus remains to prove that $\mathfrak{m}$ is surjective onto $\scP\big( \{1,...,n\} \big)$. Let $P \in \scP\big( \{1,...,n\} \big)$ so that $|P| \leq n$. Then there exists $p_1 \in P$ such that $1 \in p_1$. Thus $P \backslash \{p_1\} \in \scP\big( \{1, ..., n\} \backslash p_1 \big)$ and the set $\{1, ... n\}\backslash p_1$ is ordered. In turn, there exists $x\in \{1, ..., n\} \backslash p_1$ such that $x = \min \{1, ..., n\} \backslash p_1$, which allows us to call $p_2$ the unique element of $P$ such that  $x\in p_{2}$. Continuing in this fashion, we obtain an enumeration of $P$ in the form $P= \{p_1, p_2, ...,p_{m} \}$, with $m \leq n$. 
		
		Define the sequence $(a_i)_{i=1, ..., n}$ by $a_i = j$ if and only if $i \in p_j$. We verify that $a \in A_n$. Firstly, $1\in p_1$ so that $a_1 =1$. Next, it is clear that $2 = \inf\{1, ..., n\}\backslash\{1\}$ so either
		$$
		2 =  \inf\{1, ..., n\}\backslash p_1 \quad\mbox{or}\quad 2\in p_1
		$$
		in which case $a_2 \in \{1, 2\}$. 
		
		Next suppose for $k <n$ that $\tilde{a} = (a_i)_{i=1, ..., k} \in A_{k}$. Then $1, ..., k \in \cup_{j=1}^{m[\tilde{a}]} p_{j}$ and either
		$$
		k+1 \in \bigcup_{j=1}^{m[\tilde{a}]} p_j \quad \mbox{or} \quad k+1 \in \big\{ 1, ..., n \big\} \backslash \Big( \bigcup_{j=1}^{m[a]} p_j \Big)
		$$
		In the former case, we have that $a_{k+1} \in \{1, ..., m[\tilde{a}] \}$ while in the latter we have that
		$$
		k+1 = \min \big\{ 1, ..., n\big\} \backslash \Big( \bigcup_{j=1}^{m[\tilde{a}]} p_j \Big)
		$$
		so that $k+1 \in p_{m[\tilde{a}]+1}$ and $a_{k+1} = m[\tilde{a}]+1$. Hence $(a_i)_{i=1, ..., k+1} \in A_{k+1}$. We conclude by induction. 
	\end{proof}

	The set of partitions $\scP\big( \{1, ..., n\}\big)$ is a partially ordered set with the non-strict partial order $\subseteq$ (where $P \subseteq Q$ indicates 
	\begin{equation*}
		\forall p \in P \quad \exists q \in Q \quad \mbox{such that} \quad p \subseteq q
	\end{equation*}
	that either $P=Q$ or the partition $P$ is finer than $Q$) and $\subset$ (where $P\subset Q$ indicates that the partition $P$ is finer than $Q$). We use the isomorphism between $\scP\big(\{1, ..., n\}\big)$ and $A_n$ to define a partial ordering on the set $A_n$:
	\begin{definition}
		Let $n\in \bN$. For any $a, a'\in A_n$, we say that $a\subseteq a'$ if and only if
		\begin{equation*}
			\forall j=1, ..., m[a] \quad \exists j' \in \{1, ..., m[a']\} \quad \mbox{such that} \quad a^{-1}[j] \subseteq (a')^{-1}[j']. 
		\end{equation*}
		Further, we say that $a\subset a'$ if and only if
		\begin{align*}
			&\forall j=1, ..., m[a] \quad \exists j' \in \{1, ..., m[a']\} \quad \mbox{such that} \quad a^{-1}[j] \subseteq (a')^{-1}[j'] \quad \mbox{and}
			\\
			& a\neq a'. 
		\end{align*}
		Thus $(A_n, \subseteq)$ is a poset. 
		
		Let $\boldsymbol{b}=(b_i)_{i=1, ..., n}$ and let $a \subseteq \llbracket \boldsymbol{b} \rrbracket$. We denote
		\begin{equation}
			\label{eq:K-sequence'}
			\boldsymbol{b}\circ(a) = \Big( b_{a^{-1}[j]} \Big)_{j=1, ..., m[a]}. 
		\end{equation}
	\end{definition}
	That is to say, $\boldsymbol{b}\circ (\llbracket \boldsymbol{b} \rrbracket)$ is the sequence of labels from the sequence $\boldsymbol{b}$ that are associated to each of the partition elements of the partition sequence $\llbracket \boldsymbol{b} \rrbracket$. 
	
	Further, for any $a \subset \llbracket \boldsymbol{b} \rrbracket$ the sequence $\boldsymbol{b}\circ(a)$ is a longer sequence all of whose values are contained within $\boldsymbol{b}\circ(\llbracket \boldsymbol{b} \rrbracket)$, that pairs each of the partition elements of the sequence $a$ with the label from $\boldsymbol{b}$ that is associated to the partition element of $\llbracket \boldsymbol{b} \rrbracket$ within which it is contained. 
	
	\begin{example}
		Consider the sequence $\boldsymbol{b}= (i, i, \iota)$. Then $\llbracket \boldsymbol{b} \rrbracket = (1,1,2)$ and
		\begin{equation*}
			(1,1,2)\subseteq (1,1,2) \quad\mbox{and}\quad (1,2,3) \subseteq (1,1,2). 
		\end{equation*} 
		Following on from Equation \eqref{eq:K-sequence'}, we have that
		\begin{equation*}
			\boldsymbol{b}\circ (1,1,2) = (i, \iota) \quad \mbox{and}\quad \boldsymbol{b}\circ (1,2,3) = (i, i, \iota). 
		\end{equation*}
	\end{example}
	
	\subsection{Lions-Taylor expansion}
	\label{subsection:Lions-TaylorEx}
	
	To be consistent with our discussion in the previous subsection, we start with the following reminder, taken for instance from \cites{buckdahn2017mean, chassagneux2014classical, CarmonaDelarue2017book1}:
	
	\begin{definition}
		\label{definition:Twice-Different}
		We say that a function $f:\cP_2(\bR^e) \to \bR^d$ is in $C_b^{(2)}\big( \cP_2(\bR^e); \bR^d \big)$ if 
		\begin{itemize}
			\item $f$ is continuously Lions-differentiable with Lions derivative $\partial_\mu f:\cP_2(\bR^e) \times \bR^e \to \lin( \bR^e, \bR^d)$ . 
			\item For every $\mu \in \cP_2(\bR^e)$, the $\mu$-measurable function $\partial_\mu f(\mu, \cdot)$ is differentiable with bounded and Lipschitz derivative $\nabla_x \partial_\mu f$ (with the Lipschitz property with respect to $\mu$ being for $\bW^{(1)}$) that satisfies
			$$
			\nabla_x \partial_\mu f: \cP_2(\bR^e) \times \bR^e \to \lin\big( (\bR^e)^{\otimes 2}, \bR^d\big). 
			$$ 
			\item For every $x\in \bR^e$, the function $\partial_\mu f(\cdot, x)$ has Lions derivative 
			$$
			\partial_\mu \partial_\mu f: \cP_2(\bR^e) \times \bR^e \times \bR^e \to \lin\big( (\bR^e)^{\otimes 2} , \bR^d\big)
			$$
			which is bounded and Lipschitz (with the Lipschitz property with respect to $\mu$ being for $\bW^{(1)}$). 
		\end{itemize}
	\end{definition}
	When there is no ambiguity, we will often drop the output space and write $C_b^{(2)}\big( \cP_2(\bR^e) \big)$. Before we start to generalise the above definition, we feel useful to make the following remarks:
	\begin{remark}
		\label{remark:WassersteinRemark1}
		As pointed out in Definition \ref{definition:Twice-Different}, the Lipschitz property with respect to the measure argument is understood as being for the aforementioned $\bW^{(1)}$-distance. The Lipschitz properties on the product spaces $\cP_{2}(\bR^e) \times \bR^e$ and $\cP_{2}(\bR^e) \times \bR^e \times \bR^e$ are understood for the corresponding product distances, $\bR^e$ being equipped with the Euclidean norm. Our choice to impose Lipschitz continuity with respect to the $\bW^{(1)}$-distance, which is obviously coarser than $\bW^{(2)}$, is explained in Remark \ref{remark:WassersteinRemark3} below. 
	\end{remark}
	
	\begin{remark}
		\label{remark:WassersteinRemark2}
		The requirement to have joint continuity with respect to all the arguments is in fact a strong requirement, which is known to be sub-optimal in practical applications. Indeed, Lions' derivative $\partial_{\mu} f(\mu,x)$ is typically `well-defined' at elements $x \in \bR^e$ that belong to the support of $\mu$. Put differently, the definition of the derivative outside the support of $\mu$ is somewhat arbitrary in the sense that any choice outside the support leads to a convenient  derivative. However, things become much more rigid when global continuity is imposed, as is the case here. In this case, the values of $\partial_{\mu} f(\mu,x)$ for $x$ outside the support are necessarily prescribed since we can always write $\partial_{\mu} f(\mu,x)=\lim_{n \rightarrow \infty} \partial_{\mu} f(\mu_{n},x)$, where $(\mu_{n})_{n \geq 1}$ is a sequence of fully supported probability measures that converges in $\cP_2(\bR^e)$ towards $\mu$. 
		
		Thus, there is a slight loss of generality in our definition. Actually, the same restriction is imposed in \cites{buckdahn2017mean, CarmonaDelarue2017book2, 2019arXiv180205882.2B}. Handling the general case leads to many technicalities, even when the derivatives that are studied are of order 2, see \cite{chassagneux2014classical} together with \cite{CarmonaDelarue2017book1}. 
	\end{remark}
	
	\begin{remark}
		\label{remark:WassersteinRemark3}
		The boundedness requirements on $\partial_{\mu} f$ and $\partial_{\mu} \partial_{\mu} f$ are also more demanding than what the general theory could allow. Typically, the Lions derivative of a function that is continuously differentiable and Lipschitz continuous is bounded in $L^2$, i.e., $\sup_{\mu \in \cP_{2}(\bR^e)} \int_{\bR^e} |x|^2 d\mu(x) < \infty$, and not globally in $L^\infty$, as we require here. 
		
		In fact, it is pretty easy to see that requiring the derivative to be globally bounded imposes the function $f$ to be globally Lipschitz for the $\bW^{(1)}$-distance, which is obviously stronger. Once again, similar restrictions are imposed in \cites{buckdahn2017mean, CarmonaDelarue2017book2, 2019arXiv180205882.2B}, and handling the general case leads to cumbersome technicalities. 
		
		In the end, this explains why in Remark \ref{remark:WassersteinRemark1} we decided to require Lipschitz continuity for $\bW^{(1)}$. 
	\end{remark}
	
	In particular, a function $f\in C^{(2)}_b\big( \cP_2(\bR^e) \big)$ satisfies
	$$
	\partial_\mu f(\mu, \cdot)\in C_b^1\Big(\bR^e; \lin\big( \bR^e, \bR^d\big) \Big), 
	$$
	that is; it is both bounded, $\mu$-measurable and differentiable.  
	
	\subsubsection*{Higher-order Lions differentiability}
	
	We now have all the ingredients needed to define the symbols associated with higher-order Lions' derivatives. Indeed, for $n\in \bN$ and $a\in A_n$, we define $\partial_a$ inductively by
	\begin{align*}
		\partial_{(1)} =& \partial_\mu, 
		\\
		\partial_{(a_1, ..., a_{k-1}, a_k)} =&
		\begin{cases} 
			\nabla_{x_{a_k}} \cdot \partial_{(a_1, ..., a_{k-1})} & \quad a_k \leq \max \{a_1, ..., a_{k-1}\}, 
			\\
			\partial_\mu \cdot \partial_{(a_1, ..., a_{k-1})} & \quad a_k > \max \{a_1, ..., a_{k-1}\}. 
		\end{cases}
	\end{align*}
	
	\begin{definition}
		\label{def:general:Lions:derivative}
		We say that a function $f:\cP_2(\bR^e) \to \bR^d$ belongs to $C_b^{(n)}\big( \cP_2(\bR^e); \bR^d \big)$ if there exists a collection of functions $(\partial_{a} f)_{a \in A^{n}}$ such that:
		
		\begin{enumerate}
			\item For any $k \in \{0, 1,.., n\}$, for any $a \in A_{k}$
			\begin{equation}
				\label{eq:def:general:Lions:derivative}
				\begin{split}
					\partial_{a} f : \cP_{2}( \bR^e) \times ( \bR^e)^{\times m[a]} &\rightarrow
					\lin\Big( (\bR^e)^{\otimes k}, \bR^d \Big)
					\\
					\big( \mu, x_{1},...,x_{m[a]} \big) &\mapsto \partial_{a} f(\mu,x_{1}, ..., x_{m[a]}).
				\end{split}
			\end{equation}
			
			\item For any $k \in \{1,...,n\}$ and any $a \in A_{k}$,  the function 
			$\partial_{a} f$ is bounded and Lipschitz continuous on $\cP_{2}( \bR^e) \times (\bR^e)^{\times m[a]}$, the first space being equipped with the $1$-Wasserstein distance.
			
			\item For any $k \in \{1,...,n-1\}$ and any $a \in A_{k}$, the function 
			$\partial_{a} f$ is differentiable with respect to $(x_{1},...,x_{m[a]})$ and 
			$$
			\nabla_{x_{j}} \partial_{a} f = \partial_{(a_{1},\cdots,a_{k},j)} f.
			$$
			
			\item For any $k \in \{1,\cdots,n-1\}$ and any $a \in A_{k}$, the function 
			$\partial_{a} f$ is differentiable with respect to $\mu$ and
			$$
			\partial_{\mu} \partial_{a} f = \partial_{(a_{1},\cdots,a_{k},m[a] + 1)} f.
			$$
		\end{enumerate}
	\end{definition}
	
	As with Remark \ref{remark:WassersteinRemark3}, we directly require all the derivatives $\partial_{a} f$ to be Lipschitz continuous with respect to $\bW^{(1)}$. In fact, this only makes a difference for the derivatives $\partial_{a} f$, with $|a|=n$: These derivatives should just be required to be $\bW^{(2)}$-Lipschitz continuous if we wanted to fit the standard construction of the Lions' derivative. Whenever $|a| \leq n-1$, we have by assumption that $\partial_{\mu} \partial_{a} f$ is bounded which, by the same third item, implies that $\partial_{a} f$ is necessarily $\bW^{(1)}$-Lipschitz continuous.
	
	\begin{lemma}
		\label{Empirical-nDeriv}
		Let $n, N\in \bN$, let $x_1, ..., x_N \in \bR^e$ and $\boldsymbol{x} = (x_1, ..., x_N)$. Let $f: \cP_2(\bR^e) \to \bR^d$ and suppose that $f\in C_b^{(n)}\big( \cP_2(\bR^e); \bR^d \big)$.

		We define $\overline{f}: (\bR^e)^{\oplus N} \to \bR^d$ by
		\begin{equation}
			\label{eq:Empirical-nDeriv-defn}
			\overline{f}\big( \boldsymbol{x} \big) = \overline{f}\big( (x_1, ..., x_N) \big) = f\Big( \bar{\mu}_N \big[ \boldsymbol{x} \big] \Big), \quad \bar{\mu}_N\big[ \boldsymbol{x} \big] = \frac{1}{N} \sum_{j=1}^N \delta_{x_j}. 
		\end{equation}
		Following on from \eqref{eq:Empirical-Dist}, let $\boldsymbol{i}$ be a multi-index taking values in the set $\{1, ..., N \}$ such that $|\boldsymbol{i}|\leq n$. Then
		\begin{equation}
			\label{eq:Empirical-nDeriv}
			\nabla_{\boldsymbol{i}} \overline{f}\Big( \boldsymbol{x} \Big) = \sum_{\substack{a \in A_{|\boldsymbol{i}|} \\ a \subseteq \llbracket \boldsymbol{i} \rrbracket}} \tfrac{1}{N^{m[a]}} \cdot \partial_a f \Big( \bar{\mu}_N\big[ \boldsymbol{x} \big], \boldsymbol{x}_{\boldsymbol{i}\circ (a)} \Big) 
		\end{equation}
		where 
		\begin{equation*}
			\nabla_{\boldsymbol{i}} = \nabla_{x_{i_{|\boldsymbol{i}|}}} ... \nabla_{x_{i_1}}, 
		\end{equation*}
		and (recalling Equation \eqref{eq:K-sequence'})
		\begin{equation*}
			\boldsymbol{x}_{\boldsymbol{i}\circ (a)} = \big( x_{i_{a^{-1}[1]}}, ..., x_{i_{a^{-1}[m[a]]}} \big). 
		\end{equation*}
	\end{lemma}
	In words, the summation in Equation \eqref{eq:Empirical-nDeriv} runs over all partition sequences that are finer than the partition sequence $\llbracket \boldsymbol{i} \rrbracket$. 
	
	\begin{proof}
		We have established that Equation \eqref{eq:Empirical-nDeriv} holds in the case $|\boldsymbol{i}| = 1$ with Equation \eqref{eq:Empirical-1Deriv} and $|\boldsymbol{i}| = 2$ with Equation \eqref{eq:Empirical-2Deriv}. 
		
		We proceed via induction on $|\boldsymbol{i}|$. First assume that for $\boldsymbol{i}$ taking values in $\{1, ..., N\}$ such that $|\boldsymbol{i}| = n-1$, we have that
		\begin{equation*}
			\nabla_{\boldsymbol{i}} \overline{f} \Big( \boldsymbol{x} \Big) = \sum_{\substack{a \in A_{|\boldsymbol{i}|} \\ a\subseteq \llbracket \boldsymbol{i} \rrbracket}} \tfrac{1}{N^{m[a]}} \partial_a f \Big( \bar{\mu}_N\big[ \boldsymbol{x} \big], \boldsymbol{x}_{\boldsymbol{i}\circ (a)} \Big)
		\end{equation*}
		and $i_{n} \in \{1, ..., N\}$. Then for any $a\in A_{n-1}$ such that $a\subseteq \llbracket \boldsymbol{i} \rrbracket$, 
		\begin{align*}
			\nabla_{i_{n}} &\big(\partial_a f\big) \Big( \bar{\mu}_{N}\big[ \boldsymbol{x} \big], \boldsymbol{x}_{\boldsymbol{i}\circ (a)} \Big)
			\\
			&=
			\tfrac{1}{N} \cdot \partial_\mu \partial_a f\Big( \bar{\mu}_{N}\big[ \boldsymbol{x} \big], \boldsymbol{x}_{\boldsymbol{i}\circ (a)}, x_{i_{n}} \Big)  
			+ 
			\sum_{j=1}^{m[a]} \partial_{(a\cdot j)} f\Big( \bar{\mu}_{N}\big[ \boldsymbol{x} \big], \boldsymbol{x}_{\boldsymbol{i}\circ (a)} \Big) \cdot \delta_{\{ i_n = i_{a^{-1}[j]} \} }.  
		\end{align*}
		Hence
		\begin{align*}
			\nabla_{i_n} \nabla_{\boldsymbol{i}} &\overline{f}\Big( \boldsymbol{x} \Big) 
			= 
			\sum_{\substack{a\in A_{|\boldsymbol{i}|} \\ a \subseteq \llbracket \boldsymbol{i} \rrbracket}} \tfrac{1}{N^{m[a]}} \cdot \nabla_{i_n} \big(\partial_a f\big) \Big( \bar{\mu}_N\big[ \boldsymbol{x} \big], \boldsymbol{x}_{\boldsymbol{i}\circ (a)} \Big) 
			\\
			=& \sum_{\substack{a\in A_{|\boldsymbol{i}|} \\ a \subseteq \llbracket \boldsymbol{i} \rrbracket}} \tfrac{1}{N^{m[a]}} \bigg( \tfrac{1}{N} \cdot \partial_\mu \partial_a f\Big( \bar{\mu}_N\big[ \boldsymbol{x} \big], \boldsymbol{x}_{\boldsymbol{i}\circ (a)} \Big) + \sum_{j=1}^{m[a]} \partial_{(a\cdot j)} f\Big( \bar{\mu}_N\big[ \boldsymbol{x} \big], \boldsymbol{x}_{\boldsymbol{i}\circ (a)} \Big) \cdot \delta_{\{i_n = i_{a^{-1}[j]}\}} \bigg) 
			\\
			=&\sum_{\substack{a\in A_{|(\boldsymbol{i}, i_n)|} \\ a\subseteq \llbracket (\boldsymbol{i}, i_n)\rrbracket}} \tfrac{1}{N^{m[a]}} \cdot \partial_a f\Big( \bar{\mu}_N\big[ \boldsymbol{x} \big], \boldsymbol{x}_{(\boldsymbol{i}, i_n)\circ (a)} \Big) 
		\end{align*}
		which implies the inductive hypothesis. 
	\end{proof}
	
	On the road to a general Lions-Taylor expansion, we recall the first-order expansion, which underpins the very definition of the Lions derivative. For any two $\mu$ and $\nu$ in $\cP_{2}(\bR^e)$, let $\Pi^{\mu, \nu}$ be a measure on $(\bR^e)^{\oplus 2}$ with marginal distributions $\mu$ and $\nu$. Then, for a continuously Lions differentiable function $f : \cP_{2}(\bR^e) \rightarrow \bR^d$,
	\begin{equation}
		\label{eq:n=1:ito:lions}
		f(\nu) - f(\mu) = \int_{(\bR^e)^{\oplus 2}} \partial_{\mu} f(\mu,x) \cdot (y-x) d \Pi^{\mu,\nu}(x,y) 
		+ 
		o \biggl[ \biggl( \int_{(\bR^e)^{\oplus 2}} |y-x|^2 d\Pi^{\mu,\nu}(x,y) \biggr)^{1/2} \biggr]. 
	\end{equation}
	In fact, the remainder can be explicitly written out:
	\begin{align}
		\nonumber
		o \biggl[& \biggl( \int_{(\bR^e)^{\oplus 2}} |y-x|^2 d\Pi^{\mu,\nu}(x,y) \biggr)^{1/2} \biggr]
		\\
		\label{eq:n=1:expression:remainder}
		&=  
		\int_{0}^1 \bigg( \partial_{\mu} f \Big( \Pi_{\xi}^{\mu,\nu}, x + \xi (y-x) \Big) - \partial_{\mu} f \Big( \mu, x \Big) \bigg) \cdot (y-x) d \Pi^{\mu,\nu}(x,y),
	\end{align}
	with the notation
	\begin{equation}
		\label{eq:Pixi}
		[0,1] \ni \xi \mapsto \Pi_\xi^{\mu, \nu} = \Pi^{\mu, \nu} \circ \Big( x + \xi(y-x)\Big)^{-1} \in \cP_2(\bR^e). 
	\end{equation}
	In particular, when $f$ is in $C^{(1)}(\cP_{2}(\bR^e))$ in the sense of 
	Definition \ref{def:general:Lions:derivative}, the Landau symbol in \eqref{eq:n=1:ito:lions} can be easily upper bounded by
	\begin{equation}
		\label{eq:n=1:bound:remainder}
		\begin{split}
			o \Bigg( \biggl( \int_{(\bR^e)^{\oplus 2}} |y-x|^2 d\Pi^{\mu,\nu}(x,y) \biggr)^{1/2} \Bigg)
			&\leq O \Bigl( \bW^{(2)}(\mu,\nu)^2 \Bigr).
		\end{split}
	\end{equation}
	Finally, observe that there is no other constraint on the probability measure $\Pi^{\mu,\nu}$ than it being a \textit{coupling} of $\mu$ and $\nu$. There is no need to require any optimality (say for instance in the sense of Equation \eqref{eq:WassersteinDistance}) in the choice of the coupling. In fact, the possible accuracy of the coupling (for the $L^2$ norm) reads not only in the first term in the right-hand side but also in the second term.
	
	In order to generalise \eqref{eq:n=1:ito:lions}, we define, for $a\in A_n$,
	the corresponding differential operator, which is acting on elements $f \in C_b^{(n)}\big( \cP_2(\bR^e); \bR^d \big)$ in the following way:
	\begin{align}
		\nonumber
		\rD^a &f(\mu)[\Pi^{\mu, \nu}] 
		\\
		\label{eq:D^a:without:0}
		=& \underbrace{\int_{(\bR^e)^{\oplus 2}} ... \int_{(\bR^e)^{\oplus 2}} }_{\times m[a]} \partial_a f\Big( \mu, \boldsymbol{x}_{m\{a\}} \Big) \cdot \bigotimes_{i=1}^{|a|} ( y_{a_i} - x_{a_i}) \cdot d\big( \Pi^{\mu, \nu}\big)^{\times m[a]} \Big( (\boldsymbol{x}, \boldsymbol{y})_{m\{a\}}\Big) . 
	\end{align}
	
	Here, for compact notation we have denoted
	\begin{align*}
		&d\big( \Pi^{\mu, \nu}\big)^{\times m[a]} \Big( (\boldsymbol{x}, \boldsymbol{y})_{m\{a\}}\Big) = d\Pi^{\mu, \nu}(x_1, y_1) \times ... \times d\Pi^{\mu, \nu}(x_{m[a]}, y_{m[a]} ) \quad \mbox{and}
		\\
		& \boldsymbol{x}_{m\{a\}} = (x_1, ..., x_{m[a]}). 
	\end{align*}
	
	We also use the convention that
	$$
	\rD^{\emptyset} f(\mu)[\Pi^{\mu, \nu}]	= f(\mu)
	$$
	
	We introduce a norm on the collection of functions described in Definition \ref{def:general:Lions:derivative}:
	\begin{definition}
		\label{definition:FunctionNorm-non}
		For $a\in A_n$, we denote
		\begin{equation}
			\label{eq:definition:FunctionNorm1-non}
			\| \partial_a f\|_\infty:= \sup_{\mu \in \cP_2(\bR^e)}  \sup_{(x_1, ..., x_{m[a]}) \in (\bR^e)^{\times m[a]}} \big| \partial_a f(\mu, x_1, ..., x_{m[a]}) \big|. 
		\end{equation}
		
		Further, we denote
		\begin{equation}
			\label{eq:definition:FunctionNorm2-non}
			\left.\begin{aligned}
				\big\| \partial_a f \big\|_{\lip, \mu}
				&:=
				\sup_{\substack{ \mu, \nu\in \cP_1(\bR^e) \\ x_1, ..., x_{m[a]} \in \bR^e}} \frac{\big| \partial_a f(\mu, x_1, ..., x_{m[a]}) - \partial_a f(\nu, x_1, ..., x_{m[a]}) \big|}{\bW^{(1)}(\mu, \nu)}, 
				\\
				\big\| \partial_a f \big\|_{\lip, j}
				&:= 
				\sup_{\substack{x_1, ..., x_{m[a]} \in \bR^e \\ y_j \in \bR^e \\ \mu \in \cP_1(\bR^e)}} \frac{\big| \partial_a f(\mu, x_1, ..., x_j, ..., x_{m[a]}) - \partial_a f(\mu, x_1,..., y_j, ..., x_{m[a]}) \big|}{|x_j - y_j|}. 
			\end{aligned}\right\rbrace
		\end{equation}
		
		We define
		\begin{equation}
			\label{eq:definition:FunctionNorm-non}
			\big\| f \big\|_{C_b^{(n)}} := \sum_{ a\in A^{n}} \big\| \partial_a f \big\|_\infty
			+
			\sum_{a\in A_{n}} \bigg( \big\| \partial_a f \big\|_{\lip, \mu} + \sum_{j=1}^{m[a]} \big\| \partial_a f \big\|_{\lip, j} \bigg). 
		\end{equation}
	\end{definition}
	
	To motivate the use of the differential operator described by Equation \eqref{eq:D^a:without:0}, we provide the following Proposition which links the classical Taylor expansion with Lemma \ref{Empirical-nDeriv}. 
	\begin{proposition}
		\label{proposition:classicTay<=>LionsTay}
		Let $n, N\in \bN$. Let $f: \cP_2(\bR^e) \to \bR^d$ and suppose that $f\in C_b^{(n)}\big( \cP_2(\bR^e); \bR^d\big)$. Let $\overline{f}: (\bR^d)^{\oplus N} \to \bR^d$ be defined as in Equation \eqref{eq:Empirical-nDeriv-defn} and additionally for $\boldsymbol{x}, \boldsymbol{y}\in (\bR^e)^{\oplus N}$ we define $\Pi^{\boldsymbol{x}, \boldsymbol{y}} \in \cP_2\big( \bR^e \oplus \bR^e \big)$ by
		\begin{equation}
			\label{eq:proposition:classicTay<=>LionsTay}
			\Pi^{\boldsymbol{x}, \boldsymbol{y}} = \frac{1}{N} \sum_{j=1}^N \delta_{(x_j, y_j)}. 
		\end{equation}
		Then
		\begin{equation}
			\label{eq:proposition:classicTay<=>LionsTay1}
			\overline{f}\Big( \boldsymbol{y} \Big) = \sum_{a\in A^n} \frac{1}{|a|!} \cdot \rD^a f\Big(\mu_N\big[ \boldsymbol{x} \big] \Big)\Big[ \Pi^{\boldsymbol{x}, \boldsymbol{y}} \Big]  + R_n^{\boldsymbol{x}, \boldsymbol{y}}\big( \overline{f} \big)
		\end{equation}
		where
		\begin{align}
			\nonumber
			&R_n^{\boldsymbol{x}, \boldsymbol{y}}\big(\overline{f} \big) 
			\\
			\label{eq:proposition:classicTay<=>LionsTayRem}
			&= \tfrac{1}{(n-1)!} \sum_{a\in A_n} \underbrace{\int_{(\bR^e)^{\oplus 2}} ... \int_{(\bR^e)^{\oplus 2}}}_{\times m[a]} f^a[\Pi^{\boldsymbol{x}, \boldsymbol{y}}] \cdot \bigotimes_{i=1}^{|a|} ( \hat{y}_{a_i} - \hat{x}_{a_i} ) \cdot d\big( \Pi^{\boldsymbol{x}, \boldsymbol{y}} \big)^{\times m[a]} \Big( (\hat{\boldsymbol{x}}, \hat{\boldsymbol{y}})_{m\{a\}} \Big)
		\end{align}
		and
		\begin{align*}
			&f^a[\Pi^{\boldsymbol{x}, \boldsymbol{y}}]:= f^a[\Pi^{\boldsymbol{x}, \boldsymbol{y}}]\Big( (\hat{\boldsymbol{x}}, \hat{\boldsymbol{y}})_{m\{a\}} \Big)
			\\
			&=\int_0^1 \bigg( \partial_a f\Big( \bar{\mu}_N\big[ \boldsymbol{x}+\xi(\boldsymbol{y}-\boldsymbol{x}) \big], \big( \hat{\boldsymbol{x}} + \xi(\hat{\boldsymbol{y}} - \hat{\boldsymbol{x}} )_{m\{a\}} \big) \Big) - \partial_a f\Big( \bar{\mu}_N\big[ \boldsymbol{x} \big], \hat{\boldsymbol{x}}_{m\{a\}} \Big) \bigg) \cdot (1-\xi)^{|a|-1} d\xi. 
		\end{align*}
	
		Finally, when $\| f \|_{C_b^{(n)}}< \infty$ we have that
		\begin{equation}
			\label{eq:proposition:classicTay<=>LionsTayRem+}
			\big| R_n^{\boldsymbol{x}, \boldsymbol{y}}( \overline{f} ) \big| \leq \tfrac{1}{n!} \sum_{a \in A_n} \Big( \big\| \partial_a f \big\|_{\lip, \mu} + \sum_{j=1}^{m[a]} \big\| \partial_a f \big\|_{\lip, j} \Big) \cdot \prod_{j=1}^{m[a]} \bigg( \frac{1}{N} \sum_{i=1}^N \big| x_i - y_i\big|^{\times |a^{-1}[j]|} \bigg)
		\end{equation}
	\end{proposition}
	
	For notational purposes, it is worth highlighting the difference between $x, y \in (\bR^e)^{\oplus N}$ the vector containing all points of the empirical measures $\bar{\mu}_N[x]$ $\bar{\mu}_N[y]$ and separately $\hat{x}, \hat{y} \in \bR^e$ the free variables over which the measures are integrated.

	\begin{proof}
		Firstly, for any $\boldsymbol{x}, \boldsymbol{y}\in (\bR^e)^{\oplus N}$ note that Equation \eqref{eq:proposition:classicTay<=>LionsTay} defines a coupling between the two measures
		$$
		\bar{\mu}_N\big[ \boldsymbol{x} \big] = \frac{1}{N} \sum_{j=1}^N \delta_{x_j} 
		\quad \mbox{and} \quad 
		\bar{\mu}_N\big[ \boldsymbol{y} \big] = \frac{1}{N} \sum_{j=1}^N \delta_{y_j}. 
		$$
		Thanks to the well-known classical Taylor expansion, we have that
		\begin{equation}
			\label{eq:classic-Taylor}
			\overline{f}\Big( \boldsymbol{y} \Big) 
			= 
			\sum_{\substack{\boldsymbol{i} \\ |\boldsymbol{i}| = 0}}^n \frac{1}{|\boldsymbol{i}|!} \cdot \nabla_{\boldsymbol{i}} \overline{f} \big( \boldsymbol{x} \big) \cdot \bigotimes_{j=1}^{|\boldsymbol{i}|} \big( y_{i_j} - x_{i_j} \big) + R_n^{\boldsymbol{x}, \boldsymbol{y}}\big( \overline{f} \big) 
		\end{equation}
		where
		\begin{equation}
			\label{eq:proposition:classicTay<=>LionsTayRem*}
			R_n^{\boldsymbol{x}, \boldsymbol{y}}\big( \overline{f} \big) = \sum_{\substack{\boldsymbol{i}: \\ |\boldsymbol{i}| = n}} \int_0^1 \Big( \nabla_{\boldsymbol{i}} \overline{f}\big( \boldsymbol{x} + \xi(\boldsymbol{y}-\boldsymbol{x}) \big)  - \nabla_{\boldsymbol{i}} \overline{f}\big( \boldsymbol{x} \big) \Big) \cdot \bigotimes_{j=1}^n \big( y_{i_j} - x_{i_j} \big) \cdot (1-\xi)^{n-1} d\xi
		\end{equation}
		By applying Lemma \ref{Empirical-nDeriv}, we get that this is equivalent to	
		\begin{equation*}
			\overline{f}\Big( \boldsymbol{y} \Big) = \sum_{a\in A^n} \frac{1}{|a|!} \cdot \frac{1}{N^{m[a]}} \sum_{i_1, ..., i_{m[a]} = 1}^{N} \partial_a f\Big( \bar{\mu}_N\big[ \boldsymbol{x} \big], (\boldsymbol{x}_i)_{m\{a\}} \Big) \cdot \bigotimes_{j=1}^{|a|} ( y_i -  x_i)_{a_j} + R_n^{\boldsymbol{x}, \boldsymbol{y}}\big( \overline{f} \big) 
		\end{equation*}
		where we denoted
		\begin{equation*}
			(\boldsymbol{x}_i)_{m\{a\}} = \big( x_{i_1},..., x_{i_{m[a]}} \big) 
			\quad \mbox{and} \quad 
			(y_i - x_i)_{a_j} = \big( y_{i_{a_j}} - x_{i_{a_j}} \big). 
		\end{equation*}
		Finally, we replace the summations over indexes by integrals over the coupling $\Pi^{\boldsymbol{x}, \boldsymbol{y}}$ between empirical distributions and obtain \eqref{eq:proposition:classicTay<=>LionsTay1}. 
		
		A similar argument also implies that Equation \eqref{eq:proposition:classicTay<=>LionsTayRem*} and Equation \eqref{eq:proposition:classicTay<=>LionsTayRem} are identical. Finally, Equation \eqref{eq:proposition:classicTay<=>LionsTayRem+} follows from the identities
		\begin{align*}
			&\underbrace{\int_{(\bR^e)^{\oplus 2}} ... \int_{(\bR^e)^{\oplus 2}}}_{\times m[a]} \bigotimes_{i=1}^{|a|} ( \hat{y}_{a_i} - \hat{x}_{a_i} ) \cdot d\big( \Pi^{\boldsymbol{x}, \boldsymbol{y}} \big)^{\times m[a]} \Big( (\hat{\boldsymbol{x}}, \hat{\boldsymbol{y}})_{m\{a\}} \Big) = \prod_{j=1}^{m[a]} \bigg( \frac{1}{N} \sum_{i=1}^N \big| x_i - y_i \big|^{\times |a^{-1}[j]|} \bigg), 
			\\
			&\bigg( \partial_a f\Big( \bar{\mu}_N\big[ \boldsymbol{x}+\xi(\boldsymbol{y}-\boldsymbol{x}) \big], \big( \hat{\boldsymbol{x}} + \xi(\hat{\boldsymbol{y}} - \hat{\boldsymbol{x}} )_{m\{a\}} \big) \Big) - \partial_a f\Big( \bar{\mu}_N\big[ \boldsymbol{x} \big], \hat{\boldsymbol{x}}_{m\{a\}} \Big) \bigg)
			\\
			&\leq \Big( \big\| \partial_a f \big\|_{\lip, \mu} + \sum_{j=1}^{m[a]} \big\| \partial_a f \big\|_{\lip, j} \Big) 
			\quad \mbox{and}\quad
			\tfrac{1}{(n-1)!} \int_0^1 (1-\xi)^{n-1} d\xi = \tfrac{1}{n!}. 
		\end{align*}
	\end{proof}
	
	Next, we want to establish a Taylor expansions for a function $f:\cP_2(\bR^e) \to \bR^d$ for a general choice of $\mu \in \cP_2(\bR^d)$. Inspired by Equation \eqref{eq:proposition:classicTay<=>LionsTay1}, we obtain the following:
	\begin{theorem}[Lions-Taylor Theorem]
		\label{theorem:LionsTaylor1}
		Let $n \in \bN$ and let $f\in C_b^{(n)} \big( \cP_2(\bR^e); \bR^d \big)$. Then for any $\mu, \nu \in \cP_{n+1} (\bR^e)$ with joint distribution $\Pi^{\mu, \nu}$, we have that
		\begin{align}
			\label{eq:TaylorExpansion}
			f(\nu) &= \sum_{a\in A^n} \frac{1}{|a|!} \cdot \rD^a f(\mu)[\Pi^{\mu, \nu}]  + R_n^{\Pi^{\mu, \nu}}(f), 
		\end{align}
		where
		\begin{equation*}
			R_n^{\Pi^{\mu, \nu}}(f) 
			= \tfrac{1}{(n-1)!} \sum_{a\in A_{n}} \underbrace{\int_{(\bR^e)^{\oplus 2}} ... \int_{(\bR^e)^{\oplus 2}} }_{\times m[a]} f^{a}[\Pi^{\mu, \nu}] \cdot \bigotimes_{i=1}^{|a|} ( y_{a_i} - x_{a_i}) \cdot d\big( \Pi^{\mu, \nu}\big)^{\times m[a]} \Big( (\boldsymbol{x}, \boldsymbol{y})_{m\{a\}}\Big), 
		\end{equation*}
		and
		\begin{align*}
			f^{a}[\Pi^{\mu, \nu}] = \int_0^1 \bigg(& \partial_a f\Big( \Pi^{\mu, \nu}_\xi, \big(\boldsymbol{x} + \xi(\boldsymbol{y} - \boldsymbol{x} )\big)_{m\{a\}} \Big) 
			- \partial_a f\Big( \mu, \boldsymbol{x}_{m\{a\}} \Big) \bigg) (1-\xi)^{|a|-1} d\xi. 
		\end{align*}
		Here we denoted
		\begin{align*}
			\big( \boldsymbol{x} + \xi(\boldsymbol{y}-\boldsymbol{x}) \big)_{m\{a\}} 
			= 
			\Big( \boldsymbol{x}_1 + \xi(\boldsymbol{y}_1 - \boldsymbol{x}_1), ..., \boldsymbol{x}_{m[a]} + \xi(\boldsymbol{y}_{m[a]} - \boldsymbol{x}_{m[a]} ) \Big)
		\end{align*}
	
		The probability measure $\Pi^{\mu,\nu}_{\xi}$ is defined as in \eqref{eq:Pixi}. Further, the remainder term on the second line of \eqref{eq:TaylorExpansion} can be upper bounded by:
		\begin{equation}
			\label{eq:bound:remainder:TaylorExpansion}
			\bigl| R_n^{\Pi^{\mu, \nu}}(f) \bigr|
			\leq 
			C \int_{(\bR^e)^{\oplus 2}} |y-x|^{n+1} d\Pi^{\mu, \nu}(x, y),
		\end{equation}
		for a constant $C$ depending only on the bounds for $f$ and its derivatives (including the Lipschitz bounds).
	\end{theorem}
	
	\begin{remark}
		Notice that the integrability property at order $n+1$ (imposed on $\mu$ and $\nu$) explicitly appears in the bound \eqref{eq:bound:remainder:TaylorExpansion}. This is consistent with \eqref{eq:n=1:bound:remainder}, which corresponds to $n=1$.  
		
		Returning to Proposition \ref{proposition:classicTay<=>LionsTay} for a moment, we should note that the empirical distributions $\bar{\mu}_N[x]$ and $\bar{\mu}_N[y]$ have moments of all orders and the inclusion of moment estimates in Theorem \ref{theorem:LionsTaylor1} provides the framework within which we can take limits over the Wasserstein space to extend our Lions-Taylor expansion. 
	\end{remark}
	
	Higher order flat derivatives have been considered in \cite{TseHigher2021}, but this is to study the non-linear Fokker-Planck equation for mean-field equations. It is well understood that when one takes a Lions derivative, one introduces a new variable which takes values over the support of the measure at which one differentiates. However, the nature of the derivatives evolves as more and more iterative derivatives are taken, leading to Taylor expansions that grow super-geometrically when compared to classical tensor space expansions. 
	
	\begin{proof}
		Firstly, we observe that \eqref{eq:n=1:ito:lions}-\eqref{eq:n=1:expression:remainder} can be rewritten in the form:
		\begin{align*}
			f(\nu) - f(\mu) =\rD^{(1)} f(\mu)[ \Pi^{\mu, \nu}] + \int_{(\bR^e)^{\oplus 2}} f^{(1)}[\Pi^{\mu, \nu}] \cdot (y_1-x_1) d\Pi^{\mu, \nu}(x_1, y_1). 
		\end{align*}
		Together with \eqref{eq:n=1:bound:remainder}, this gives the result when $n=1$.
		
		Now we proceed by induction on $n$. Consider an integer $n \geq 2$ such that the conclusion of the statement holds true for any $a\in A_{n-1}$ and any $f \in C_b^{(n-1)} \big( \cP_2(\bR^e); \bR^d \big)$. In turn, for $f \in C_b^{(n)} \big( \cP_2(\bR^e);\bR^d \big)$ and $a \in A_{n-1}$, we have that $\partial_a f: \cP_2(\bR^e) \times (\bR^e)^{\times m[a]} \to \lin\big( (\bR^e)^{\otimes |a|}, \bR^d \big)$ is differentiable in all variables and the derivatives are bounded and Lipschitz. Hence
		\begin{align*}
			f^{a}[&\Pi^{\mu, \nu}] = \int_0^1 \int_0^\xi \frac{d\bigg( \partial_a f\Big(\Pi^{\mu, \nu}_\theta, \big( \boldsymbol{x} + \theta(\boldsymbol{y}-\boldsymbol{x}) \big)_{m\{a\}} \Big) \bigg)}{d\theta} d\theta (1-\xi)^{n-2} d\xi
			\\
			=&\int_0^1 \int_0^\xi \bigg( \int_{(\bR^e)^{\oplus 2}} \partial_\mu \partial_a f\Big( \Pi_\theta^{\mu, \nu}, \big( \boldsymbol{x} + \theta(\boldsymbol{y}-\boldsymbol{x}) \big)_{m\{a\}},  x_{m[a]+1} + \theta(y_{m[a]+1} - x_{m[a]+1}) \Big) 
			\\
			& \hspace{40pt} \cdot (y_{m[a]+1} - x_{m[a]+1}) d\Pi^{\mu, \nu} (x_{m[a]+1}, y_{m[a]+1}) \bigg) d\theta (1-\xi)^{n-2} d\xi 
			\\
			&+\int_0^1 \int_0^\xi \sum_{i=1}^{m[a]} \nabla_{x_i} \partial_a f\Big( \Pi_\theta^{\mu, \nu}, \big(\boldsymbol{x} + \theta(\boldsymbol{y}-\boldsymbol{x})\big)_{m\{a\}} \Big) \cdot (y_i - x_i) d\theta (1-\xi)^{n-2} d\xi
			\\
			=&\tfrac{1}{n(n-1)} \int_{(\bR^e)^{\oplus 2}} \partial_\mu \partial_a f\Big( \mu, \boldsymbol{x}_{m\{a\}}, x_{m[a]+1} \Big) \cdot (y_{m[a]+1} - x_{m[a]+1}) \cdot d\Pi^{\mu, \nu}(x_{m[a]+1}, y_{m[a]+1}) 
			\\
			&+\tfrac{1}{n(n-1)} \sum_{i=1}^{m[a]} \nabla_{x_i} \partial_a f\Big( \mu, \boldsymbol{x}_{m\{a\}} \Big) \cdot (y_i - x_i) 
			\\
			&+ \tfrac{1}{n-1} \int_0^1 \bigg( \int_{(\bR^e)^{\oplus 2}} \partial_\mu \partial_a f\Big( \Pi_\xi^{\mu, \nu}, \big( \boldsymbol{x} + \xi( \boldsymbol{y} - \boldsymbol{x}) \big)_{m\{a\}}, x_{m[a]+1} + \xi( y_{m[a]+1} - x_{m[a]+1}) \Big) 
			\\
			&\hspace{40pt} - \partial_\mu \partial_a f\Big(\mu, \boldsymbol{x}_{m\{a\}}, x_{m[a]+1} \Big)  \cdot ( y_{m[a]+1} - x_{m[a]+1}) d\Pi^{\mu, \nu}(x_{m[a]+1}, y_{m[a]+1})\bigg) (1-\xi)^{n-1} d\xi
			\\
			&+\tfrac{1}{n-1} \int_0^1 \sum_{i=1}^{m[a]} \bigg( \nabla_{x_i} \partial_a f\Big( \Pi_\xi^{\mu, \nu}, \big(\boldsymbol{x} + \xi(\boldsymbol{y} - \boldsymbol{x}) \big)_{m\{a\}} \Big) - \nabla_{x_i}\partial_a f\Big( \mu, \boldsymbol{x}_{m\{a\}}\Big) \bigg) 
			\\
			&\hspace{40pt} \cdot (y_i - x_i) (1-\xi)^{n-1} d\xi. 
		\end{align*}
		Substituting this into Equation \eqref{eq:TaylorExpansion} at rank $n-1$, we get the same expansion at rank $n$.
		
		As for the estimate \eqref{eq:bound:remainder:TaylorExpansion} of the remainder, we have that the error term $R_n^{\Pi^{\mu, \nu}}(f)$ satisfies that
		\begin{align*}
			R_n^{\Pi^{\mu, \nu}}(f)
			\leq&
			\tfrac{1}{(n-1)!} \sum_{a\in A_n} \underbrace{\int_{(\bR^e)^{\oplus 2}} ... \int_{(\bR^e)^{\oplus 2}} }_{\times m[a]} \bigg( \big\| \partial_a f \big\|_{\lip, \mu} \cdot \bW^{(1)}(\nu, \mu) + \sum_{j=1}^{m[a]} \big\| \partial_a f \big\|_{\lip, j} \cdot |y_j - x_j| \bigg) 
			\\
			&\qquad \cdot \prod_{i=1}^{|a|} |y_{a_i} - x_{a_i} | \cdot d\big( \Pi^{\mu, \nu}\big)^{\times m[a]} \Big( (\boldsymbol{x}, \boldsymbol{y})_{m\{a\}}\Big)
			\\
			\lesssim& \| f\|_{C_b^{(n)}} \cdot \sum_{a\in A_n} \bigg( \bW^{(2)} (\mu, \nu) \cdot \prod_{j=1}^{m[a]} \int_{(\bR^e)^{\oplus 2}} |y-x|^{\times |a^{-1}[j]|} d\Pi^{\mu, \nu}(x, y) 
			\\
			&\qquad + \sum_{i=1}^{m[a]}  \prod_{j=1}^{m[a]} \int_{(\bR^e)^{\oplus 2}} |y-x|^{\times (|a^{-1}[j]| + \delta_{i,j}|} d\Pi^{\mu, \nu}(x, y) \bigg)
			\\
			=& O\Bigg( \int_{(\bR^e)^{\oplus 2}} |y-x|^{n+1} d\Pi^{\mu, \nu}(x, y) \Bigg), 
		\end{align*}
		where we used H\"older's inequality.
	\end{proof}
	
	\subsection{Multivariate Lions-Taylor expansion}
	\label{subsect:Multivariate_Lions-Taylor}
	
	The Lions-Taylor expansion given in the statement of Theorem \ref{theorem:LionsTaylor1} cannot suffice for the study of mean-field equations of the form \eqref{eq:meanfield:equation}, as we need to consider functionals depending on both a Euclidean variable $x_0$ and a measure argument $\mu$. To address this increase in complexity, we must revisit the framework introduced in Section \ref{subsection:1-Lip_sup_envelope}, in order to have a convenient system of notation for the mixed derivatives with respect to $x_0$ and $\mu$. 
	
	Indeed, unlike the derivative $\partial_{a} f$ in Equation \eqref{eq:def:general:Lions:derivative}, in which the $m^{\rm th}$ variable $x_{i}$ can only appear if $\partial_{a} f$ contains at least $i$ derivatives with respect to $\mu$, the variable $x_0$ now appears in any derivatives of $f$ whenever $f$ is a function of the form $f(x_0,\mu)$.  
	
	We clarify this in the Definition \ref{definition:A_0^n} below. Intuitively, derivatives with respect to the $x_0$-component are encoded in the corresponding sequence $a$ through insertions of a `$0$'. Repeated $0$'s thus account for repeated derivatives in the direction of $x_0$. There is no constraint on the way that those $0$'s may appear in the corresponding $a$. 
	
	Firstly, we illustrate how some simple functions can be dependent on both a spatial and measure variable:
	\begin{example}
		\label{example:Simple_Lions3}
		As a simple example, consider the measure functional
		\begin{equation}
			f(x_0, \mu) = \int k(x_0, y) d\mu(y) \quad \mbox{so that} \quad F(x_0, X) = \bE \Big[ k\big(x_0,  X(\omega) \big) \Big]. 
		\end{equation}
		For some $(v, h) \in \bR^e \times L^2(\Omega, \cF, \bP; \bR^e)$, we have that
		$$
		F(x_0+v, X+h) - F(x_0, X) = \bE\Big[ k\big( x_0 + v, X(\omega) + h(\omega) \big) - k\big( x_0, X(\omega) \big) \Big]
		$$
		Then the Gateaux derivative in direction $(v,h)$ is
		$$
		DF(X) \big[ (v,h) \big] = \bE\Big[ \nabla_1 k(x_0, X) \cdot v \Big] + \bE\Big[ \nabla_2 k(x_0, X) \cdot h \Big]
		$$
		By extending this to a continuous linear operator over the whole space and using the duality identity that $L^2(\Omega, \bP; \bR^e)^* = L^2(\Omega, \bP; \bR^e)$, we get that the Fr\'echet derivative is the $L^2(\Omega, \bP; \bR^e)$ valued function
		$$
		DF(x_0, X) = \bE\Big[ \nabla_1 k\big( x_0, X(\omega) \big) \Big] + \nabla_2 k(x_0, X). 
		$$
		Then the derivatives of $f$ are
		$$
		\nabla_{x_0} f(x_0, \mu) = \int \nabla_1 k(x_0, y) d\mu(y)
		\quad\mbox{and} \quad
		\partial_\mu f(x_0, \mu)(x_1) = \nabla k(x_0, x_1). 
		$$
		In particular, when $f$ is of the form
		$$
		f(x_0, \mu) = \int k(x_0 - y) d\mu(y)
		$$
		we have that
		$$
		\nabla_{x_0} f(x_0, \mu) = \int \nabla k(x_0 - y) d\mu(y) 
		\quad \mbox{and} \quad
		\partial_{\mu} f(x_0, \mu, x_1) = \nabla k(x_0 - x_1). 
		$$
		
		Further, for $N\in \bN$ let $x_1, ..., x_N \in  \bR^e$, let $\boldsymbol{x} = (x_1, ..., x_N)$ and denote $\bar{\mu}_N\big[ \boldsymbol{x} \big] = \tfrac{1}{N} \sum_{j=1}^N \delta_{x_j}$. We define $\overline{f}_i: (\bR^e)^{\oplus N} \to \bR$ by
		\begin{equation}
			\label{eq:Empirical-Dist_0}
			\overline{f}_i\Big( \boldsymbol{x} \Big) = \overline{f}_i\Big( (x_1, ..., x_N) \Big) = f\Big(x_i, \bar{\mu}_N\big[ \boldsymbol{x} \big] \Big). 
		\end{equation}
		Then
		\begin{align}
			\nonumber
			\nabla_{j} \overline{f}_i\big( \boldsymbol{x} \big) =& \nabla_{x_j} \overline{f}_i\big( (x_1, ..., x_N) \big) 
			\\
			\label{eq:Empirical-1Deriv_0}
			=& \tfrac{1}{N} \partial_\mu f\Big( x_i, \bar{\mu}_N\big[ \boldsymbol{x} \big], x_j \Big) + \nabla_{x_0} f\Big( x_i, \bar{\mu}_N\big[ \boldsymbol{x} \big] \Big) \cdot \delta_{i = j}. 
		\end{align}
		We contrast this with Equation \eqref{eq:Empirical-1Deriv}. 
	\end{example}
	
	As illustrated in Example \ref{example:Simple_Lions2}, by considering the lift of a function $f:\bR^e \times \cP_2(\bR^e) \to \bR^d$ and taking the second order derivative, we find that the Fr\'echet derivative is naturally expressed in terms of two functions which we identify with the derivatives
	$$
	\partial_{\mu} \partial_{\mu} f (\mu, x_1, x_2)
	\quad \mbox{and} \quad
	\nabla_x \partial_{\mu} f( \mu, x). 
	$$
	However, if we consider the second order derivatives of some function of the form $f: \bR^e \times \cP_2(\bR^e) \to \bR^d$, it turns out that the Fr\'echet derivative of the lift can be expressed as five functions which are identifiable with derivatives
	\begin{align*}
		&\nabla_{x_0} \nabla_{x_0} f(x_0, \mu), 
		\quad
		\nabla_{x_0} \partial_\mu f(x_0, \mu, x_1), 
		\quad
		\partial_{\mu} \nabla_{x_0} f(x_0, \mu, x_1), 
		\\
		&\nabla_{x_1} \partial_{\mu} f(x_0, \mu, x_1)
		\quad \mbox{and} \quad
		\partial_{\mu} \partial_{\mu} f(x_0, \mu, x_1, x_2). 
	\end{align*}
	
	\begin{example}
		\label{example:Simple_Lions4}
		Following on from Equation \eqref{eq:Empirical-1Deriv_0} for any $j_1, j_2 \in \{1, ..., N\}$ we have that 
		\begin{align}
			\nonumber
			\nabla_{j_1,j_2}& \overline{f}_i\Big( \boldsymbol{x} \Big)
			=
			\nabla_{j_2} \nabla_{j_1} \overline{f}_i\Big( (x_1, ..., x_N) \Big)
			\\
			\nonumber
			=& \big(\nabla_{x_0} \nabla_{x_0} f\big)\Big( x_i, \bar{\mu}_N\big[\boldsymbol{x} \big] \Big) \cdot \delta_{j_1 = j_2 = i}
			\\
			\nonumber
			&+ \tfrac{1}{N} \big(\partial_{\mu} \nabla_{x_0} f \big)\Big( x_i, \bar{\mu}_N\big[ \boldsymbol{x} \big], x_{j_2} \Big) \cdot \delta_{j_1 = i} +  \tfrac{1}{N} \big(\nabla_{x_0} \partial_{\mu}  f \big)\Big( x_i, \bar{\mu}_N\big[ \boldsymbol{x} \big], x_{j_1} \Big) \cdot \delta_{j_2 = i} 
			\\
			\label{eq:example:Simple_Lions4}
			&+ \tfrac{1}{N} \big(\nabla_{x_1} \partial_\mu f \big)\Big( x_i, \bar{\mu}_N\big[ \boldsymbol{x} \big], x_{j_1} \Big) \delta_{j_1 = j_2} + \tfrac{1}{N^2} \big(\partial_\mu \partial_\mu f\big) \Big( x_i, \bar{\mu}_N\big[ \boldsymbol{x} \big], x_{j_1}, x_{j_2} \Big). 
		\end{align}
		Therefore, by abstracting these summations we make the observation that for $i, j, k \in \{1, ..., N\}$ not equal, 
		\begin{align*}
			&(i, i) \iff \Big\{ (0,0), (0,1), (1, 0), (1,1), (1, 2) \Big\},
			\\
			&(i, j) \iff \Big\{ (0,1), (1,2) \Big\},
			\quad
			(j, i) \iff \Big\{ (1, 0), (1,2) \Big\},
			\\
			&(j, j) \iff \Big\{ (1,1), (1,2) \Big\},
			\quad
			(j, k) \iff \Big\{ (1,2) \Big\}. 
		\end{align*}
	\end{example}
	
	With this in mind, we extend Definition \ref{def:a} to capture the additional derivatives of a function of two variables:
	\begin{definition}
		\label{definition:A_0^n}
		Let $k, n\in \bN_0$ and denote $(a\cdot b)$ to be the concatenation of two sequences $a$ and $b$. Let $A_{k,n}[0]$ be the collection of all sequences 
		$a' = (a_i')_{i=1, ..., k+n}$ of length $k+n$ taking values in $\{0, ..., n\}$ of the form $a' = \sigma( (0) \cdot a)$, where $(0) = (0)_{i=1, ..., k}$ is the sequence of length $k$ with all entries $0$, $a\in A_n$ and $\sigma$ is a $(k,n)$-shuffle, i.e., a permutation of $\{1, ..., n+k\}$ such that $\sigma(1) < ... < \sigma(k)$ and $\sigma(k+1)<... < \sigma(n+k)$. 
		
		For a given $n \in \bN$, we let 
		$$
		\A{n}= \bigcup_{k=0}^n A_{k,n-k}[0], 
		\quad
		A^n[0] = \bigcup_{i=0}^n \A{i} \quad \mbox{and} 
		\quad
		A[0] = \bigcup_{n=0}^{\infty} \A{n}. 
		$$
		Given $a\in A_n[0]$, we denote 
		$$
		|a| = n \quad \mbox{and}\quad m[a] = \max_{i=1, ..., n} a_i.
		$$ 
		We denote $\llbracket a\rrbracket_c$ to be the equivalence class of all sequences such that
		\begin{align*}
			(b_i)_{i=1, ..., n} \in \llbracket a\rrbracket_c \quad \iff \quad &\Big\{ b^{-1} \Big\} = \Big\{ a^{-1}[j]: j=0, ..., m[a] \Big\} \in \scP\big( \{1, ..., n\} \big)
			\\ 
			& \mbox{and}\quad b^{-1}[c] = a^{-1}[0]. 
		\end{align*}
		
		For $a, a'\in A_n[0]$, we say that $a \subseteq a'$ if and only if
		\begin{align*}
			&a^{-1}[0] \subseteq (a')^{-1}[0] \quad \mbox{and}
			\\
			&\forall j=1, ..., m[a]\quad \exists j' \in \{0, 1, ..., m[a']\} \quad \mbox{such that} \quad a^{-1}[j] \subseteq (a')^{-1}[j']. 
		\end{align*}
		
		For a sequence $\boldsymbol{b} = (b_i)_{i=1, ..., n}$ and $a\in A_n[0]$ such that $\boldsymbol{b} \subseteq \llbracket a \rrbracket_c$ we define
		\begin{equation}
			\label{eq:K-sequence2}
			\boldsymbol{b}\circ(a) = \Big( b_{a^{-1}[j]} \Big)_{j=1, ..., m[a]}. 
		\end{equation}	
	\end{definition}
	
	The rationale for requiring $\sigma$ to be a $(k,n)$-suffle is that the positive entries of $a$, which encode the derivatives with respect to the measure argument $\mu$, obey the prescriptions of Definition  \ref{def:a}. $k$ is the number of $0$ in the sequence and $n$ is the number of non-zero values in the sequence. 
	
	Following on from Example \ref{example:Simple_Lions4}, we remark that these five collections of partition sequences can be identified with the five sets
	\begin{align*}
		\Big\{ a' \in A_2[0]: a' \subseteq a \Big\} \quad \mbox{for each $a\in A_2[0]$. }
	\end{align*}
	
	\begin{example}
		We have
		\begin{align*}
			A_{1,1}[0] =& \Big\{ (0,1), (1,0) \Big\}, 
			\quad 
			A_{1,2}[0] = \Big\{ (0,1,1), (0,1,2), (1,0,1), (1,0,2), (1,1,0), (1,2,0) \Big\}, 
			\\
			A_{2,1}[0] =& \Big\{ (0,0,1), (0,1,0), (1,0,0) \Big\}, 
			\quad
			A_{0,3}[0] = \Big\{ (1,1,1), (1,1,2), (1,2,1), (1,2,2), (1,2,3) \Big\}
		\end{align*}
		and so on. Similarly, 
		\begin{align*}
			A_1[0] =& \Big\{ (0), (1) \}, 
			\quad 
			A_2[0] = \Big\{ (0,0), (0,1), (1, 0), (1,1), (1,2) \Big\} \quad\mbox{and}
			\\
			A_3[0] =& \Big\{ (0,0,0), (0,0,1), (0,1,0), (1,0,0), (0,1,1), (1,0,1), (1,1,0), 
			\\
			&\quad (0,1,2), (1,0,2), (1,2,0), (1,1,1), (1,1,2), (1,2,1), (1,2,2), (1,2,3) \Big\}
		\end{align*}
	\end{example}
	
	In contrast to the techniques of Section \ref{subsection:Lions-TaylorEx} where we graded partition sequences based on their length, we need a way of grading the partition sequences of $A[0]$ that captures how many elements of the sequence are tagged (equal to 0) and detagged (positive integer valued). 
	
	\begin{definition}
		\label{definition:Special-A}
		Let $\alpha, \beta>0$. We define $\scG_{\alpha, \beta}: A[0] \to \bR^{+}$ by
		$$
		\scG_{\alpha, \beta}[a] = \alpha \cdot \Big| a^{-1}\big[ 0 \big] \Big| + \beta \cdot \Big| \bigcup_{j=1}^{m[a]} a^{-1}\big[ j\big] \Big|. 
		$$
		
		For $\gamma>\alpha \wedge \beta$, we define
		$$
		A^{\gamma, \alpha, \beta}[0]:= \Big\{ a\in A[0]: \scG_{\alpha, \beta}[a] \leq \gamma \Big\}, 
		$$
		Further, we define the three sets
		\begin{equation}
			\label{eq:definition:Special-A}
			\left.
			\begin{aligned}
				A_{+}^{\gamma, \alpha, \beta}[0]:= \Big\{ a\in A[0]:\quad & \scG_{\alpha, \beta}[a] \in \big( \gamma-(\alpha\vee\beta), \gamma-(\alpha\wedge\beta) \big] , 
				\\
				&\quad \forall k= 1, ..., |a|, \quad (a_i)_{i=1, ..., |a|-k} \notin A_{+}^{\gamma, \alpha, \beta} \Big\}. 
				\\
				A_{\ast}^{\gamma, \alpha, \beta}[0]:= \Big\{ a\in A[0]:\quad & \scG_{\alpha, \beta}[a] \in \big( \gamma-(\alpha\wedge\beta), \gamma \big], 
				\\
				&\quad \forall k= 1, ..., |a|, \quad (a_i)_{i=1, ..., |a|-k} \notin A_{+}^{\gamma, \alpha, \beta} \Big\}. 
				\\
				A_{\times}^{\gamma, \alpha, \beta}[0]:= \Big\{ a\in A[0]: \quad & \scG_{\alpha, \beta}[a] \in \big( \gamma-(\alpha\wedge\beta), \gamma \big], 
				\\
				&\quad \exists k= 1, ..., |a|, \quad (a_i)_{i=1, ..., |a|-k} \in A_{+}^{\gamma, \alpha, \beta} \Big\}. 
			\end{aligned}
			\right\}
		\end{equation}
	\end{definition}
	
	The choice of the values $\alpha$, $\beta$ and $\gamma$ allow us to describe a collection of derivatives with a specified number of spatial and measure variables. All the necessary derivatives are collected neatly into the single set $A^{\gamma, \alpha, \beta}$. The three sets described in Equation \eqref{eq:definition:Special-A} allow us to list all the terms of the Taylor expansion remainder term, where each set indicating the derivatives that occur in a particular form within the remainder.  
	
	\subsubsection*{Higher-order multivariate Lions differentiability}
	
	Following our notation $\partial_{a}$ for $a \in A_{n}$, we want to define an iterated multivariate Lions derivative in terms of a partition sequence $a\in \A{n}$: for $n\in \bN$ and $a\in A_n[0]$, we define inductively
	\begin{align*}
		\partial_{(0)} =& \nabla_{x_0}, \quad \partial_{(1)} = \partial_\mu.  
		\\
		\partial_{(a_1, ..., a_{k-1}, a_k)} =&
		\begin{cases} 
			\nabla_{x_0} \cdot \partial_{(a_1, ..., a_{k-1})} & \quad a_k =0, 
			\\
			\nabla_{x_{a_k}} \cdot \partial_{(a_1, ..., a_{k-1})} & \quad 0 < a_k \leq \max \{ a_1, ..., a_{k-1}\}, 
			\\
			\partial_\mu \cdot \partial_{(a_1, ..., a_{k-1})} & \quad a_k > \max \{a_1, ..., a_{k-1}\}.  
		\end{cases}
	\end{align*}

	For brevity, given a partition sequence $a\in A_n[0]$ we denote $a^{-}:=(a_i)_{i=1, ..., n-1}\in A_{n-1}[0]$. 
	
	\begin{definition}
		\label{def:general:Lions-spatial:derivative}
		Let $\alpha, \beta>0$ and let $\gamma > \alpha \wedge \beta$. We say that a function $f: \bR^e \times \cP_2(\bR^e) \to \bR^d$ belongs in $C_b\big[ A^{\gamma, \alpha, \beta}[0] \big]\big(\bR^e \times \cP_2(\bR^e); \bR^d\big)$ if there exists a collection of functions $\big\{ \partial_a f: a\in A^{\gamma, \alpha, \beta} \big\}$ such that:
		\begin{enumerate}
			\item For each $a\in A^{\gamma, \alpha, \beta}[0]$, the function $\partial_a f$ satisfies
			\begin{align*}
				\partial_a f : \bR^e \times \cP_2(\bR^e) \times (\bR^e)^{\times m[a]} &\to \lin\big( (\bR^e)^{\otimes |a|},  \bR^d\big)
				\\
				\big( x_0, \mu, (x_1, ..., x_{m[a]}) \big) &\mapsto \partial_a f(x_0, \mu, x_1, ..., x_{m[a]}).
			\end{align*} 
			\item For each $a\in A^{\gamma, \alpha, \beta}[0]$, the functions $\partial_a f$ are bounded and Lipschitz continuous, the second space equipped with the $\bW^{(1)}$-distance. 
			\item For any $(a_i)_{i=1, ..., n} \in A^{\gamma, \alpha, \beta}[0]$ such that $a_n = 0$, the function $\partial_{a^-} f(x_0, \mu, ..., x_{m[a]})$ is differentiable with respect to $x_0$ and 
			$$
			\nabla_{x_0} \partial_{a^-} f(x_0, \mu, x_1, ..., x_{m[a]}) = \partial_a f (x_0, \mu, x_1, ..., x_{m[a]}) . 
			$$
			\item For any $(a_i)_{i=1, ..., n} \in A^{\gamma, \alpha, \beta}[0]$ such that $a_n = m[a]$, the function $\partial_{a^-} f(x_0, \mu, ..., x_{m[a]-1})$ is differentiable with respect to $\mu$ and
			$$
			\partial_\mu \partial_{a^-} f(x_0, \mu, x_1, ..., x_{m[a]-1}, x_{m[a]}) = \partial_a f(x_0, \mu, x_1, ..., x_{m[a]-1}, x_{m[a]}). 
			$$
			\item For any $(a_i)_{i=1, ..., n} \in A^{\gamma, \alpha, \beta}[0]$ such that $a_n \in \{1, ..., m[a]-1\}$, the function \\ $\partial_{a^-} f(x_0, \mu, ..., x_{m[a^-]})$ is differentiable in $x_{a_n}$ and 
			$$
			\nabla_{x_{a_n}} \partial_{a^-} f(x_0, \mu, x_1, ..., x_{m[a]}) = \partial_a f(x_0, \mu, x_1, ..., x_{m[a]}). 
			$$
		\end{enumerate}
	\end{definition}
	
	Motivated by Equation \eqref{eq:particle:system}, we consider the derivatives of the vector field $\boldsymbol{f}$ (see Equation \eqref{eq:bf-field} below):
	\begin{lemma}
		\label{lemma:Empirical-nDeriv-fullsystem}
		Let $n, N\in \bN$. Let $f: \cP_2(\bR^e) \to \bR^d$ and suppose that $f\in C_b\big[ A^{1, 1/n, 1/n} \big] \big( \bR^e \times \cP_2(\bR^e); \bR^d \big)$. For $\boldsymbol{x} = (x_1, ..., x_N) \in (\bR^e)^{\oplus N}$, we define $\boldsymbol{f}: (\bR^e)^{\oplus N} \to (\bR^d)^{\oplus N}$ by
		\begin{equation}
			\label{eq:bf-field}
			\boldsymbol{f}\Big( \boldsymbol{x} \Big) = \boldsymbol{f}\Big( (x_1, ..., x_N) \Big) = \bigoplus_{i=1}^N \overline{f}_i \Big( \boldsymbol{x} \Big)
		\end{equation}
		where $\overline{f}_i: (\bR^e)^{\oplus N} \to \bR^d$ is defined by
		\begin{equation*}
			\overline{f}_i \Big( \boldsymbol{x} \Big) = f\Big( x_i, \bar{\mu}_N\big[ \boldsymbol{x} \big] \Big), \quad \bar{\mu}_N\big[ \boldsymbol{x} \big] = \sum_{j=1}^N \delta_{x_j}. 
		\end{equation*}
		Let $\boldsymbol{j}$ be a multi-index taking values in the set $\{1, ..., N\}$ such that $|\boldsymbol{j}|\leq n$. Then
		\begin{equation}
			\label{eq:lemma:Empirical-nDeriv-fullsystem*}
			\nabla_{\boldsymbol{j}} \overline{f}_i \Big( \boldsymbol{x} \Big) 
			= 
			\sum_{\substack{a \in A_{|\boldsymbol{j}|}[0] \\ a\subseteq \llbracket \boldsymbol{j}\rrbracket_i }} \tfrac{1}{N^{m[a]}} \partial_a f \Big( x_i, \bar{\mu}_N\big[ \boldsymbol{x} \big], \boldsymbol{x}_{\boldsymbol{j}\circ(a)} \Big) 
		\end{equation}
		where 
		\begin{equation*}
			\nabla_{\boldsymbol{j}} = \nabla_{j_{|\boldsymbol{j}|}} ... \nabla_{j_1}
		\end{equation*}
		and (recalling Equation \eqref{eq:K-sequence2})
		\begin{equation*}
			\boldsymbol{x}_{\boldsymbol{j}\circ(a)} = \Big( x_{j_{a^{-1}[1]}}, ..., x_{j_{a^{-1}[m[a]]}} \Big). 
		\end{equation*}
		
		In particular, this means that
		\begin{equation}
			\label{eq:lemma:Empirical-nDeriv-fullsystem}
			\nabla_{\boldsymbol{j}} \boldsymbol{f} \Big( \boldsymbol{x} \Big) = \bigoplus_{i=1}^N \bigg( \sum_{\substack{a \in A_{|\boldsymbol{j}|}[0] \\ a\subseteq \llbracket \boldsymbol{j} \rrbracket_i }} \tfrac{1}{N^{m[a]}} \cdot \partial_a f \Big( x_i, \bar{\mu}_N\big[ \boldsymbol{x} \big], \boldsymbol{x}_{\boldsymbol{j}\circ(a)} \Big) \bigg). 
		\end{equation}
	\end{lemma}
	
	\begin{proof}
		Referring back to Equation \eqref{eq:example:Simple_Lions4}, it is simple to verify that Equation \eqref{eq:lemma:Empirical-nDeriv-fullsystem*} holds for any choice of $\boldsymbol{j}$ such that $|\boldsymbol{j}| = 1$ or $|\boldsymbol{j}| = 2$. Further, it is simple to verify that Equation \eqref{eq:lemma:Empirical-nDeriv-fullsystem*} implies \eqref{eq:lemma:Empirical-nDeriv-fullsystem}, so we proceed to prove \eqref{eq:lemma:Empirical-nDeriv-fullsystem*} via induction on $|\boldsymbol{j}|$. Suppose that $\boldsymbol{j}$ is a multi-index such that $|\boldsymbol{j}|=n-1$ and suppose additionally that we have
		\begin{equation*}
			\nabla_{\boldsymbol{j}} \overline{f}_i \Big( \boldsymbol{x} \Big) 
			= 
			\sum_{\substack{a \in A_{|\boldsymbol{j}|}[0] \\ a\subseteq \llbracket \boldsymbol{j}\rrbracket_i }} \tfrac{1}{N^{m[a]}} \partial_a f \Big( x_i, \bar{\mu}_N\big[ \boldsymbol{x} \big], \boldsymbol{x}_{\boldsymbol{j}\circ(a)} \Big). 
		\end{equation*}
		Let $j_n\in \{1, .., N\}$. Then for any $a\in A_{n-1}[0]$ such that $a\subseteq \llbracket \boldsymbol{j}\rrbracket_i$, 
		\begin{align*}
			\nabla_{j_n}& \big( \partial_a f\big) \Big( \bar{\mu}_N\big[ \boldsymbol{x} \big], \boldsymbol{x}_{\boldsymbol{j}\circ(a)} \Big) 
			\\
			&= \partial_{(a\cdot 0)} f\Big( x_i, \bar{\mu}_N\big[ \boldsymbol{x} \big], \boldsymbol{x}_{\boldsymbol{j}\circ(a)} \Big) \cdot \delta_{\{ j_n = i\}}
			+
			\sum_{\iota = 1}^{m[a]} \partial_{(a\cdot \iota)} f\Big( x_i, \bar{\mu}_N\big[ \boldsymbol{x} \big], \boldsymbol{x}_{\boldsymbol{j}\circ(a)} \Big) \cdot \delta_{\{ j_n = j_{a^{-1}[\iota]} \}} 
			\\
			&\qquad + 
			\tfrac{1}{N} \cdot \partial_\mu \partial_a f\Big( x_i, \bar{\mu}_N\big[ \boldsymbol{x} \big], \boldsymbol{x}_{\boldsymbol{j}\circ(a)}, x_{j_n} \Big). 
		\end{align*}
		Hence
		\begin{align*}
			\nabla_{j_n} \nabla_{\boldsymbol{j}} &\overline{f}_i\Big( \boldsymbol{x} \Big) = \sum_{\substack{a \in A_{|\boldsymbol{j}|}[0] \\ a \subseteq \llbracket \boldsymbol{j} \rrbracket}} \tfrac{1}{N^{m[a]}} \cdot \nabla_{j_n} \big( \partial_a f \big) \Big( x_i, \bar{\mu}_N\big[ \boldsymbol{x} \big], \boldsymbol{x}_{\boldsymbol{j}\circ(a)} \Big)
			\\
			=& \sum_{\substack{a \in A_{|\boldsymbol{j}|}[0] \\ a \subseteq \llbracket \boldsymbol{j} \rrbracket}} \tfrac{1}{N^{m[a]}} \bigg( \partial_{(a\cdot 0)} f\Big( x_i, \bar{\mu}_N\big[ \boldsymbol{x} \big], \boldsymbol{x}_{\boldsymbol{j}\circ(a)} \Big) + \sum_{\iota=1}^{m[a]} \partial_{(a\cdot \iota)} f\Big( x_i, \bar{\mu}_N\big[ \boldsymbol{x} \big], \boldsymbol{x}_{\boldsymbol{j}\circ(a)} \Big)\cdot \delta_{\{ j_n = j_{a^{-1}[\iota]} \}} 
			\\
			&\qquad + \tfrac{1}{N} \cdot \partial_\mu \partial_a f \Big(x_i, \bar{\mu}_N \big[ \boldsymbol{x} \big], \boldsymbol{x}_{\boldsymbol{j}\circ(a)} \Big)\bigg)
			\\
			=& \sum_{\substack{a \in A_{|(\boldsymbol{j}, j_n)|} \\ a\subseteq \llbracket (\boldsymbol{j}, j_n) \rrbracket}} \tfrac{1}{N^{m[a]}} \cdot \partial_a f\Big( x_i, \bar{\mu}_N\big[ \boldsymbol{x} \big], \boldsymbol{x}_{(\boldsymbol{j}, j_n)\circ(a)} \Big)
		\end{align*}
		which implies the inductive hypothesis. 
	\end{proof}
	
	We can now extend the differential operator defined in \eqref{eq:D^a:without:0} to the multivariate case. Let $k, n \in \bN_0$ and let $a\in A_{k, n}$. Let $x_0, y_0\in \bR^e$ and let $\Pi^{\mu, \nu}$ be a measure on $(\bR^e)^{\oplus 2}$ with marginal distribution $\mu, \nu \in \cP_{n+1}$. Then we define the operator
	\begin{align}
		\nonumber
		\rD^a& f(x_0, \mu)[ y_0-x_0, \Pi^{\mu, \nu}]
		\\
		\label{eq:rDa}
		=& \underbrace{\int_{(\bR^e)^{\oplus 2}} ... \int_{(\bR^e)^{\oplus 2}} }_{\times m[a]} \partial_a f\Big( x_0, \mu, \boldsymbol{x}_{m\{a\}} \Big) \cdot \bigotimes_{i=1}^{|a|} ( y_{a_i} - x_{a_i}) \cdot d\big( \Pi^{\mu, \nu}\big)^{\times m[a]} \Big( (\boldsymbol{x}, \boldsymbol{y})_{m\{a\}}\Big)
	\end{align}
	where as before we have
	\begin{align*}
		\boldsymbol{x}_{m\{a\}} =& \big( x_1, ..., x_{m[a]} \big)
		\\
		d\big( \Pi^{\mu, \nu}\big)^{\times m[a]} \Big( (\boldsymbol{x}, \boldsymbol{y})_{m\{a\}}\Big) =& d \Pi^{\mu, \nu}\big( (x_1, y_1) \big) \cdot ... \cdot d \Pi^{\mu, \nu}\big( (x_{m[a]}, y_{m[a]}) \big). 
	\end{align*}
	
	Next, we introduce a norm on the collection of functions described in Definition \ref{def:general:Lions-spatial:derivative}. For clarity, we use the convention that
	$$
	\sum_{a \in A^{\gamma, \alpha, \beta}[0]} = \sum_{\substack{ a\in A[0] \\ \scG_{\alpha, \beta}[a] \leq \gamma }} = \sum_{a \in A[0]}^{\gamma, \alpha, \beta}. 
	$$
	
	\begin{definition}
		\label{definition:FunctionNorm}
		For $a\in \A{n}$, we denote
		\begin{equation}
			\label{eq:definition:FunctionNorm1}
			\| \partial_a f\|_\infty:= \sup_{x_0 \in \bR^e} \sup_{\mu \in \cP_2(\bR^e)}  \sup_{(x_1, ..., x_{m[a]}) \in (\bR^e)^{\times m[a]}} \big| \partial_a f(x_0, \mu, x_1, ..., x_{m[a]}) \big|. 
		\end{equation}
		
		Further, we denote
		\begin{equation}
			\label{eq:definition:FunctionNorm2}
			\left.\begin{aligned}
				\big\| \partial_a f \big\|_{\lip, 0}
				&:=
				\sup_{\substack{x_0, y_0\in \bR^e \\ \mu \in \cP_1(\bR^e) \\ x_1, ..., x_{m[a]} \in \bR^e}} \frac{\big| \partial_a f(x_0, \mu, ..., x_{m[a]}) - \partial_af(y_0, \mu, ..., x_{m[a]}) \big|}{|x_0 - y_0|}, 
				\\
				\big\| \partial_a f \big\|_{\lip, \mu}
				&:=
				\sup_{\substack{ \mu, \nu\in \cP_1(\bR^e) \\ x_0, x_1, ..., x_{m[a]} \in \bR^e}} \frac{\big| \partial_a f(x_0, \mu, ..., x_{m[a]}) - \partial_af(x_0, \nu, ..., x_{m[a]}) \big|}{\bW^{(1)}(\mu, \nu)}, 
				\\
				\big\| \partial_a f \big\|_{\lip, j}
				&:= 
				\sup_{\substack{x_0,x_1, ..., x_{m[a]} \in \bR^e \\ y_j \in \bR^e \\ \mu \in \cP_1(\bR^e)}} \frac{\big| \partial_a f(x_0, \mu, ..., x_j, ..., x_{m[a]}) - \partial_af(x_0, \mu, ..., y_j, ..., x_{m[a]}) \big|}{|x_j - y_j|}. 
			\end{aligned}\right\rbrace
		\end{equation}
		
		For $f\in C_b\big[ A^{\gamma, \alpha, \beta}[0] \big]\big(\bR^e \times \cP_2(\bR^e); \bR^d\big)$, we define
		\begin{align}
			\nonumber
			\big\| f &\big\|_{C_b[A^{\gamma, \alpha, \beta}[0]]} 
			:= 
			\sum_{ a\in A[0]}^{\gamma, \alpha, \beta} \big\| \partial_a f \big\|_\infty 
			+ 
			\sum_{a\in A_{\ast}[0]}^{\gamma, \alpha, \beta} \bigg( \big\| \partial_a f \big\|_{\lip, 0} + \big\| \partial_a f \big\|_{\lip, \mu} + \sum_{j=1}^{m[a]} \big\| \partial_a f \big\|_{\lip, j} \bigg)
			\\
			\label{eq:definition:FunctionNorm}
			&+
			\left\{ 
			\begin{aligned}
				\sum_{a\in A_{+}[0]}^{\gamma, \alpha, \beta} \bigg( \big\| \partial_a f \big\|_{\lip, \mu} + \sum_{j=1}^{m[a]} \big\| \partial_a f \big\|_{\lip, j} \bigg) + \sum_{a\in A_{\times}[0]}^{\gamma, \alpha, \beta} \big\| \partial_a f \big\|_{\lip, 0} 
				\quad & \quad 
				\mbox{if $\beta>\alpha$}
				\\
				\sum_{a\in A_{\times}[0]}^{\gamma, \alpha, \beta} \bigg( \big\| \partial_a f \big\|_{\lip, \mu} + \sum_{j=1}^{m[a]} \big\| \partial_a f \big\|_{\lip, j} \bigg) + \sum_{a\in A_{+}[0]}^{\gamma, \alpha, \beta} \big\| \partial_a f \big\|_{\lip, 0}. 
				\quad &\quad 
				\mbox{if $\beta<\alpha$}
			\end{aligned}
			\right.
		\end{align}
	\end{definition}
	
	Now we extend Proposition \ref{proposition:classicTay<=>LionsTay} to the function $\boldsymbol{f}$ described above:
	\begin{proposition}
		\label{proposition:classicTay<=>LionsTay*}
		Let $n, N\in \bN$. Let $f: \bR^e \times \cP_2(\bR^e) \to \bR^d$ ans suppose that $f\in C_b\big[A^{n}[0]\big]\big( \bR^e \times \cP_2(\bR^e); \bR^d\big)$. We define $\boldsymbol{f}: (\bR^e)^{\oplus N} \to (\bR^d)^{\oplus N}$ by Equation \eqref{eq:bf-field} and additionally for $\boldsymbol{x}, \boldsymbol{y}\in (\bR^e)^{\oplus N}$ we define $\Pi^{\boldsymbol{x}, \boldsymbol{y}} \in \cP_2(\bR^e \oplus \bR^e)$ by Equation \eqref{eq:proposition:classicTay<=>LionsTay}. Then
		\begin{equation}
			\label{eq:proposition:classicTay<=>LionsTay*}
			\boldsymbol{f}\Big( \boldsymbol{y} \Big) = \sum_{a\in A^{n}[0]} \frac{1}{|a|!} \bigoplus_{i=1}^N \rD^a f \Big( x_i, \bar{\mu}_N\big[ \boldsymbol{x} \big] \Big)\Big[ y_i - x_i, \Pi^{\boldsymbol{x}, \boldsymbol{y}} \Big] + \boldsymbol{R}_n^{\boldsymbol{x}, \boldsymbol{y}}\big( \boldsymbol{f} \big)
		\end{equation}
		where
		\begin{equation*}
			\boldsymbol{R}_n^{\boldsymbol{x}, \boldsymbol{y}}\big( \boldsymbol{f} \big) 
			= \sum_{a\in A_n[0]} \tfrac{1}{(n-1)!} \cdot \bigoplus_{i = 1}^N \underbrace{\int ... \int}_{\times m[a]} f^a\big[ (y_i, x_i), \Pi^{\boldsymbol{x}, \boldsymbol{y}} \big] \cdot \bigotimes_{\iota = 1}^{|a|} \big( \hat{y}_{a_{\iota}} - \hat{x}_{a_{\iota}} \big) \cdot d\big( \Pi^{\boldsymbol{x}, \boldsymbol{y}} \big)^{\times m[a]} \Big( (\hat{\boldsymbol{x}}, \hat{\boldsymbol{y}})_{m\{a\}} \Big)
		\end{equation*}
		and
		\begin{align*}
			&f^a\big[ (y_i, x_i), \Pi^{\boldsymbol{x}, \boldsymbol{y}} \big] 
			\\
			&\qquad = \int_0^1 \bigg( \partial_a f\Big( x_i^{\xi}, \bar{\mu}_N \big[ \boldsymbol{x}+\xi(\boldsymbol{y}-\boldsymbol{x}) \big], \boldsymbol{x}_{m\{a\}}^{\xi} \Big) - \partial_a f\Big( x_i, \bar{\mu}_N\big[ \boldsymbol{x} \big], \boldsymbol{x}_{m\{a\}} \Big) \bigg) \cdot (1-\xi)^{|a| - 1} d\xi. 
			\\
			&d\big( \Pi^{\boldsymbol{x}, \boldsymbol{y}} \big)^{\times m[a]} \Big( (\hat{\boldsymbol{x}}, \hat{\boldsymbol{y}})_{m\{a\}} \Big) 
			= 
			d\delta_{(x_i, y_i)} (\hat{x}_0, \hat{y}_0) \cdot d\Pi^{\boldsymbol{x}, \boldsymbol{y}}\big( \hat{x}_1, \hat{y}_1 \big) \cdot ... \cdot d\Pi^{\boldsymbol{x}, \boldsymbol{y}}\big( \hat{x}_{m[a]}, \hat{y}_{m[a]} \big). 
		\end{align*}
	
		Finally, when $\big\| f \big\|_{C_b[A^n[0]]}< \infty$ we have that
		\begin{align}
			\nonumber
			\Big| \boldsymbol{R}_n^{\boldsymbol{x}, \boldsymbol{y}}\big( \boldsymbol{f} \big) \Big|_{(\bR^d)^{\oplus N}} \leq& \tfrac{1}{n!} \sum_{a\in A_n[0]} \bigg( \big\| \partial_a f \big\|_{\lip, 0} + \big\| \partial_a f \big\|_{\lip, \mu} + \sum_{j=1}^{m[a]} \big\| \partial_a f \big\|_{\lip, j} \bigg) 
			\\
			\label{eq:proposition:classicTay<=>LionsTayRem++}
			&\cdot \bigg( \sum_{i=1}^N \big| x_i - y_i \big|^{\times |a^{-1}[0]|} \bigg) \cdot \prod_{j=1}^{m[a]} \bigg( \frac{1}{N} \sum_{i=1}^N \big| x_i - y_i \big|^{\times |a^{-1}[j]|} \bigg) 
		\end{align}
	\end{proposition}

	\begin{proof}
		The proof is similar to that of Proposition \ref{proposition:classicTay<=>LionsTay}: Using the classical Taylor expansion Equation \eqref{eq:classic-Taylor} and applying Lemma \ref{lemma:Empirical-nDeriv-fullsystem} we can conclude that
		\begin{align*}
			\overline{f}_i\Big( \boldsymbol{y} \Big) =& \sum_{\substack{\boldsymbol{j} \\ |\boldsymbol{j}| = 0}}^n \frac{1}{|\boldsymbol{j}|!} \cdot \nabla_{\boldsymbol{j}} \overline{f}\Big( \boldsymbol{x} \Big) \cdot \bigotimes_{\iota=1}^{|\boldsymbol{j}|} (y_{j_{\iota}} - x_{j_{\iota}}) + R_n^{\boldsymbol{x}, \boldsymbol{y}} \big( \overline{f}_i \big)
			\\
			=& \sum_{\substack{\boldsymbol{j} \\ |\boldsymbol{j}| = 0}}^n \frac{1}{|\boldsymbol{j}|!} \sum_{\substack{a \in A_{|\boldsymbol{j}|}[0] \\ a\subseteq \llbracket \boldsymbol{j} \rrbracket_i }} \tfrac{1}{N^{m[a]}} \cdot \partial_a f \Big( x_i, \bar{\mu}_N\big[\boldsymbol{x}\big], \boldsymbol{x}_{\boldsymbol{j}\circ (a)} \Big) \cdot \bigotimes_{\iota=1}^{|\boldsymbol{j}|} (y_{j_{\iota}} - x_{j_{\iota}}) + R_n^{\boldsymbol{x}, \boldsymbol{y}} \big( \overline{f}_i \big)
			\\
			=&\sum_{a\in A^n[0]} \frac{1}{|a|!} \frac{1}{N^{m[a]}} \sum_{j_1, ..., j_{m[a]} = 1}^N \partial_a f\Big( x_i, \bar{\mu}_N\big[ \boldsymbol{x} \big], (\boldsymbol{x}_{j})_{m\{a\}} \Big) \cdot \bigotimes_{\iota=1}^{|a|} (y_j - x_j)_{a_\iota} + R_n^{\boldsymbol{x}, \boldsymbol{y}} \big( \overline{f}_i \big). 
		\end{align*}
		where we denote 
		\begin{equation*}
			(\boldsymbol{x}_j)_{m\{a\}} = \big( x_{j_1}, ..., x_{j_{m[a]}} \big), 
			\quad
			\big( x_{j_0}, y_{j_0} \big) = \big( x_i, y_i \big)
			\quad \mbox{and}\quad 
			(y_j - x_j)_{a_{\iota}} = \big( y_{j_{a_{\iota}}} - x_{j_{a_{\iota}}} \big), 
		\end{equation*}
		and
		\begin{align*}
			&R_n^{\boldsymbol{x}, \boldsymbol{y}}\big( \overline{f}_i \big) = \sum_{\substack{\boldsymbol{j}: \\ |\boldsymbol{j}| = n}} \int_0^1 \bigg( \nabla_{\boldsymbol{j}} \overline{f}\Big( \boldsymbol{x} + \xi(\boldsymbol{y}-\boldsymbol{x}) \Big)  - \nabla_{\boldsymbol{j}} \overline{f}\Big( \boldsymbol{x} \Big) \bigg) \cdot \bigotimes_{\iota=1}^n \big( y_{i_\iota} - x_{i_\iota} \big) \cdot (1-\xi)^{n-1} d\xi
			\\
			&= \sum_{a\in A_n[0]} \frac{1}{(n-1)!} \cdot \frac{1}{N^{m[a]}} \sum_{j_1, ..., j_{m[a]} = 1}^N \bigg( \partial_a f\Big( x_i^{\xi}, \bar{\mu}_N\big[ \boldsymbol{x}^{\xi} \big], \big( \boldsymbol{x}_j^\xi \big)_{m\{a\}} \Big) 
			\\
			&\qquad - \partial_a f \Big( x_i, \bar{\mu}\big[\boldsymbol{x}\big], (\boldsymbol{x}_j)_{m\{a\}} \Big) \bigg) \cdot \bigotimes_{\iota = 1}^{|a|} \big( y_j - x_j \big)_{a_{\iota}} 
		\end{align*}
		where 
		\begin{align*}
			&x_i^{\xi} = x_i + \xi(y_i - x_i), \quad \boldsymbol{x}^{\xi}= x\boldsymbol{+}\xi(\boldsymbol{y}-\boldsymbol{x}) \quad \mbox{and}\quad 
			\\
			&\big( \boldsymbol{x}_{j}^{\xi} \big)_{m\{a\}} = \Big( x_{j_1} + \xi(y_{j_1} - x_{j_1}), ..., x_{j_{m[a]}} + \xi(y_{j_{m[a]}} - x_{j_{m[a]}}) \Big). 
		\end{align*}
		
		Next, we replace the summations over indices by integrals over the couplings $\Pi^{\boldsymbol{x}, \boldsymbol{y}}$ and obtain
		\begin{equation*}
			\overline{f}_i\big( \boldsymbol{y}\big) = \sum_{a\in A^{n}[0]} \frac{1}{|a|!} \rD^a f \Big(x_i, \bar{\mu}_N\big[ \boldsymbol{x} \big] \Big)\Big[ y_i - x_i, \Pi^{\boldsymbol{x}, \boldsymbol{y}} \Big] + {R}_n^{\boldsymbol{x}, \boldsymbol{y}}\big( \overline{f} \big)
		\end{equation*}
		and this implies Equation \eqref{eq:proposition:classicTay<=>LionsTay*}. Finally, Equation \eqref{eq:proposition:classicTay<=>LionsTayRem++} follows from the same ideas as Equation \eqref{eq:proposition:classicTay<=>LionsTayRem+}
	\end{proof}
	A nice detail about Equation \eqref{eq:proposition:classicTay<=>LionsTayRem++} is that the norm on the left hand side is over a $d \times N$ dimensional Euclidean space so we would expect this to scale order $N$. This is the case because a single one of the $m[a]+1$ summations is not normalised. 
	
	Theorem \ref{theorem:LionsTaylor1} admits a multivariate generalisation. In contrast to Proposition \ref{proposition:classicTay<=>LionsTay*} where the number of derivatives in the spatial and measure variables was equal (see $A^n[0]$ in Equation \eqref{eq:proposition:classicTay<=>LionsTay*}), the goal of this Theorem is to describe a Taylor expansion of some function $f:\bR^e \times \cP_2(\bR^e) \to \bR^d$ that takes a different number of derivatives in the measure and spatial variables (where the number of each derivative is described by the grading $\scG_{\alpha, \beta}$):
	\begin{theorem}
		\label{theorem:LionsTaylor2}
		Let $\alpha, \beta>0$ let $\gamma>\alpha\wedge \beta$ and denote $n:= \big\lfloor \tfrac{\gamma}{\beta} \big\rfloor$. Let $f: \bR^e \times \cP_2(\bR^e) \to \bR^d$ such that $f\in C_b\big[ A^{\gamma, \alpha, \beta}[0] \big]\big(\bR^e \times \cP_2(\bR^e); \bR^d\big)$. Let $x_0, y_0\in \bR^e$ and let $\mu, \nu\in \cP_{n+1}(\bR^e)$ with joint distribution $\Pi^{\mu, \nu}$. Then we have that
		\begin{equation}
			\label{eq:FullTaylorExpansion}
			f(y_0, \nu) = \sum_{a\in A[0]}^{\gamma, \alpha, \beta} \frac{1}{|a|!}\cdot \rD^af(x_0, \mu)\big[ y_0 - x_0, \Pi^{\mu, \nu} \big] + R_{\gamma, \alpha, \beta}^{(x_0, y_0), \Pi^{\mu, \nu}}(f)
		\end{equation}
		where
		\begin{align}
			\nonumber
			&R_{\gamma, \alpha, \beta}^{(x_0, y_0), \Pi^{\mu, \nu}}(f) 
			\\
			\nonumber
			&= \sum_{a\in A_{\ast}[0]}^{\gamma, \alpha, \beta} \tfrac{1}{(|a| - 1)!} \underbrace{ \int_{(\bR^e)^{\oplus 2}} ... \int_{(\bR^e)^{\oplus 2}} }_{\times m[a]} f_{\ast}^{a}[ (x_0, y_0), \Pi^{\mu, \nu}] 
			\cdot 
			\bigotimes_{p=1}^{|a|} ( y_{a_p} - x_{a_p}) \cdot d\big( \Pi^{\mu, \nu}\big)^{\times m[a]} \Big( (\boldsymbol{x}, \boldsymbol{y})_{m\{a\}}\Big)
			\\
			\nonumber
			&\quad+\sum_{a\in A_{+}[0]}^{\gamma, \alpha, \beta} \tfrac{1}{(|a| - 1)!} \underbrace{ \int_{(\bR^e)^{\oplus 2}} ... \int_{(\bR^e)^{\oplus 2}} }_{\times m[a]} f_{+}^{a}[ (x_0, y_0), \Pi^{\mu, \nu}] 
			\cdot 
			\bigotimes_{p=1}^{|a|} ( y_{a_p} - x_{a_p}) \cdot d\big( \Pi^{\mu, \nu}\big)^{\times m[a]} \Big( (\boldsymbol{x}, \boldsymbol{y})_{m\{a\}}\Big)
			\\
			\label{eq:theorem:LionsTaylor2_remainder1}
			&\quad+\sum_{a\in A_{\times}[0]}^{\gamma, \alpha, \beta} \tfrac{1}{(|a| - 1)!} \underbrace{ \int_{(\bR^e)^{\oplus 2}} ... \int_{(\bR^e)^{\oplus 2}} }_{\times m[a]} f_{\times}^{a}[ (x_0, y_0), \Pi^{\mu, \nu}] 
			\cdot 
			\bigotimes_{p=1}^{|a|} ( y_{a_p} - x_{a_p}) \cdot d\big( \Pi^{\mu, \nu}\big)^{\times m[a]} \Big( (\boldsymbol{x}, \boldsymbol{y})_{m\{a\}}\Big). 
		\end{align}
		For compactness, we denote $x_i^{\xi} = x_i + \xi\big( y_i - x_i \big)$ and define
		\begin{equation*}
			f_{\ast}^{a}[(x_0, y_0), \Pi^{\mu, \nu}] 
			= 
			\int_0^1 \bigg( \partial_a f \Big(x_0^\xi, \Pi^{\mu, \nu}_\xi,  \boldsymbol{x}_{m\{a\}}^\xi \Big) 
			- \partial_a f \Big(x_0, \mu, \boldsymbol{x}_{m\{a\}} \Big)\bigg) \cdot (1-\xi)^{|a|-1} d\xi, 
		\end{equation*}
		\begin{equation*}
			f_{+}^{a}[(x_0, y_0), \Pi^{\mu, \nu}]
			=\left\{ 
			\begin{aligned}
				&\left.
				\begin{aligned}
					\int_0^1& \bigg( \partial_a f \Big(x_0^\xi, \Pi^{\mu, \nu}_\xi, \boldsymbol{x}_{m\{a\}}^\xi \Big) 
					\\
					&- \partial_a f \Big(x_0^\xi, \mu, \boldsymbol{x}_{m\{a\}} \Big)\bigg) \cdot (1-\xi)^{|a|-1} d\xi
				\end{aligned} \right\} \mbox{ if $\beta>\alpha$,}
				\\
				&\left.
				\begin{aligned}
					\int_0^1& \bigg( \partial_a f \Big(x_0^\xi, \Pi^{\mu, \nu}_\xi, \boldsymbol{x}_{m\{a\}}^\xi \Big) 
					\\
					&- \partial_a f \Big(x_0, \Pi^{\mu, \nu}_\xi, \boldsymbol{x}_{m\{a\}}^\xi  \Big)\bigg) \cdot (1-\xi)^{|a|-1} d\xi,
				\end{aligned} \right\} \mbox{ if $\beta<\alpha$,}
				\\
				&0 \quad \mbox{ if $\beta=\alpha$.}
			\end{aligned}\right.
		\end{equation*}
		and
		\begin{equation*}
			f_{\times}^{a}[(x_0, y_0), \Pi^{\mu, \nu}]
			=\left\{ 
			\begin{aligned}
				&\left.
				\begin{aligned}
					\int_0^1& \bigg( \partial_a f \Big(x_0^{\xi}, \mu, \boldsymbol{x}_{m\{a\}} \Big) 
					\\
					&- \partial_a f \Big(x_0, \mu, \boldsymbol{x}_{m\{a\}} \Big)\bigg) \cdot (1-\xi)^{|a|-1} d\xi
				\end{aligned} \right\} \mbox{ if $\beta>\alpha$,}
				\\
				&\left.
				\begin{aligned}
					\int_0^1& \bigg( \partial_a f \Big(x_0, \Pi^{\mu, \nu}_\xi, \boldsymbol{x}_{m\{a\}}^{\xi} \Big) 
					\\
					&- \partial_a f \Big(x_0, \mu, \boldsymbol{x}_{m\{a\}} \Big)\bigg) \cdot (1-\xi)^{|a|-1} d\xi,
				\end{aligned} \right\} \mbox{ if $\beta<\alpha$,}
				\\
				&0 \quad \mbox{ if $\beta=\alpha$.}
			\end{aligned}\right.
		\end{equation*}
		The probability measure $\Pi^{\mu,\nu}_{\xi}$ is defined as in \eqref{eq:Pixi}. 
		
		Finally, the error term satisfies 
		\begin{align}
			\label{eq:theorem:LionsTaylor2_Rem1}
			&R_{\gamma, \alpha, \beta}^{(x_0, y_0), \Pi^{\mu, \nu}}(f)
			\\
			\nonumber
			&\lesssim \sum_{a\in A_{\ast}[0]}^{\gamma, \alpha, \beta} \bigg( \big\| \partial_a f\big\|_{\lip, 0} \cdot |y_0 - x_0|^{\times (|a^{-1}[0]|+1)} \cdot \prod_{p=1}^{m[a]} \int_{(\bR^e)^{\oplus 2}} |y_p - x_p|^{\times|a^{-1}[p]|} d\Pi^{\mu, \nu}(x_p, y_p) 
			\\
			\nonumber
			&\qquad 
			+ \big\| \partial_a f\big\|_{\lip, \mu} \cdot \bW^{(1)} (\mu,\nu) \cdot |y_0 - x_0|^{\times|a^{-1}[0]|} \cdot \prod_{p=1}^{m[a]} \int_{(\bR^e)^{\oplus 2}} |y_p - x_p|^{\times|a^{-1}[p]|} d\Pi^{\mu, \nu}(x_p, y_p) 
			\\
			&\qquad
			\nonumber
			+ \sum_{q=1}^{m[a]} \big\| \partial_a f\big\|_{\lip, q} \cdot |y_0 - x_0|^{\times|a^{-1}[0]|} \cdot \prod_{p=1}^{m[a]} \int_{(\bR^e)^{\oplus 2}} |y_p - x_p|^{\times(|a^{-1}[p]| + \delta_{p,q})} d\Pi^{\mu, \nu}(x_p, y_p) \bigg)
			\\
			\nonumber
			&+\left\{
			\begin{aligned}
				\left.
				\begin{aligned}
					& \sum_{a\in A_{+}[0]}^{\gamma, \alpha, \beta} \bigg( \big\| \partial_a f\big\|_{\lip, \mu} \cdot \bW^{(1)} (\mu,\nu) \cdot |y_0 - x_0|^{\times|a^{-1}[0]|} \cdot \prod_{p=1}^{m[a]} \int_{(\bR^e)^{\oplus 2}} |y_p - x_p|^{\times|a^{-1}[p]|} d\Pi^{\mu, \nu}(x_p, y_p) 
					\\
					&\qquad
					+ \sum_{q=1}^{m[a]} \big\| \partial_a f\big\|_{\lip, q} \cdot |y_0 - x_0|^{\times|a^{-1}[0]|} \cdot \prod_{p=1}^{m[a]} \int_{(\bR^e)^{\oplus 2}} |y_p - x_p|^{\times(|a^{-1}[p]| + \delta_{p,q})} d\Pi^{\mu, \nu}(x_p, y_p) \bigg)
					\\
					&+ \sum_{a\in A_{\times}[0]}^{\gamma, \alpha, \beta} \bigg( \big\| \partial_a f\big\|_{\lip, 0} \cdot |y_0 - x_0|^{\times(|a^{-1}[0]|+1)} \cdot \prod_{p=1}^{m[a]} \int_{(\bR^e)^{\oplus 2}} |y_p - x_p|^{\times|a^{-1}[p]|} d\Pi^{\mu, \nu}(x_p, y_p) \bigg)
				\end{aligned}
				\right\} 
				\\
				\mbox{if $\alpha<\beta$, }
				\\
				\left.
				\begin{aligned}
					& \sum_{a\in A_{\times}[0]}^{\gamma, \alpha, \beta} \bigg( \big\| \partial_a f\big\|_{\lip, \mu} \cdot \bW^{(1)} (\mu,\nu) \cdot |y_0 - x_0|^{\times|a^{-1}[0]|} \cdot \prod_{p=1}^{m[a]} \int_{(\bR^e)^{\oplus 2}} |y_p - x_p|^{\times|a^{-1}[p]|} d\Pi^{\mu, \nu}(x_p, y_p) 
					\\
					&\qquad
					+ \sum_{q=1}^{m[a]} \big\| \partial_a f\big\|_{\lip, q} \cdot |y_0 - x_0|^{\times|a^{-1}[0]|} \cdot \prod_{p=1}^{m[a]} \int_{(\bR^e)^{\oplus 2}} |y_p - x_p|^{\times(|a^{-1}[p]| + \delta_{p,q})} d\Pi^{\mu, \nu}(x_p, y_p) \bigg)
					\\
					&+ \sum_{a\in A_{+}[0]}^{\gamma, \alpha, \beta} \bigg( \big\| \partial_a f\big\|_{\lip, 0} \cdot |y_0 - x_0|^{\times(|a^{-1}[0]|+1)} \cdot \prod_{p=1}^{m[a]} \int_{(\bR^e)^{\oplus 2}} |y_p - x_p|^{\times|a^{-1}[p]|} d\Pi^{\mu, \nu}(x_p, y_p) \bigg)
				\end{aligned}
				\right\}
				\\
				\mbox{if $\beta<\alpha$. }
			\end{aligned}
			\right.
		\end{align}
	\end{theorem}
	
	The proof of Theorem \ref{theorem:LionsTaylor2} is delayed until later. 
	
	
	\subsubsection*{Differentiability of multivariate Lions derivatives}
	
	A key algebraic property of Taylor expansions is that the derivatives of a function can also be expressed as a Taylor expansion and the jets that correspond to these Taylor expansions of the derivatives are a translation of the original jet in a specific sense. In this section, we replicate this property for the multivariate Lions derivatives of a function. 
	
	When considering the Taylor expansion for the function $\partial_a f$ in the approach developed below, we remark that the free variables $(x_1, ..., x_{m[a]})$ generated by $\partial_a$ are here treated as tagged variables just like the spatial variable $x_0$. This can be seen in Equation \eqref{eq:Grading_bar[a]_given-a} below. 
	
	\begin{definition}
		\label{definition:A_0^n-H}
		Let $k, n\in \bN_0$ and let $H$ be a set that (for the purposes of clarity) does not contain positive integers. 
		
		Let $A_{k, n}[I]$ be the collection of all sequences $a'=(a'_i)_{i=1, ..., k+n}$ of length $k+n$ of the form $a' = \sigma\big( b\cdot c\big)$ where $b = (b_i)_{i=1, ..., k}$ is a sequence of length $k$ taking values in the set $H$, $c = (c_i)_{i=1, ..., n} \in A_n$ and $\sigma$ is a $(k, n)$-shuffle. 
		
		For $n\in \bN$, we denote
		$$
		A_{n}[H] = \bigcup_{k=0}^n A_{k, n-k}[H], 
		\quad
		A^n[H] = \bigcup_{i=0}^n A_{i}[H] 
		\quad \mbox{and} \quad
		A[H] = \bigcup_{n=0}^\infty A_{n}[H]. 
		$$
		
		We denote $m:A[H] \to \bN_0$ by
		$$
		m\big[ a \big] = \max_{\substack{i = 1, ..., |a| \\ a_i \notin H}} |a_i|. 
		$$
		with the convention that $\max_{\emptyset} a_i = 0$. 
	\end{definition}
	
	We can see that when the set $H = \{0\}$, Definition \ref{definition:A_0^n-H} agrees with Definition \ref{definition:A_0^n} which is key.
	
	\begin{example}
		Let us consider a set $H = \{ h_1, h_2, h_3\}$. Then we have
		\begin{align*}
			A_{1,1}[H]=&\Big\{ (h_1, 1), (h_2, 1), (h_3, 1), (1, h_1), (1, h_2), (1, h_3) \Big\}
			\\
			A_{0, 3}[H]=&\Big\{ (1,1,1), (1,1,2), (1,2,1), (1,2,2), (1,2,3) \Big\}
			\\
			A_{2, 0}[H]=&\Big\{  (h_1, h_1), (h_1, h_2), (h_1, h_3), (h_2, h_1), (h_2, h_2), (h_2, h_3), (h_3, h_1), (h_3, h_2), (h_3, h_3) \Big\}
		\end{align*}
	\end{example}
	
	\begin{remark}
		In practice, we are going to be interested in sets of partition sequences $A[H]$ where 
		$$
		H = \big\{ a^{-1}[j]: j=1, ..., m[a] \big\} \equiv \big\{ 1, ..., m[a] \big\}
		$$
		for some fixed choice of partition sequence $a\in A_n$. As we emphasise, there is a canonical identification with the elements of this set and positive integers. However, we emphasise again that the set $H$ does not contain integers. 
	\end{remark}
	
	\begin{definition}
		\label{definition:A_0^n-a}
		Let $a\in A[0]$ and let $n\in \bN_0$. Let $A_{n}[a]$ be the collection of all sequences $\overline{a} = \big( \overline{a}_i \big)_{i=1, ..., n}$ such that 
		$$
		\overline{a}_1 \in \big\{ 0, 1, ..., m[a], m[a] + 1 \big\}, \quad \mbox{and} \quad \overline{a}_i \in \Big\{ 0, 1, ..., m[a], m[a]+1, ..., \max_{l = 1, ..., i-1} \overline{a}_i \Big\}. 
		$$
		For $n\in \bN$, we denote
		$$
		A^n[a] = \bigcup_{i=0}^n A_i[a], \quad A[a] = \bigcup_{i=0}^{\infty} A_i[a]. 
		$$
		Further, we denote $m:A[a] \to \bN_0$ by
		$$
		m\big[ \overline{a} \big]=  \max_{i=1, ..., |\overline{a}| } |a_i| - m[a]
		$$
	\end{definition}
	
	\begin{example}
		We have that
		\begin{align*}
			A_{2}[(1,2,1)] = \Big\{& (0,0), (0,1), (0,2), (0,3), (1,0), (1,1), (1,2), (1,3), 
			\\
			& (2,0), (2,1), (2,2), (2,3), (3,0), (3,1), (3,2), (3,3), (3,4) \Big\}
		\end{align*}
	\end{example}
	
	\begin{lemma}
		\label{lemma:A_0^n-H(and)A_0^n-a}
		There is an isomorphism between the sets 
		$$
		A_n\big[ \{ a^{-1}[j]: j=0, ..., m[a] \} \big] \quad \mbox{and} \quad A_n[a]. 
		$$
		Further, both of these sets are isomorphic to the set
		$$
		\Big\{ \overline{a} \in A_{n+|a|}[0]: \quad \forall i\in \{1, ..., |a|\} \quad \overline{a}_i = a_i \Big\}. 
		$$
	\end{lemma}
	
	\begin{proof}
		\textit{Step 1.}
		We compare Definition \ref{definition:A_0^n-H} and Definition \ref{definition:A_0^n-a}:  Let $n\in \bN$, let $a\in A_n[0]$ and let
		\begin{equation}
			\label{eq:lemma:A_0^n-H(and)A_0^n-a_2}
			H = \Big\{ a^{-1}[j]: j=0, 1, ..., m[a] \Big\}. 
		\end{equation}
		The set $H$ contains $m[a]+1$ elements where the element $a^{-1}[0]$ may be the empty set. For each $j=1, ..., m[a]$, the sets are distinct and non-empty and $a^{-1}[0]$ does not intersect with any of these sets. In fact, $H$ is not a partition of $\{1, ... |a|\}$ since it may contain $\emptyset$. 
		
		Treating $\emptyset$ as a possible element of $H$, we consider $\tilde{a} \in A_n[H]$. We define the mapping $\cI: A_n[H] \to A_n[a]$ by
		\begin{equation}
			\label{eq:lemma:A_0^n-H(and)A_0^n-a_3}
			\big( \cI[\tilde{a}]_i \big)_{i=1, ..., n}, 
			\quad 
			\cI[\tilde{a}]_i = 
			\begin{cases}
				j \quad & \quad \mbox{if $\tilde{a}_i \in \big\{ a^{-1}[j]: j=0, ..., m[a] \big\}$,}
				\\
				\tilde{a}_i + m[a] \quad & \quad \mbox{if $\tilde{a}_i \in \big\{ 1, ..., m[\tilde{a}] \big\}$}. 
			\end{cases}
		\end{equation}
		Then we get that
		$$
		\cI[\tilde{a}]_1 \in \big\{ 0, 1, ..., m[a], m[a] +1 \big\}
		$$
		and more generally for any $i\in 1, ..., n$, 
		\begin{align*}
			\cI[\tilde{a}]_i \in& \Big\{ 0, 1, ..., m[a], m[a] +1, ..., m[a] + \max_{k=1, ..., i-1} \tilde{a}_k \Big\}
			\\
			=& \Big\{ 0, 1, ..., m[a], m[a] +1, ..., 1+ \max_{k=1, ..., i-1} \cI[\tilde{a}]_k \Big\}
		\end{align*}
		which verifies that $\cI[\tilde{a}] \in A_n[a]$. 
		
		To show that $\cI$ is injective suppose that $\tilde{a}^1, \tilde{a}^2\in A_n[a]$ such that $\cI[\tilde{a}^1] = \cI[\tilde{a}^2]$. Then $|\cI[\tilde{a}^1]| = |\cI[\tilde{a}^2]| = n$. In particular, this means that $\cI[\tilde{a}^1]_1 = \cI[\tilde{a}^2]_1$ and their value is contained in the set $\{0, 1, ..., m[a]+1 \}$. First, suppose that $\cI[\tilde{a}^1]_1 \in \{0, 1, ..., m[a]\}$ in which case $\tilde{a}_1^1 = a^{-1}\big[ \cI[\tilde{a}^1]_1 \big]$. By symmetry, we also have that $\tilde{a}_1^2 = a^{-1}\big[ \cI[\tilde{a}^2]_1 \big] = \tilde{a}_1^1$. On the other hand, if $\cI[\tilde{a}^1]_1 \in \{m[a]+1\}$ then both $\tilde{a}_1^1 = \tilde{a}_1^2 = 1$ so that either way $\tilde{a}_1^1 = \tilde{a}_1^2$. 
		
		More generally, for any $i\in \{1, ..., n\}$ we have that $\cI[\tilde{a}^1]_i = \cI[\tilde{a}^2]_i$. Either $\cI[\tilde{a}^1]_i \in \{0, 1, ..., m[a]\}$ or $\cI[\tilde{a}^1]_i \in \{m[a] + 1, ..., 1+ \max_{k=1, ..., i-1} \cI[\tilde{a}^1]_k \}$. These sets are necessarily disjoint and 
		$$
		\max_{k=1, ..., i-1} \cI[\tilde{a}^1]_k = \max_{k=1, ..., i-1} \cI[\tilde{a}^2]_k
		$$
		so the sets are equal. If $\cI[\tilde{a}^1]_i = \cI[\tilde{a}]_i \in \{0, 1, ..., m[a]\}$ then by construction $\tilde{a}_i^1 = a^{-1}\big[ \cI[\tilde{a}^2]_i \big] = a^{-1}\big[ \cI[\tilde{a}^2]_i \big] = \tilde{a}_i^2$. On the other hand, $\tilde{a}_i^1 = \cI[\tilde{a}^1]_i - m[a] = \cI[\tilde{a}^2]_i - m[a] = \tilde{a}_i^2$ so that we have $\tilde{a}^1 = \tilde{a}^2$ and we conclude that $\cI$ is injective. 
		
		On the other hand, the inverse operator $\cI^{-1}: A_n[a] \to A_n[H]$ defined for $\overline{a} \in A_n[a]$ by
		\begin{equation}
			\label{eq:lemma:A_0^n-H(and)A_0^n-a_1}
			\big( \cI^{-1} [\overline{a}]_i \big)_{i=1, ..., n}, \quad \cI^{-1} [\overline{a}]_i = 
			\begin{cases}
				a^{-1}\big[ a_i \big] \quad& \mbox{if $\overline{a}_i \in \big\{ 0, 1, ..., m[a] \big\}$}
				\\
				\overline{a}_i - m[a] \quad& \mbox{if $\overline{a}_i> m[a]$.}
			\end{cases}
		\end{equation}
		Let $J, J' \subseteq \{1, ..., n\}$ such that $J \cap J' = \emptyset$, $J \cup J' = \{1, ..., n\}$ and
		\begin{align*}
			J =& \big\{ i\in \{1, ..., n\}: \overline{a}_i \in \{0, 1, ..., m[a]\} \big\}, 
			\\
			J' =& \big\{ i\in \{1, ..., n\}: \overline{a}_i > m[a] \big\}. 
		\end{align*}
		We define $p = |J|$ and $q = |J'|$. Then $\exists \sigma \in \Shuf(p, q)$ such that $\overline{a} = \sigma\big( b\cdot c\big)$ where
		\begin{align*}
			\big( b_i \big)_{i\in J}, \quad  b_i = \overline{a}_i
			\\
			\big( c_i \big)_{i \in J'}, \quad c_i = \overline{a}_i
		\end{align*}
		and
		$$
		\cI^{-1}\big[ \overline{a} \big] = \cI^{-1} \Big[ \sigma\big( b \cdot c \big) \Big] = \sigma\Big( \cI^{-1}\big[ b \big] \cdot \cI^{-1}\big[ c \big] \Big). 
		$$
		Thanks to Equation \eqref{eq:lemma:A_0^n-H(and)A_0^n-a_1}, we have that
		\begin{align*}
			\cI^{-1}\big[ b \big] =& \big( a^{-1} [ \overline{a}_j ] \big)_{j\in J} \quad \mbox{and}
			\\
			\cI^{-1}\big[ c \big] =& \big( \overline{a}_j \big)_{j\in J'}. 
		\end{align*}
		Hence, the sequence $\cI^{-1}\big[ b \big]$ takes its values in $H$ (see Equation \eqref{eq:lemma:A_0^n-H(and)A_0^n-a_3}). On the other hand, 
		\begin{align*}
			\cI^{-1}\big[ c \big]_1 =& 1
			\\
			\cI^{-1}\big[ c \big]_i \in& \Big\{ 1, ..., 1+ \max_{\substack{j \in J' \\ j < i}} \cI^{-1}\big[ \overline{a} \big]_j \Big\}
		\end{align*}
		so that $\cI^{-1}\big[ c \big] \in A_q$. As such we conclude that
		$$
		\cI^{-1}\big[ \overline{a} \big] \in A_{p, q}[ H] \subseteq A_{n}. 
		$$
		
		Finally, we can easily verify from Equation \eqref{eq:lemma:A_0^n-H(and)A_0^n-a_3} and \eqref{eq:lemma:A_0^n-H(and)A_0^n-a_1} that $\cI \circ \cI^{-1}$ and $\cI^{-1} \circ \cI$ are the respective identity operators on $A_n[a]$ and $A_n[H]$ respectively. 
		
		\textit{Step 2.}
		Next, we denote
		$$
		\tilde{A}_n[a] = \Big\{ \overline{a} \in A_{n+|a|}[0]: \quad \forall i\in \{1, ..., |a|\} \quad \overline{a}_i = a_i \Big\}.
		$$
		and consider the mapping $\cJ: A_n[a] \to \tilde{A}_n[a]$ by
		$$
		\cJ\big[ \overline{a} \big] = \big( a\cdot \overline{a} \big). 
		$$
		Then
		$$
		\cJ\big[ \overline{a} \big] = 
		\begin{cases}
			a_i \quad&\quad \mbox{for $i=1, ..., |a|$,}
			\\
			\overline{a}_{i - |a|} \quad&\quad \mbox{for $i=1 + |a|, ..., |\overline{a}| + |a|$. }
		\end{cases}
		$$
		Then $\cJ[\overline{a}] \in A_{n+|a|}[0]$ and $\big( \cJ[a]_i \big)_{i=1, ..., |a|} = a$ so that $\cJ[\overline{a}] \in \tilde{A}_n[a]$. It is easy to verify that $\cJ$ is injective. 
		
		On the other hand, the inverse operator $\cJ^{-1}: \tilde{A}_{n}[a] \to A_n[a]$ satisfies
		$$
		\big( \cJ^{-1} [ \hat{a}] \big)_{i=1, ..., n} = \big( \hat{a}_{i+|a| } \big)_{i=1, ..., n} 
		$$
		so that
		\begin{align*}
			\cJ^{-1}[\hat{a}]_1 = \hat{a}_{|a|+1} \in& \Big\{ 0, 1, ..., 1 + \max_{i=1, ..., |a|} \hat{a}_i \Big\}
			\quad \mbox{and} \quad
			\cJ^{-1}[\hat{a}]_i = \hat{a}_{|a|+i} \in& \Big\{ 0, 1, ..., 1 + \max_{\substack{ k:\\ k \leq i+|a| }} \hat{a}_k \Big\}. 
		\end{align*}
		Equivalently, 
		$$
		\cJ^{-1}[\hat{a}]_i \in \Big\{ 0, 1, ..., m[a], ..., 1 + \max_{\substack{k:\\k< i}} \cJ^{-1}[\hat{a}]_k \Big\}
		$$
		so that $\cJ^{-1}[\hat{a}] \in A_n[a]$. 
		
		To prove injectivity of $\cJ^{-1}$, let $\hat{a}^1, \hat{a}^2 \in \tilde{A}_n[a]$ and suppose that $\cJ^{-1}[\hat{a}^1] = \cJ^{-1}[\hat{a}^2]$. Then $\forall i \in \{1, ..., n\}$, $\cJ^{-1}[\hat{a}^1]_i = \cJ^{-1}[\hat{a}^2]_i$ so that $\hat{a}_i^1 = \hat{a}_i^2$ for any $i\in \{|a|+1, ..., |a| + n\}$. On the other hand, since $\hat{a}^1, \hat{a}^2\in \tilde{A}_n[a]$, we also have that $\forall i\in \{1, ..., n\}$, $\hat{a}_i^1 = \hat{a}_i^2$. Hence $\hat{a}^1 = \hat{a}^2$ and $\cJ^{-1}$ is injective. Thus we conclude that $\cJ$ is a bijection. 
	\end{proof}
	
	In the next lemma, we provide another interpretation of the set $\A{n}$, as we prove it to be in bijection with the set of partitions $\scP \big( \{0,1, ..., n\}\big)$ of $\{0,1, ..., n\}$. It turns out that the set of partitions is key to understanding the Lions derivative. Intuitively, the partition will indicate how the free variables generated by the iterated Lions derivatives will interact with each other, thus providing us with information on which probability space the Lions derivatives should be considered. 
	
	\begin{lemma}
		\label{Lemma:Bijection-Partition}
		For $n\in \bN$ and set $H$, there exists a bijection between the set 
		$$
		A_n[H] \quad\mbox{and}\quad \scP\big( H \cup \{1, ..., n\}\big).
		$$ 
	\end{lemma}
	
	In particular, this means that there is a bijection between the sets $A_n[0]$ and the set of partitions $\scP\big( \{0, 1, ..., n\}\big)$ which has important implications for the construction of Lions trees. Intuitively, the partition associated with an element $a \in A_{n}[0]$ is obtained by gathering (in a common element of the partition) the indices $i$ of $\{0, ..., n\}$ that have the same value $a_{i}$ in the sequence $a$, with the convention that $a_{0}=0$. 
	
	\begin{proof}
		This proof is just an adaption of Lemma \ref{Lemma:Bijection-Partition-simple} and is left to the reader. 
	\end{proof}
	
	
	Motivated by the equivalence between Definitions \ref{definition:A_0^n-H} and \ref{definition:A_0^n-a}, we are interested in finding a Lions-Taylor expansions for the function
	$$
	\partial_a f: \bR^e \times \cP_2(\bR^e) \times (\bR^e)^{\times m[a]} \to \lin\big( (\bR^e)^{\otimes |a|}, \bR^d \big)
	$$
	To do this, we introduce a grading equivalent to that of Definition \ref{definition:Special-A}:
	\begin{definition}
		Let $\overline{a}\in A_n[a]$, let $\alpha, \beta>0$ and let $\gamma> \alpha\wedge \beta$. We define $\scG_{\alpha, \beta}^a: A_n[a] \to \bR^+$ by
		\begin{equation}
			\label{eq:Grading_bar[a]_given-a}
			\scG_{\alpha, \beta}^a\big[ \overline{a} \big] = \alpha\cdot \Big| \Big\{ \overline{a}^{-1}\big[ \{ 0, 1, ..., m[a] \} \big] \Big\} \Big| + \beta \cdot \Big| \Big\{ \overline{a}^{-1} \big[ \{m[a]+1, ..., m[a] + m[\overline{a}] \} \big] \Big\} \Big|. 
		\end{equation}
		Finally, let
		$$
		A^{\gamma, \alpha, \beta}[a]:=\Big\{ \overline{a}\in A[a]: \scG_{\alpha, \beta}^a[\overline{a}] \leq \gamma \Big\} 
		$$
		and the remainder sets
		\begin{equation}
			\left.
			\begin{aligned}
				A_{+}^{\gamma, \alpha, \beta}[a]:=\Big\{ \overline{a} \in A[a]:\quad & \scG_{\alpha, \beta}^a[\overline{a}] \in \big( \gamma - (\alpha\vee\beta), \gamma - (\alpha\wedge\beta)\big] , 
				\\
				&\quad \forall k=1, ..., |\overline{a}|, \quad (\overline{a}_i)_{i=1, ..., |\overline{a}|-k} \notin A_+^{\gamma, \alpha, \beta}[a] \Big\}. 
				\\
				A_{\ast}^{\gamma, \alpha, \beta}[a]:=\Big\{ \overline{a} \in A[a]:\quad & \scG_{\alpha, \beta}^a[\overline{a}] \in \big( \gamma - (\alpha\wedge\beta), \gamma \big] , 
				\\
				&\quad \forall k=1, ..., |\overline{a}|, \quad (\overline{a}_i)_{i=1, ..., |\overline{a}|-k} \notin A_+^{\gamma, \alpha, \beta}[a] \Big\}. 
				\\
				A_{\times}^{\gamma, \alpha, \beta}[a]:=\Big\{ \overline{a} \in A[a]:\quad & \scG_{\alpha, \beta}^a[\overline{a}] \in \big( \gamma - (\alpha\wedge\beta), \gamma \big], 
				\\
				&\quad \forall k=1, ..., |\overline{a}|, \quad (\overline{a}_i)_{i=1, ..., |\overline{a}|-k} \in A_+^{\gamma, \alpha, \beta}[a] \Big\}. 
			\end{aligned}
			\right\}
		\end{equation}
	\end{definition}
	
	\begin{remark}
		\label{remark:Link-G+G^a}
		We emphasise that this grading is different from the grading $\scG_{\alpha, \beta}$ is that for any $\overline{a} \in A_n[a]$, there exists $(a \cdot \overline{a} ) \in A_{n + |a|}[0]$ but 
		$$
		\scG_{\alpha, \beta}\big[ (a\cdot \overline{a}) \big] \neq \scG_{\alpha, \beta}[a] + \scG_{\alpha, \beta}^a[\overline{a}]. 
		$$
		What this technical detail is capturing is that when we consider the Lions-Taylor expansions of some function of the form $\partial_a f:\bR^e \times \cP_2(\bR^e) \times (\bR^e)^{\times m[a]} \to \bR^e$, each of the free variables generated by the application of $\partial_a$ are treated as tagged variables when we apply $\partial_{\overline{a}}$
	\end{remark}
	
	Let $a\in A_{n}[0]$ and let $f:\bR^e \times \cP_2(\bR^e) \times (\bR^e)^{\times m[a]} \to \bR^d$. Let $x_0, y_0, x_1, y_1, ..., x_{m[a]}, y_{m[a]} \in \bR^e$ and $\mu , \nu \in \cP_2(\bR^e)$ with joint distribution $\Pi^{\mu, \nu}$. For $\overline{a} \in A_n[a]$, we define the operator
	\begin{align}
		\nonumber
		&\rD^{\overline{a}} f(x_0, \mu, \boldsymbol{x}_{m\{a\}} )\big[ y_0-x_0, \Pi^{\mu, \nu},  (\boldsymbol{y}-\boldsymbol{x})_{m\{a\}} \big]
		\\
		&= \underbrace{\int_{(\bR^e)^{\oplus 2}} ... \int_{(\bR^e)^{\oplus 2}}}_{\times m[\overline{a}]} \partial_{\overline{a}} f\Big( x_0, \mu, \boldsymbol{x}_{m\{a\}}, \boldsymbol{x}_{m\{ \overline{a} \} }\Big) 
		\cdot 
		\bigotimes_{i=1}^{|\overline{a}|} (y_{\overline{a}_i} - x_{\overline{a}_i}) \cdot d\big( \Pi^{\mu, \nu} \Big)^{\times m[\overline{a}]} \Big( (\boldsymbol{x}, \boldsymbol{y})_{m\{\overline{a}\} } \Big)
	\end{align}
	where
	\begin{align*}
		&d\big( \Pi^{\mu, \nu}\big)^{\times m[\overline{a}]}\Big( (\boldsymbol{x}, \boldsymbol{y})_{m\{\overline{a}\} } \Big) = d\Pi^{\mu, \nu}(x_{m[a]+1}, y_{m[a]+1}) \times ... \times d\Pi^{\mu, \nu}(x_{m[a]+m[\overline{a}]}, y_{m[a]+m[\overline{a}]}), 
		\\
		&\boldsymbol{x}_{m\{a\}} = (x_1, ..., x_{m[a]}) \quad \mbox{and}\quad \boldsymbol{x}_{m\{\overline{a}\}} = (x_{m[a]+1}, ..., x_{m[a] + m[\overline{a}]}), 
		\\
		&(\boldsymbol{y}-\boldsymbol{x})_{m\{a\}} = \big( y_1 - x_1, ..., y_{m[a]} - x_{m[a]} \big). 
	\end{align*}
	
	Theorem \ref{theorem:LionsTaylor2} admits the following Corollary:
	\begin{corollary}
		\label{corollary:LionsTaylor2}
		Let $\alpha, \beta>0$ and let $\gamma>\alpha\wedge\beta$ and denote $n:=\big\lfloor \tfrac{\gamma}{\beta} \big\rfloor$. Let $f: \bR^e \times \cP_2(\bR^e) \to \bR^d$ such that $f\in C_b[A^{\gamma, \alpha, \beta}]\big( \bR^e \times \cP_2(\bR^e); \bR^d\big)$. Let $x_0, y_0, x_1, y_1, ..., x_{m[a]}, y_{m[a]} \in \bR^e$ and let $\mu,\nu\in \cP_{n+1}(\bR^e)$ with joint distribution $\Pi^{\mu, \nu}$. Let $\boldsymbol{x}_{m\{a\}} = \big( x_1, ..., x_{m[a]}\big)$ and $\boldsymbol{y} = \big( y_1, ..., y_{m[a]} \big)$. 
		
		Let $a\in A^{\gamma, \alpha, \beta}$ and let $\eta = \gamma - \scG_{\alpha, \beta}[a]$. Then we have that
		\begin{align}
			\nonumber
			\partial_a f\Big( y_0, \nu, \boldsymbol{y}_{m\{a\}}\Big)
			=& 
			\sum_{\overline{a}\in A[a]}^{\eta, \alpha, \beta} \frac{1}{|\overline{a}|!}\cdot \rD^{\overline{a}} \partial_a f\Big( x_0, \mu, \boldsymbol{x}_{m\{a\}} \Big) \Big[ y_0 - x_0, \Pi^{\mu, \nu}, (\boldsymbol{y} - \boldsymbol{x})_{m\{a\}} \Big]
			\\
			&+
			R_{\eta, \alpha, \beta|a}\Big( \partial_a f\Big)\Big[ (x_0, y_0), \Pi^{\mu, \nu}, (\boldsymbol{y} - \boldsymbol{x})_{m\{a\}} \Big]
		\end{align}
		where
		\begin{align}
			\nonumber
			&R_{\eta, \alpha, \beta|a}\Big( \partial_a f\Big)\Big[ (x_0, y_0), \Pi^{\mu, \nu}, (\boldsymbol{y} - \boldsymbol{x})_{m\{a\}} \Big]
			\\
			\nonumber
			&= 
			\sum_{\overline{a} \in A_{\ast}[a]}^{\eta, \alpha, \beta} \tfrac{1}{(|\overline{a}| - 1)!} \underbrace{ \int_{(\bR^e)^{\oplus 2}} ... \int_{(\bR^e)^{\oplus 2}} }_{\times m[\overline{a}]} f_{\ast}^{\overline{a}}
			\cdot 
			\bigotimes_{p=1}^{|\overline{a}|} ( y_{a_p} - x_{a_p}) \cdot d\big( \Pi^{\mu, \nu}\big)^{\times m[\overline{a}]} \Big( (\boldsymbol{x}, \boldsymbol{y})_{m[\overline{a}]}\Big)
			\\
			\nonumber
			&\quad + \sum_{\overline{a} \in A_{+}[a]}^{\eta, \alpha, \beta} \tfrac{1}{(|\overline{a}| - 1)!} \underbrace{ \int_{(\bR^e)^{\oplus 2}} ... \int_{(\bR^e)^{\oplus 2}} }_{\times m[\overline{a}]} f_{+}^{\overline{a}}
			\cdot 
			\bigotimes_{p=1}^{|\overline{a}|} ( y_{a_p} - x_{a_p}) \cdot d\big( \Pi^{\mu, \nu}\big)^{\times m[\overline{a}]} \Big( (\boldsymbol{x}, \boldsymbol{y})_{m[\overline{a}]}\Big)
			\\
			\label{eq:corollary:LionsTaylor3:remUpper_Y}
			&\quad + \sum_{\overline{a} \in A_{\times}[a]}^{\eta, \alpha, \beta} \tfrac{1}{(|\overline{a}| - 1)!} \underbrace{ \int_{(\bR^e)^{\oplus 2}} ... \int_{(\bR^e)^{\oplus 2}} }_{\times m[\overline{a}]} f_{\times}^{\overline{a}}
			\cdot 
			\bigotimes_{p=1}^{|\overline{a}|} ( y_{a_p} - x_{a_p}) \cdot d\big( \Pi^{\mu, \nu}\big)^{\times m[\overline{a}]} \Big( (\boldsymbol{x}, \boldsymbol{y})_{m[\overline{a}]}\Big). 
		\end{align}
		For compactness, we denote $x_i^{\xi} = x_i + \xi\big( y_i - x_i \big)$ and define
		\begin{align*}
			f_{\ast}^{\overline{a}} 
			&= 
			\int_0^1 \bigg( \partial_{(a\cdot \overline{a})} f \Big(x_0^{\xi}, \Pi^{\mu, \nu}_\xi, \boldsymbol{x}_{m\{a\}}^{\xi}, \boldsymbol{x}_{m\{ \overline{a}\} }^{\xi} \Big) 
			- 
			\partial_{(a\cdot \overline{a})} f \Big(x_0, \mu, \boldsymbol{x}_{m\{a\}}, \boldsymbol{x}_{m\{ \overline{a}\}} \Big) \bigg) \cdot (1-\xi)^{|\overline{a}|-1} d\xi, 
		\end{align*}
		\begin{align*}
			f_{+}^{\overline{a}}
			&=\left\{ 
			\begin{aligned}
				&\left.
				\begin{aligned}
					\int_0^1& \bigg( \partial_{(a\cdot \overline{a})} f \Big(x_0^{\xi}, \Pi^{\mu, \nu}_\xi, \boldsymbol{x}_{m\{a\}}^{\xi}, \boldsymbol{x}_{m\{a\}}^{\xi} \Big) 
					\\
					&- \partial_{(a\cdot \overline{a})} f \Big(x_0^{\xi}, \mu, \boldsymbol{x}_{m\{a\}}^{\xi}, \boldsymbol{x}_{m\{\overline{a}\}} \Big)\bigg) \cdot (1-\xi)^{|\overline{a}|-1} d\xi
				\end{aligned} \right\} \mbox{ if $\beta>\alpha$,}
				\\
				&\left.
				\begin{aligned}
					\int_0^1& \bigg( \partial_{(a\cdot \overline{a})} f \Big(x_0^{\xi}, \Pi^{\mu, \nu}_\xi, \boldsymbol{x}_{m\{a\}}^{\xi}, \boldsymbol{x}_{m\{\overline{a}\}}^{\xi} \Big) 
					\\
					&- \partial_{(a\cdot \overline{a})} f \Big(x_0, \Pi^{\mu, \nu}_\xi, \boldsymbol{x}_{m\{a\}}, \boldsymbol{x}_{m\{\overline{a}\}}^{\xi} \Big)\bigg) \cdot (1-\xi)^{|\overline{a}|-1} d\xi,
				\end{aligned} \right\} \mbox{ if $\beta<\alpha$,}
				\\
				&0 \quad \mbox{ if $\beta=\alpha$.}
			\end{aligned}\right.
		\end{align*}
		and
		\begin{align*}
			f_{\times}^{\overline{a}}
			&=\left\{ 
			\begin{aligned}
				&\left.
				\begin{aligned}
					\int_0^1& \bigg( \partial_{(a\cdot \overline{a})} f \Big(x_0^{\xi}, \mu, \boldsymbol{x}_{m\{a\}}^{\xi}, \boldsymbol{x}_{m\{\overline{a}\}} \Big) 
					\\
					&- \partial_{(a\cdot \overline{a})} f \Big(x_0, \mu, \boldsymbol{x}_{m\{a\}}, \boldsymbol{x}_{m\{\overline{a}\}} \Big)\bigg) \cdot (1-\xi)^{|\overline{a}|-1} d\xi
				\end{aligned} \right\} \mbox{ if $\beta>\alpha$,}
				\\
				&\left.
				\begin{aligned}
					\int_0^1& \bigg( \partial_{(a\cdot \overline{a})} f \Big(x_0, \Pi^{\mu, \nu}_\xi, \boldsymbol{x}_{m\{a\}}, \boldsymbol{x}_{m\{\overline{a}\}}^{\xi} \Big) 
					\\
					&- \partial_{(a\cdot \overline{a})} f \Big(x_0, \mu, \boldsymbol{x}_{m\{a\}}, \boldsymbol{x}_{m\{\overline{a}\}} \Big)\bigg) \cdot (1-\xi)^{|\overline{a}|-1} d\xi,
				\end{aligned} \right\} \mbox{ if $\beta<\alpha$,}
				\\
				&0 \quad \mbox{ if $\beta=\alpha$.}
			\end{aligned}\right.
		\end{align*}
	\end{corollary}
	
	\begin{proof}
		The proof of Corollary \ref{corollary:LionsTaylor2} is just an extension of the proof of Theorem \ref{theorem:LionsTaylor2} that treats each of the free variables of $\partial_a f$ as tagged variables, see also Remark \ref{remark:Link-G+G^a}. 
	\end{proof}
	
	\subsection{Schwarz Theorem for Multivariate Lions derivatives}

	What this representation does not convey is that many of the derivatives $\partial_a f$ within the expansion such as Equation \eqref{eq:FullTaylorExpansion} have symmetry properties. The \emph{equality of mixed partials} refers to the phenomena that, under adequate smoothness conditions, the order of partial derivatives is interchangeable. For instance, suppose that $U, V$ are vector spaces and $f: U \to V$ has continuous $n^{th}$ order derivatives. Then for any permutation $\sigma \in \Shuf(n)$, we have that
	\begin{equation}
		\label{eq:shuffle_FrechetDeriv}
		D^n f \cdot \bigotimes_{i=1}^n h_i = D^n f \cdot \bigotimes_{i=1}^n h_{\sigma(i)}, 
	\end{equation}
	see for instance \cite{Rodney2012Calculus}. In \cite{CarmonaDelarue2017book1}, this result was applied to prove the symmetric properties of Lions derivatives. In particular, they concluded that
	\begin{equation}
		\label{eq:CD-Schwarz-1}
		\partial_{(1,1)} f\Big( \mu, x \Big) = \partial_{(1,1)} f \Big( \mu, x \Big)^T, 
		\quad
		\partial_{(1,2)} f \Big( \mu, x_1, x_2 \Big) = \partial_{(1,2)} f\Big( \mu, x_2, x_1 \Big)^{T}
	\end{equation}
	where $A^T$ is the trace of the matrix $A$. 
	
	The goal of this section is to describe the equalities between the derivatives $\partial_a f$ for each $a\in A[0]$. Firstly, we include a result that is a more formal statement of \cite{CarmonaDelarue2017book2}*{Remark 4.16}
	
	\begin{proposition}
		\label{proposition:Schwartz-1}
		Let $f: \bR^e \times \cP_2(\bR^e) \to \bR^d$ and let $F: \bR^e \times L^2(\Omega, \bP; \bR^e) \to \bR^d$ be the canonical lift of $f$. Suppose that $F$ is twice Fr\'echet differentiable as a function on $\bR^e \times L^2(\Omega, \bP; \bR^e)$ with continuous second derivative. Then $\forall y \in \supp(\mu)$
		\begin{equation}
			\label{eq:proposition:Schwartz-1}
			\partial_{(0,1)} f\Big( x, \mu, y \Big) = \partial_{(1,0)} f \Big(x, \mu, y \Big)^T 
		\end{equation}
	\end{proposition}
	
	An important distinction to make here is that spatial and Lions derivatives commute when the spatial derivative is not in terms of the free variable generated by the application of that Lions derivative. In particular, Equation \eqref{eq:proposition:Schwartz-1} contrasts with \eqref{eq:CD-Schwarz-1}. 
	\begin{proof}
		Let $(x, X) \in \bR^e \times L^2(\Omega, \bP; \bR^e)$ be a point around which the function $F$ is twice Fr\'echet differentiable and let $(y, h) \in \bR^e \times L^2(\Omega, \bP; \bR^e)$. To avoid confusion, for the purposes of this proof we denote $\fD$ as the Fr\'echet derivative of the function $F$ that maps $(x, X) \mapsto F(x, X)$ and $D$ as the Fr\'echet derivative of the function $F(x, \cdot)$ that maps $X \mapsto F(x, X)$. 
		
		We note that 
		\begin{align}
			\label{eq:CommutativeDerivLR}
			\Big[ F(x+y, X+h)& - F(x, X+h) \Big] - \Big[ F(x+y, X) - F(x, X) \Big] 
			\\
			\label{eq:CommutativeDerivRL}
			=&\Big[ F(x+y, X+h) - F(x+y, X) \Big] - \Big[ F(x, X+h) - F(x, X) \Big]. 
		\end{align}
		We denote $\Pi_\theta^{X, X+h} = \bP \circ \big( X(\cdot) + \theta h(\cdot) \big)^{-1}$. 
		Using the Mean Value Theorem, we have that
		\begin{align*}
			\eqref{eq:CommutativeDerivLR} =&  \int_0^1 \bE\Big[ \partial_\mu f \Big( x+y, \Pi_{\theta}^{X, X+h}, X(\omega) + \theta h(\omega) \Big) \cdot h(\omega) \Big]  d\theta
			\\
			& - \int_0^1 \bE\Big[ \partial_\mu f\Big( x, \Pi_{\theta}^{X, X+h}, X(\omega)+\theta h(\omega) \Big) \cdot h(\omega) \Big]  d\theta 
			\\
			=& \int_0^1 \int_0^1 \bE \Big[ \nabla_x \partial_\mu f\Big( x+ \theta' y, \Pi_{\theta}^{X, X+h}, X(\omega)+\theta h(\omega) \Big) \cdot h(\omega) \otimes y \Big] d\theta d\theta'
		\end{align*}
		Similarly, 
		$$
		\eqref{eq:CommutativeDerivRL} = \int_0^1 \int_0^1 \bE \Big[ \partial_\mu \nabla_x  f\Big( x+ \theta y, \Pi_{\theta'}^{X, X+h}, X(\omega)+\theta' h(\omega) \Big) \cdot y \otimes h(\omega) \Big] d\theta d\theta'. 
		$$
		Choose $h(\omega) = \varepsilon \cdot \xi\big( X(\omega) \big)$ where $\xi:\bR^e \to \bR^e$ is a bounded measurable function and $y = \varepsilon y'$. By dividing by $\varepsilon$ and taking a limit as $\varepsilon \to 0$ plus the continuity of $\fD^2 F$, we conclude that
		\begin{align*}
			\bE \Big[& \partial_\mu \nabla_x f\Big( x, \bP\circ X^{-1}, X(\omega) \Big) \cdot \xi\big( X(\omega)\big) \otimes y' \Big] = \bE \Big[ \nabla_x \partial_\mu f\Big( x, \bP\circ X^{-1}, X(\omega) \Big) \cdot y' \otimes \xi\big( X(\omega) \big) \Big]
			\\
			&= \bE \Big[ \nabla_x \partial_\mu f\Big( x, \bP\circ X^{-1}, X(\omega) \Big)^T \cdot \xi\big( X(\omega) \big) \otimes y' \Big]. 
		\end{align*}
		This is true for any choice of $\xi$ bounded Borel measurable function on the support of the random variable $X$, so that Equation \eqref{eq:proposition:Schwartz-1} holds. 
	\end{proof}
	
	Given $a\in A_n$ and a permutation $\sigma \in \Shuf(n)$, we denote
	$$
	\sigma[a] = \big( \sigma[a]_i \big)_{i=1, ..., n} = \big( a_{\sigma(i)} \big)_{i=1, ..., n}. 
	$$
	
	\begin{theorem}[Schwarz Theorem for Lions Derivatives]
		\label{thm:Schwarz-Lions}
		Let $\alpha, \beta>0$ and let $\gamma>\alpha\wedge \beta$. Let $f: \bR^e \times \cP_2(\bR^e) \to \bR^d$ such that $f\in C_b\big[ A^{\gamma, \alpha, \beta}[0] \big]\big( \bR^e\times \cP_2(\bR^e); \bR^d \big)$. Let $\Shuf(|a|)$ be the set of permutations on the set $\{1, ..., |a| \}$. 
		
		Then $\forall a\in A^{\gamma, \alpha, \beta}[0]$ and $\forall \sigma\in \Shuf(|a|)$, $v_1, ..., v_{|a|} \in \bR^e$, $\mu\in \cP_2(\bR^e)$, $x_0, x_1, ..., x_{m[a]} \in \bR^e$, 
		\begin{equation}
			\label{eq:proposition:Schwartz-2}
			\partial_{a} f\Big( x_0, \mu, \boldsymbol{x}_{m\{a\}} \Big) \cdot \bigotimes_{i=1}^n v_i 
			=
			\partial_{\llbracket \sigma[a]\rrbracket_0} f\Big( x_0, \mu, \boldsymbol{x}_{\llbracket \sigma(a)\rrbracket_0 \circ (a)} \Big) \cdot \bigotimes_{i=1}^n v_{\sigma(i)} . 
		\end{equation}
		where
		\begin{equation*}
			\boldsymbol{x}_{m\{a\}} = \Big( x_1, ..., x_{m[a]} \Big)
			\quad \mbox{and}\quad 
			\boldsymbol{x}_{\llbracket \sigma(a)\rrbracket_0 \circ (a)} = \Big( x_{\sigma(a)_{a^{-1}[1]}}, ..., x_{\sigma(a)_{a^{-1}[m[a]]}} \Big). 
		\end{equation*}
	\end{theorem}
	
	\begin{proof}
		Firstly, we restrict ourselves to the case when $a\in A$ and $f\in C_b^{(n)}\big( \cP_2(\bR^e); \bR^d\big)$ for $n = \big\lfloor\tfrac{\gamma}{\beta} \big\rfloor$. We start with the case $|a|=2$, which can be found in \cite{CarmonaDelarue2017book1}*{Corollary 5.89}. By the symmetry of the Fr\'echet derivative, we have that
		\begin{align*}
			&\bE^1\Big[ \partial_{(1,1)} f\Big( \mu, X(\omega_1) \Big) \cdot h_1(\omega_1) \otimes h_2(\omega_1) \Big] + \bE^{1,2}\Big[ \partial_{(1,2)} f\Big( \mu, X(\omega_1), X(\omega_2) \Big) \cdot h_1(\omega_1) \otimes h_2(\omega_2) \Big]
			\\
			&= \bE^1\Big[ \partial_{(1,1)} f\Big( \mu, X(\omega_1) \Big) \cdot h_2(\omega_1) \otimes h_1(\omega_1) \Big] + \bE^{1,2}\Big[ \partial_{(1,2)} f\Big( \mu, X(\omega_1), X(\omega_2) \Big) \cdot h_2(\omega_1) \otimes h_1(\omega_2) \Big]. 
		\end{align*}
		We choose $h_i(\omega) = \varepsilon(\omega) \xi_i \big( X(\omega) \big)$ where $\xi$ is a bounded Borel measurable function and $\varepsilon: \Omega \to \{-1, 1\}$ be uniformly distributed and independent of $X$. Then
		$$
		\bE^1\Big[ \partial_{(1,1)} f \Big(\mu, X(\omega_1)\Big) \cdot \xi_1\big( X(\omega_1) \big) \otimes \xi_2\big( X(\omega_1) \big) \Big] 
		=
		\bE^1\Big[ \partial_{(1,1)} f \Big(\mu, X(\omega_1)\Big)^T \cdot \xi_2\big( X(\omega_1)\big) \otimes \xi_1\big( X(\omega_1) \big) \Big] . 
		$$
		By continuity of the mapping $x\mapsto \partial_{(1,1)} f(\mu, x)$, we conclude that $\forall v_1, v_2\in \bR^e$ and $\forall x\in \supp(\mu)$, 
		\begin{equation}
			\label{eq:proposition:Schwartz-2.2}
			\partial_{(1,1)}  f\Big( \mu, x\Big) \cdot v_1 \otimes v_2 
			= 
			\partial_{(1,1)}  f\Big( \mu, x\Big) \cdot v_2 \otimes v_1. 
		\end{equation}
		Subtracting this from above, we get that
		\begin{align*}
			\bE^{1,2}\Big[& \partial_{(1,2)} f \Big(\mu, X(\omega_1), X(\omega_2) \Big) \cdot \xi_1\big( X(\omega_1) \big) \otimes \xi_2\big( X(\omega_2) \big) \Big] 
			\\
			=&
			\bE^{1,2}\Big[ \partial_{(1,2)} f \Big(\mu, X(\omega_1), X(\omega_2) \Big) \cdot \xi_2\big( X(\omega_1) \big) \otimes \xi_1\big( X(\omega_2) \big) \Big]. 
		\end{align*}
		By the continuity of of the mapping $(x_1, x_2) \mapsto \partial_{(1,2)} f(\mu, x_1, x_2)$, we conclude that $\forall v_1, v_2\in \bR^e$ and $\forall x\in \supp(\mu)$, 
		\begin{equation}
			\label{eq:proposition:Schwartz-2.1}
			\partial_{(1,2)}  f\Big( \mu, x_1, x_2 \Big) \cdot v_1 \otimes v_2 
			= 
			\partial_{(1,2)}  f\Big( \mu, x_2, x_1 \Big) \cdot v_2 \otimes v_1. 
		\end{equation}
		
		Now we consider some partition sequence $a\in A_n$ and permutation $\sigma \in \Shuf(n)$. This can be written as the composition of a collection of adjacent cycles of length 2. Therefore, we simplify the problem down to considering the symmetry relations of three possible types of pairs of derivatives
		$$
		\partial_\mu \partial_\mu, \quad \nabla_{x} \partial_\mu \quad \mbox{and} \quad \nabla_{x} \nabla_{x}. 
		$$
		The first case has been proved in Equation \eqref{eq:proposition:Schwartz-2.1}. For the second case, there are two scenarios: either the spatial derivative is for the free variable generated by the application of the Lions derivative, in which case we use Equation \eqref{eq:proposition:Schwartz-2.2}. Alternatively, the spatial derivative is in a variable that is separate from the free variable generated by the Lions derivative, in which case we use Equation \eqref{eq:proposition:Schwartz-1}. The final scenario, the swapping of two spatial derivatives, is well documented and needs no further remarks. 
		
		By applying these identities iteratively until we obtain the permutation $\sigma$, we remark that permuting two spatial derivatives and permuting a spatial and Lions derivative permutes the sequence $a$ and the directional component $\bigotimes v$. On the other hand, permuting two Lions derivatives leads to a permutation in the variables $(x_1, ..., x_{m[a]})$ which is captured by the permutation $\llbracket \sigma(a) \rrbracket_0\circ (a)$. This results in Equation \eqref{eq:proposition:Schwartz-2}. 
		
		The extension to the case $a\in A^{\gamma, \alpha, \beta}[0]$ and $f\in C_b\big[ A^{\gamma, \alpha, \beta}[0] \big] \big( \bR^e \times \cP_2(\bR^e); \bR^d\big)$ is straightforward using Leibniz identities. 
	\end{proof}
	
	\begin{remark}
		As a concluding remark, we leave it as an exercise to the reader to find a more compact statement of Equation \eqref{theorem:LionsTaylor2} that additionally captures the equalities proved in \ref{thm:Schwarz-Lions}. 
	\end{remark}

	\subsection{Proof of Theorem \ref{theorem:LionsTaylor2}}
	
	\begin{proof}
		Firstly, let's consider the simpler case where $\alpha=\beta$ so that
		\begin{equation*}
			A^{\gamma, \alpha, \alpha}[0] = \bigcup_{i=0}^{\big\lfloor \tfrac{\gamma}{\alpha} \big\rfloor } A_i[0]
		\end{equation*}
		Choosing $\alpha <\gamma<2\alpha$, we observe that both the functions $\partial_{\mu} f$ and $\nabla_{x_0} f$ exist and the initial case
		\begin{align*}
			f(y_0, \nu)& - f(x_0, \mu) = \int_0^1 \frac{d}{d\xi} f\Big( x_0^\xi, \Pi_{\xi}^{\mu, \nu}\Big) d\xi
			\\
			=& \int_0^1 \nabla_{x_0} f\Big( x_0^\xi, \Pi_{\xi}^{\mu, \nu}\Big) \cdot (y_0-x_0) d\xi
			\\
			&+\int_0^1 \int_{(\bR^e)^{\oplus 2}} \partial_\mu f\Big( x_0^\xi, \Pi_{\xi}^{\mu, \nu}, x_1^\xi \Big) \cdot (y_1 - x_1) \cdot d\Pi^{\mu, \nu}(x_1, y_1) d\xi
			\\
			=& \sum_{a\in A_{1}[0]} \rD^a f(x_0, \mu)\big[ y_0-x_0, \Pi^{\mu, \nu}\big]
			\\
			&+\sum_{a\in A_{\ast}[0]}^{\gamma, \alpha, \alpha} \underbrace{\int_{(\bR^e)^{\oplus 2}} ... \int_{(\bR^e)^{\oplus 2}}}_{\times m[a]} f_{\ast}^a\big[ (x_0, y_0), \Pi^{\mu, \nu} \big] \cdot \bigotimes_{p=1}^{|a|} ( y_{a_p} - x_{a_p} ) \cdot d\big( \Pi^{\mu, \nu} \big)^{\times m[a]}\Big( (\boldsymbol{x}, \boldsymbol{y})_{m[a]} \Big)
		\end{align*}
		which is just Equation \eqref{eq:FullTaylorExpansion}. 
		
		Now suppose that Equation \eqref{eq:FullTaylorExpansion} holds for some choice of $\gamma$ and let $\varepsilon>0$ satisfy that $\big\lfloor \tfrac{\gamma + \varepsilon}{\alpha} \big\rfloor = \big\lfloor \tfrac{\gamma}{\alpha} \big\rfloor +1$. In particular, this means that
		\begin{equation*}
			A^{\gamma, \alpha, \alpha}[0] = \bigcup_{i=0}^{\big\lfloor \tfrac{\gamma}{\alpha} \big\rfloor} A_i[0], 
			\quad
			A^{\gamma+\varepsilon, \alpha, \alpha}[0] = \bigcup_{i=0}^{\big\lfloor \tfrac{\gamma}{\alpha} \big\rfloor +1} A_i[0]. 
		\end{equation*}
		
		Suppose that $f\in C_b\big[ A^{\gamma+\varepsilon, \alpha, \alpha} [0] \big]\big(\bR^e \times \cP_2(\bR^e); \bR^d\big)$, so that $f\in C_b\big[ A^{\gamma, \alpha, \alpha}[0] \big]\big(\bR^e \times \cP_2(\bR^e); \bR^d\big)$ too. Thus we have that Equation \eqref{eq:FullTaylorExpansion} holds for $f$. Consider each term in $R_{\gamma, \alpha, \beta}^{(x_0, y_0), \Pi^{\mu, \nu}}(f)$: for $a\in A_{\ast}^{\gamma, \alpha, \alpha}[0]$, 
		\begin{align*}
			&\underbrace{ \int_{(\bR^e)^{\oplus 2}} ... \int_{(\bR^e)^{\oplus 2}} }_{\times m[a]} f_{\ast}^{a}[ (x_0, y_0), \Pi^{\mu, \nu}] 
			\cdot 
			\bigotimes_{p=1}^{|a|} ( y_{a_p} - x_{a_p}) \cdot d\big( \Pi^{\mu, \nu}\big)^{\times m[a]} \Big( (\boldsymbol{x}, \boldsymbol{y})_{m\{a\} }\Big)
			\\
			&=\int_0^1 \int_0^\xi \underbrace{ \int_{(\bR^e)^{\oplus 2}} ... \int_{(\bR^e)^{\oplus 2}} }_{\times m[a]} \partial_{(a \cdot 0)} f\Big( x_0^{\theta}, \Pi_{\theta}^{\mu, \nu}, \boldsymbol{x}_{m\{a\}}^\theta \Big) \cdot \bigotimes_{p=1}^{|a|}(y_{a_p} - x_{a_p})
			\\
			&\qquad \otimes (y_0 - x_0) \cdot d \big(\Pi^{\mu, \nu}\big)\Big( (\boldsymbol{x}, \boldsymbol{y})_{m\{a\} }\Big) d\theta (1-\xi)^{|a|-1} d\xi 
			\\
			&+\int_0^1 \int_0^\xi \underbrace{ \int_{(\bR^e)^{\oplus 2}} ... \int_{(\bR^e)^{\oplus 2}} }_{\times m[a]} \sum_{j=1}^{m[a]} \partial_{(a\cdot j)} f\Big( x_0^\theta, \Pi_{\theta}^{\mu, \nu}, \boldsymbol{x}_{m\{a\}}^\theta \Big) \cdot \bigotimes_{p=1}^{|a|}(y_{a_p} - x_{a_p}) 
			\\
			&\qquad \otimes (y_j - x_j) \cdot d \big(\Pi^{\mu, \nu}\big)\Big( (\boldsymbol{x}, \boldsymbol{y})_{m\{a\} }\Big) d\theta (1-\xi)^{|a|-1} d\xi 
			\\
			&+\int_0^1 \int_0^\xi \underbrace{ \int_{(\bR^e)^{\oplus 2}} ... \int_{(\bR^e)^{\oplus 2}} }_{\times m[a]+1} \partial_\mu \partial_{a} f\Big( x_0^\theta, \Pi_{\theta}^{\mu, \nu}, \boldsymbol{x}_{m\{a\}}^\theta, x_{m[a]+1}^\theta \Big) \cdot \bigotimes_{p=1}^{|a|}(y_{a_p} - x_{a_p}) 
			\\
			&\qquad \otimes (y_{m[a]+1} - x_{m[a]+1}) \cdot d \big(\Pi^{\mu, \nu}\big)\Big( (\boldsymbol{x}, \boldsymbol{y})_{m\{a\} }\Big) \cdot d \big(\Pi^{\mu, \nu}\big)\Big( x_{m[a]+1}, y_{m[a]+1}\Big) d\theta (1-\xi)^{|a|-1} d\xi 
			\\
			&= \sum_{j=0}^{m[a]+1} \frac{1}{|(a\cdot j)|!} \cdot \rD^{(a\cdot j)}f(x_0, \mu) \big[y_0-x_0, \Pi^{\mu, \nu} \big]
			\\
			&\quad +
			\sum_{j=0}^{m[a]+1} \underbrace{ \int_{(\bR^e)^{\oplus 2}} ... \int_{(\bR^e)^{\oplus 2}} }_{\times m[(a\cdot j)]} f_{\ast}^{(a\cdot j)}[(x_0, y_0), \Pi^{\mu, \nu}] 
			\cdot \bigotimes_{p=1}^{|(a\cdot j)|} (y_{a_p} - x_{a_p}) \cdot d\big(\Pi^{\mu, \nu}\big)^{\times m[(a\cdot j)]} \Big( (\boldsymbol{x}, \boldsymbol{y})_{m[(a\cdot j)]}\Big). 
			\\
			&\qquad 
		\end{align*}
		By summing all these terms together, we obtain the inductive hypothesis. 
		
		Without loss of generality, we assume that $\alpha<\beta$. This is the most physical scenario and the proof proceeds equivalently for when $\alpha>\beta$. We fix $\alpha$ and $\beta$ and proceed via induction on $\gamma$. 
		
		Firstly, note that by assumption $\gamma>\alpha$ so that $(0)\in A^{\gamma, \alpha, \beta}[0]$. Let us suppose for the moment that $\beta, 2\alpha>\gamma$ so that
		\begin{equation*}
			A^{\gamma, \alpha, \beta}[0] = \Big\{ (0) \Big\}, 
			\quad 
			A_{+}^{\gamma, \alpha, \beta}[0] = \Big\{ \emptyset \Big\}, 
			\quad
			A_{\ast}^{\gamma, \alpha, \beta}[0] = \Big\{ (0) \Big\}, 
			\quad 
			A_{\times}^{\gamma, \alpha, \beta}[0] = \emptyset. 
		\end{equation*}
		Then
		\begin{align*}
			f(y_0, \nu) - f(x_0, \mu) =&  \Big( f(y_0, \mu) - f(x_0, \mu) \Big) + \Big( f(y_0, \nu) - f(y_0, \mu) \Big)
			\\
			=& \sum_{a \in A[0]}^{\gamma, \alpha, \beta} \frac{1}{|a|!} \cdot \rD^af\Big( x_0, \mu \Big)\Big[ y_0 - x_0 \Big] + \Big( f(y_0, \nu) - f(y_0, \mu) \Big)
			\\
			&+ \int_0^1 \Big( \nabla_{x_0} f( x_0^\xi, \mu) - \nabla_{x_0} f( x_0, \mu) \Big) d\xi
		\end{align*}
		which accounts for the trivial case. 
		
		Next, let us consider the case when $\alpha = \beta-\varepsilon$ for some $\varepsilon \in (0, \beta)$. For $\varepsilon$ small enough so that
		\begin{equation*}
			\big\lfloor \tfrac{\gamma}{\alpha} \big\rfloor = \big\lfloor \tfrac{\gamma}{\beta} \big\rfloor
		\end{equation*}
		we remark that $A_{\ast}^{\gamma, \alpha, \beta}[0] = A_{\ast}^{\gamma, \beta, \beta}[0]$ and conclude trivially. Now suppose that $\varepsilon$ is chosen so that $\alpha \cdot(n+1) \leq \gamma,$ and $\alpha \cdot n + \beta > \gamma$. Then 
		\begin{equation*}
			A^{\gamma, \alpha, \beta}[0] \backslash A^{\gamma, \beta, \beta}[0] = \big\{ \underbrace{(0, ..., 0)}_{n+1} \big\}. 
		\end{equation*}
		By calculating the sets described in Equation \eqref{eq:definition:Special-A}, we get
		\begin{equation*}
			A_{+}^{\gamma, \alpha, \beta}[0] = \Big\{ a' \Big\}, 
			\quad
			A_{\times}^{\gamma, \alpha, \beta}[0] = \Big\{ (a', 0) \Big\}. 
		\end{equation*}
		Briefly denoting $a'=(0, ..., 0)$ such that $|a'| = n$, we note that
		\begin{align*}
			f_{\ast}^{a'}[ &(x_0, y_0), \Pi^{\mu, \nu}] \cdot ( y_{0} - x_{0})^{\otimes n}
			\\
			&= 
			f_{\times}^{a'}[ (x_0, y_0), \Pi^{\mu, \nu}] \cdot ( y_{0} - x_{0})^{\otimes n}
			+
			f_{+}^{a'}[ (x_0, y_0), \Pi^{\mu, \nu}] \cdot ( y_{0} - x_{0})^{\otimes n} 
			\\
			&=\frac{1}{n+1} \cdot \rD^{(a',0)} f(x_0, \mu)\big[ y_0-x_0, \Pi^{\mu, \nu}\big]
			+
			f_{\times}^{(a', 0)}[ (x_0, y_0), \Pi^{\mu, \nu}] \cdot ( y_{0} - x_{0})^{\otimes n+1}
			\\
			&\quad +
			f_{+}^{a'}[ (x_0, y_0), \Pi^{\mu, \nu}] \cdot ( y_{0} - x_{0})^{\otimes n} . 
		\end{align*}
		This agrees with the inductive hypothesis. 
		
		Now suppose that Equation \eqref{eq:FullTaylorExpansion} holds for some triple $(\gamma, \alpha, \beta)$ and we want to verify that Equation \eqref{eq:FullTaylorExpansion} holds for $(\gamma, \alpha-\varepsilon, \beta)$. Note that for $\varepsilon$ chosen small enough, $A^{\gamma, \alpha, \beta} = A^{\gamma, \alpha-\varepsilon, \beta}[0]$ so that the inductive hypothesis will be true for $\varepsilon$ chosen small enough. 
		
		We choose $\varepsilon$ small enough so that $A^{\gamma, \alpha-\varepsilon, \beta}[0] \backslash A^{\gamma, \alpha, \beta}[0] \neq \emptyset$ and that 
		\begin{equation*}
			\Big| \Big\{ \scG_{\alpha - \varepsilon, \beta}[a]: a\in A^{\gamma, \alpha-\varepsilon, \beta}[0] \backslash A^{\gamma, \alpha, \beta}[0] \Big\} \Big| = 1. 
		\end{equation*}
		Let us denote $\alpha' = \alpha-\varepsilon$. Let $f\in C_b\big[ A^{\gamma, \alpha', \beta}[0] \big]\big(\bR^e \times \cP_2(\bR^e); \bR^d\big)$, so that we also have that $f\in C_b\big[ A^{\gamma, \alpha, \beta}[0] \big]\big(\bR^e \times \cP_2(\bR^e); \bR^d\big)$ and Equation \eqref{eq:FullTaylorExpansion} holds for $f$. 
		
		Let $a\in A^{\gamma, \alpha', \beta}[0] \backslash A^{\gamma, \alpha, \beta}[0]$. We want to distinguish the two cases where $a_{|a|}=0$ or $a_{|a|}>0$. When we have the former, we have that $\scG_{\alpha, \beta}[a^-] \in (\gamma-\alpha, \gamma]$ and when we have the latter $\scG_{\alpha, \beta}[a^-] \in (\gamma-\beta, \gamma-\alpha]$.  We denote the sequence $a^-=(a_i)_{i=1, ..., |a|-1}$. 
		
		Let us suppose first that that $a_{|a|} = 0$. Then either $a^-\in A_{\ast}^{\gamma, \alpha, \beta}[0]$ or $a^-\in A_{\times}^{\gamma, \alpha, \beta}[0]$. By the minimality of $\varepsilon$, for each $a\in A^{\gamma, \alpha', \beta}[0] \backslash A^{\gamma, \alpha, \beta}[0]$ we have that $\scG_{\alpha', \beta}[a]\in (\gamma-\alpha, \gamma]$. 
		
		Suppose first that $a^-\in A_{\times}^{\gamma, \alpha, \beta}[0]$. By considering the remainder term and using that $\nabla_{x_0} \partial_{a^-} f = \partial_a f$, 
		\begin{align}
			\nonumber
			&\underbrace{ \int_{(\bR^e)^{\oplus 2}} ... \int_{(\bR^e)^{\oplus 2}} }_{\times m[a]} f_{\times}^{a^-}[ (x_0, y_0), \Pi^{\mu, \nu}] 
			\cdot 
			\bigotimes_{p=1}^{|a^-|} ( y_{a_p} - x_{a_p}) \cdot d\big( \Pi^{\mu, \nu}\big)^{\times m[a]} \Big( (\boldsymbol{x}, \boldsymbol{y})_{m\{a\} }\Big)
			\\
			\nonumber
			&= \frac{1}{|a|} \cdot \rD^a f(x_0, \mu)\big[ y_0- x_0, \Pi^{\mu, \nu} \big] 
			\\
			\label{eq:FullTaylorExpansion_ast_1.1}
			&\qquad + \underbrace{ \int_{(\bR^e)^{\oplus 2}} ... \int_{(\bR^e)^{\oplus 2}} }_{\times m[a]} f_{\times}^{a}[ (x_0, y_0), \Pi^{\mu, \nu}] 
			\cdot 
			\bigotimes_{p=1}^{|a|} ( y_{a_p} - x_{a_p}) \cdot d\big( \Pi^{\mu, \nu}\big)^{\times m[a]} \Big( (\boldsymbol{x}, \boldsymbol{y})_{m\{a\} }\Big)
		\end{align}
		In the same line of thought, we also note that $\scG_{\alpha', \beta}[a] \in (\gamma-\alpha', \gamma]$ and $\scG_{\alpha', \beta}[a^-]\leq \gamma-\alpha'$ so that $a^{-} \notin A_{+}^{\gamma, \alpha', \beta}[0]$ and $a\in A_{+}^{\gamma, \alpha', \beta}[0]$. Combining this with Equation \eqref{eq:FullTaylorExpansion_ast_1.1} provides the inductive hypothesis. 
		
		Secondly, suppose that $a^- \in A_{\ast}^{\gamma, \alpha, \beta}[0]$. As before, 
		\begin{align}
			\nonumber
			&\underbrace{ \int_{(\bR^e)^{\oplus 2}} ... \int_{(\bR^e)^{\oplus 2}} }_{\times m[a]} f_{\ast}^{a^-}[ (x_0, y_0), \Pi^{\mu, \nu}] 
			\cdot 
			\bigotimes_{p=1}^{|a^-|} ( y_{a_p} - x_{a_p}) \cdot d\big( \Pi^{\mu, \nu}\big)^{\times m[a]} \Big( (\boldsymbol{x}, \boldsymbol{y})_{m\{a\} }\Big)
			\\
			\nonumber
			&= \frac{1}{|a|} \cdot \rD^a f(x_0, \mu)\big[ y_0- x_0, \Pi^{\mu, \nu} \big] 
			\\
			\nonumber
			&\qquad + \underbrace{ \int_{(\bR^e)^{\oplus 2}} ... \int_{(\bR^e)^{\oplus 2}} }_{\times m[a]} f_{\ast}^{a}[ (x_0, y_0), \Pi^{\mu, \nu}] 
			\cdot 
			\bigotimes_{p=1}^{|a|} ( y_{a_p} - x_{a_p}) \cdot d\big( \Pi^{\mu, \nu}\big)^{\times m[a]} \Big( (\boldsymbol{x}, \boldsymbol{y})_{m\{a\}}\Big)
			\\
			\label{eq:FullTaylorExpansion_ast_1.2}
			&\qquad + \underbrace{ \int_{(\bR^e)^{\oplus 2}} ... \int_{(\bR^e)^{\oplus 2}} }_{\times m[a]} f_{+}^{a^-}[ (x_0, y_0), \Pi^{\mu, \nu}] 
			\cdot 
			\bigotimes_{p=1}^{|a^-|} ( y_{a_p} - x_{a_p}) \cdot d\big( \Pi^{\mu, \nu}\big)^{\times m[a]} \Big( (\boldsymbol{x}, \boldsymbol{y})_{m\{a\}}\Big)
		\end{align}
		However, when $a^- \in A^{\gamma, \alpha, \beta}[0]$ means that $\scG_{\alpha', \beta}[a^-]\in (\gamma-\beta, \gamma-\alpha']$ and there exists no $k\in 1, ..., |a|$ such that $(a_i)_{i=1, ..., k} \in A_{+}^{\gamma, \alpha, \beta}[0]$. Hence $a^- \in A_{+}^{\gamma, \alpha', \beta}[0]$ which implies that $a\in A_{\times}^{\gamma, \alpha', \beta}[0]$. Combining this with Equation \eqref{eq:FullTaylorExpansion_ast_1.2} provides the inductive hypothesis. 
		
		Now we consider the case where $a_{|a|} >0$. Then $\scG_{\alpha, \beta}[a^-] \in (\gamma-\beta, \gamma-\alpha]$ so that $a\in A_{+}^{\gamma, \alpha, \beta}[0]$ and $\exists \tilde{a} = (a^-, 0, ..., 0)\in A_{\times}^{\gamma, \alpha, \beta}[0]$. Denoting $a^-[k]=(a^-, \underbrace{0, ..., 0}_{\times k})$, we have that
		\begin{align}
			\nonumber
			\sum_{k=1}^{|\tilde{a}| - |a^-|}& \frac{1}{|a^-[k]|!} \cdot \rD^{a^-[k]} f\Big( x_0, \mu \Big)\Big[ y_0- x_0, \Pi^{\mu, \nu}\Big] 
			\\
			\nonumber
			&+ \underbrace{ \int_{(\bR^e)^{\oplus 2}} ... \int_{(\bR^e)^{\oplus 2}} }_{\times m[\tilde{a}]} f_{\times}^{\tilde{a}}[ (x_0, y_0), \Pi^{\mu, \nu}] 
			\cdot 
			\bigotimes_{p=1}^{|\tilde{a}|} ( y_{\tilde{a}_p} - x_{\tilde{a}_p}) \cdot d\big( \Pi^{\mu, \nu}\big)^{\times m[\tilde{a}]} \Big( (\boldsymbol{x}, \boldsymbol{y})_{m[\tilde{a}]}\Big)
			\\
			\nonumber
			&+ \underbrace{ \int_{(\bR^e)^{\oplus 2}} ... \int_{(\bR^e)^{\oplus 2}} }_{\times m[a^-]} f_{+}^{a^-}[ (x_0, y_0), \Pi^{\mu, \nu}] 
			\cdot 
			\bigotimes_{p=1}^{|a^-|} ( y_{a_p} - x_{a_p}) \cdot d\big( \Pi^{\mu, \nu}\big)^{\times m[a]} \Big( (\boldsymbol{x}, \boldsymbol{y})_{m\{a\}}\Big)
			\\
			\nonumber
			=&
			\underbrace{ \int_{(\bR^e)^{\oplus 2}} ... \int_{(\bR^e)^{\oplus 2}} }_{\times m[a^-]} f_{\ast}^{a^-}[ (x_0, y_0), \Pi^{\mu, \nu}] 
			\cdot 
			\bigotimes_{p=1}^{|a^-|} ( y_{a_p} - x_{a_p}) \cdot d\big( \Pi^{\mu, \nu}\big)^{\times m[a]} \Big( (\boldsymbol{x}, \boldsymbol{y})_{m\{a\}}\Big)
			\\
			\nonumber
			=& 
			\frac{1}{|a|!} \cdot \rD^{a} f\Big(x_0, \mu\Big)\Big[ (x_0, y_0), \Pi^{\mu, \nu} \Big]
			+\sum_{k=1}^{|\tilde{a}| - |a^-|} \frac{1}{|a^-[k]|!} \cdot \rD^{a^-[k]} f\Big( x_0, \mu \Big)\Big[ y_0- x_0, \Pi^{\mu, \nu}\Big] 
			\\
			\nonumber
			&+\underbrace{ \int_{(\bR^e)^{\oplus 2}} ... \int_{(\bR^e)^{\oplus 2}} }_{\times m[a]} f_{\ast}^{a}[ (x_0, y_0), \Pi^{\mu, \nu}] 
			\cdot 
			\bigotimes_{p=1}^{|a|} ( y_{a_p} - x_{a_p}) \cdot d\big( \Pi^{\mu, \nu}\big)^{\times m[a]} \Big( (\boldsymbol{x}, \boldsymbol{y})_{m\{a\}}\Big)
			\\
			&
			\nonumber
			+\underbrace{ \int_{(\bR^e)^{\oplus 2}} ... \int_{(\bR^e)^{\oplus 2}} }_{\times m[a^-]} f_{+}^{(a^-, 0)}[ (x_0, y_0), \Pi^{\mu, \nu}] 
			\cdot 
			\bigotimes_{p=1}^{|(a^-, 0)|} ( y_{a_p} - x_{a_p}) \cdot d\big( \Pi^{\mu, \nu}\big)^{\times m[a^-]} \Big( (\boldsymbol{x}, \boldsymbol{y})_{m[a^-]}\Big)
			\\
			\label{eq:FullTaylorExpansion_ast_1.3}
			&+ \underbrace{ \int_{(\bR^e)^{\oplus 2}} ... \int_{(\bR^e)^{\oplus 2}} }_{\times m[\tilde{a}]} f_{\times}^{\tilde{a}}[ (x_0, y_0), \Pi^{\mu, \nu}] 
			\cdot 
			\bigotimes_{p=1}^{|\tilde{a}|} ( y_{\tilde{a}_p} - x_{\tilde{a}_p}) \cdot d\big( \Pi^{\mu, \nu}\big)^{\times m[\tilde{a}]} \Big( (\boldsymbol{x}, \boldsymbol{y})_{m[\tilde{a}]}\Big)
		\end{align}
		However, we also have that $a^{-} \notin A_{+}^{\gamma, \alpha', \beta}[0]$ which implies that $a\in A_{\ast}^{\gamma, \alpha', \beta}[0]$. On the other hand $\scG_{\alpha, \beta}[(a^-, 0)] \in (\gamma-\beta, \gamma - \alpha']$ which implies that $(a^-, 0) \in A_{+}^{\gamma, \alpha, \beta}[0]$. Hence, $\tilde{a}\in A_{\times}^{\gamma, \alpha', \beta}[0]$. Combining this with Equation \eqref{eq:FullTaylorExpansion_ast_1.3} provides the inductive hypothesis.
	\end{proof}
	

\begin{bibdiv}
        \begin{biblist}
        
        \bib{Ambrosio2008Gradient}{book}{
              author={Ambrosio, Luigi},
              author={Gigli, Nicola},
              author={Savar\'{e}, Giuseppe},
               title={Gradient flows in metric spaces and in the space of probability
          measures},
             edition={Second},
              series={Lectures in Mathematics ETH Z\"{u}rich},
           publisher={Birkh\"{a}user Verlag, Basel},
                date={2008},
                ISBN={978-3-7643-8721-1},
              review={\MR{2401600}},
        }
        
        \bib{2019arXiv180205882.2B}{article}{
              author={Bailleul, Isma\"{e}l},
              author={Catellier, R\'{e}mi},
              author={Delarue, Fran\c{c}ois},
               title={Solving mean field rough differential equations},
                date={2020},
             journal={Electron. J. Probab.},
              volume={25},
               pages={Paper No. 21, 51},
                 url={https://doi.org/10.1214/19-ejp409},
              review={\MR{4073682}},
        }
        
        \bib{buckdahn2017mean}{article}{
              author={Buckdahn, Rainer},
              author={Li, Juan},
              author={Peng, Shige},
              author={Rainer, Catherine},
               title={Mean-field stochastic differential equations and associated
          {PDE}s},
                date={2017},
                ISSN={0091-1798},
             journal={Ann. Probab.},
              volume={45},
              number={2},
               pages={824\ndash 878},
                 url={https://doi.org/10.1214/15-AOP1076},
              review={\MR{3630288}},
        }
        
        \bib{chassagneux2014classical}{article}{
              author={Chassagneux, Jean-Fran\c{c}ois},
              author={Crisan, Dan},
              author={Delarue, Fran\c{c}ois},
               title={A probabilistic approach to classical solutions of the master
          equation for large population equilibria},
                date={2022},
                ISSN={0065-9266},
             journal={Mem. Amer. Math. Soc.},
              volume={280},
              number={1379},
               pages={v+123},
                 url={https://doi.org/10.1090/memo/1379},
              review={\MR{4493576}},
        }
        
        \bib{CarmonaDelarue2017book1}{book}{
              author={Carmona, Ren\'{e}},
              author={Delarue, Fran\c{c}ois},
               title={Probabilistic theory of mean field games with applications. {I}},
              series={Probability Theory and Stochastic Modelling},
           publisher={Springer, Cham},
                date={2018},
              volume={83},
                ISBN={978-3-319-56437-1; 978-3-319-58920-6},
                note={Mean field FBSDEs, control, and games},
              review={\MR{3752669}},
        }
        
        \bib{CarmonaDelarue2017book2}{book}{
              author={Carmona, Ren\'{e}},
              author={Delarue, Fran\c{c}ois},
               title={Probabilistic theory of mean field games with applications.
          {II}},
              series={Probability Theory and Stochastic Modelling},
           publisher={Springer, Cham},
                date={2018},
              volume={84},
                ISBN={978-3-319-56435-7; 978-3-319-56436-4},
                note={Mean field games with common noise and master equations},
              review={\MR{3753660}},
        }
        
        \bib{Rodney2012Calculus}{book}{
              author={Coleman, Rodney},
               title={Calculus on normed vector spaces},
              series={Universitext},
           publisher={Springer, New York},
                date={2012},
                ISBN={978-1-4614-3893-9},
                 url={https://doi.org/10.1007/978-1-4614-3894-6},
              review={\MR{2954390}},
        }
        
        \bib{dos2022Ito}{article}{
              author={dos Reis, Gon\c{c}alo},
              author={Platonov, Vadim},
               title={{I}t\^o-{W}entzell-{L}ions {F}ormula for {M}easure {D}ependent
          {R}andom {F}ields under {F}ull and {C}onditional {M}easure {F}lows},
                date={2022},
             journal={Potential Anal},
                 url={https://doi.org/10.1007/s11118-022-10012-1},
        }
        
        \bib{2021Probabilistic}{article}{
              author={Delarue, Francois},
              author={Salkeld, William},
               title={Probabilistic rough paths {I} {L}ions trees and coupled {H}opf
          algebras},
                date={2021},
             journal={arXiv preprint},
              eprint={2106.09801v2},
        }
        
        \bib{salkeld2021Probabilistic2}{article}{
              author={Delarue, Francois},
              author={Salkeld, William},
               title={Probabilistic rough paths {II} lions-taylor expansions and random
          controlled rough paths},
                date={2022},
             journal={arXiv preprint},
              eprint={2203.01185v1},
        }
        
        \bib{salkeld2022ExamplePRP}{article}{
              author={Delarue, Francois},
              author={Salkeld, William},
               title={An example driven introduction to {P}robabilistic rough paths},
                date={2023},
             journal={arXiv preprint, accepted in {S}{\'e}minaire de Probabilit{\'e}s},
              eprint={2106.09801v3},
        }
        
        \bib{GangboDifferentiability2019}{article}{
              author={Gangbo, Wilfrid},
              author={Tudorascu, Adrian},
               title={On differentiability in the {W}asserstein space and
          well-posedness for {H}amilton-{J}acobi equations},
                date={2019},
                ISSN={0021-7824},
             journal={J. Math. Pures Appl. (9)},
              volume={125},
               pages={119\ndash 174},
                 url={https://doi.org/10.1016/j.matpur.2018.09.003},
              review={\MR{3944201}},
        }
        
        \bib{hairer2014theory}{article}{
              author={Hairer, M.},
               title={A theory of regularity structures},
                date={2014},
                ISSN={0020-9910},
             journal={Invent. Math.},
              volume={198},
              number={2},
               pages={269\ndash 504},
                 url={https://doi.org/10.1007/s00222-014-0505-4},
              review={\MR{3274562}},
        }
        
        \bib{Jordan1998variation}{article}{
              author={Jordan, Richard},
              author={Kinderlehrer, David},
              author={Otto, Felix},
               title={The variational formulation of the {F}okker-{P}lanck equation},
                date={1998},
                ISSN={0036-1410},
             journal={SIAM J. Math. Anal.},
              volume={29},
              number={1},
               pages={1\ndash 17},
                 url={https://doi.org/10.1137/S0036141096303359},
              review={\MR{1617171}},
        }
        
        \bib{Ren2019Bismut}{article}{
              author={Ren, Panpan},
              author={Wang, Feng-Yu},
               title={Bismut formula for {L}ions derivative of distribution dependent
          {SDE}s and applications},
                date={2019},
                ISSN={0022-0396},
             journal={J. Differential Equations},
              volume={267},
              number={8},
               pages={4745\ndash 4777},
                 url={https://doi.org/10.1016/j.jde.2019.05.016},
              review={\MR{3983053}},
        }
        
        \bib{salkeld2022LionsTrees}{article}{
              author={Salkeld, William},
               title={Coupled bialgebras and lions trees},
                date={2023},
             journal={arXiv preprint},
              eprint={2303.17576},
        }
        
        \bib{TseHigher2021}{article}{
              author={Tse, Alvin},
               title={Higher order regularity of nonlinear {F}okker-{P}lanck {PDE}s
          with respect to the measure component},
                date={2021},
                ISSN={0021-7824},
             journal={J. Math. Pures Appl. (9)},
              volume={150},
               pages={134\ndash 180},
                 url={https://doi.org/10.1016/j.matpur.2021.04.005},
              review={\MR{4248465}},
        }
        
        \end{biblist}
    \end{bibdiv}

	\bibliographystyle{plain}
	
\end{document}